\documentclass[11pt,letterpaper]{amsart}
\usepackage{xcolor}
\usepackage{microtype}
\usepackage{caption}
\usepackage{amssymb}
\usepackage{amsmath,amsthm,enumerate, geometry}
\usepackage{psfrag}
\usepackage{amscd}
\usepackage{xspace}
\usepackage[all]{xy} 
\usepackage{longtable}
\usepackage{xspace}
\usepackage{pgf}
\usepackage{tikz}
\usetikzlibrary{cd}
\usepackage{stmaryrd}
\usepackage[backref=page]{hyperref}
\hypersetup{
 colorlinks,
 linkcolor=blue}

\usepackage{comment}

\DeclareMathOperator{\ch}{ch}
\DeclareMathOperator{\td}{td}

\def\cf{\textit{cf.}\kern.3em}

\def\resp{\textit{resp.}\kern.3em}
\renewcommand{\k}{\kern2pt}

\numberwithin{equation}{section} 

\let\emptyset\varnothing

\DeclareMathOperator{\rk}{rk}

\DeclareMathOperator{\Aut}{Aut}

\DeclareMathOperator{\id}{id}
\DeclareMathOperator{\im}{Im}

\DeclareMathOperator{\codim}{cd}

\DeclareMathOperator{\ext}{Ext}
\DeclareMathOperator{\Obj}{Obj}
\DeclareMathOperator{\Mor}{Mor}
\DeclareMathOperator{\Int}{Int}
\DeclareMathOperator{\sInt}{SInt}
\DeclareMathOperator{\chains}{chains}
\DeclareMathOperator{\nex}{next}
\DeclareMathOperator{\cu}{CU}
\DeclareMathOperator{\fu}{FU}
\DeclareMathOperator{\leg}{leg}

\newtheorem{construction}[equation]{Construction}
\newtheorem{proposition}[equation]{Proposition}
\newtheorem{theorem}[equation]{Theorem}
\newtheorem*{theorem*}{Theorem}
\newtheorem{corollary}[equation]{Corollary}
\newtheorem{lemma}[equation]{Lemma}
\theoremstyle{definition}
\newtheorem{definition}[equation]{Definition}

\newtheorem{remark}[equation]{\textbf{Remark}}
\newtheorem{notation}[equation]{\textbf{Notation}}

\newtheorem{example}[equation]{\textbf{Example}}
\newtheorem{hypothesis}{Hypothesis}
\newtheorem{axiom}{Axiom}

\newcommand{\w}{\mathsf{w}}

\newcommand{\LL}{\mathbb L}

\newcommand{\catJ}{{\mathfrak{C}}_{g,n}(\phi)}
\newcommand{\catJtilde}{\widetilde{\mathfrak{C}}}
\newcommand{\catJtildeE}{\widetilde{\mathfrak{C}}_E}
\newcommand{\catJtildeY}{\widetilde{\mathfrak{C}}_Y}

\newcommand{\Mb}[1]{\overline{\mathcal{M}}_{#1}}
\newcommand{\Mmb}[2]{\overline{\mathcal{M}}_{#1,#2}}

\newcommand{\J}[2]{{\mathcal{J}}^{#1}_{#2}}

\newcommand{\Jb}[2]{\overline{\mathcal{J}}^{#1}_{#2}}
\newcommand{\Jmb}[3]
{\overline{\mathcal{J}}^{#1}_{#2,#3}}

\newcommand{\Cb}[1]{\overline{\mathcal{C}}_{#1}}
\newcommand{\Cmb}[2]{\overline{\mathcal{C}}_{#1,#2}}

\newcommand{\cat}{\mathfrak{C}}
\newcommand{\oMor}{\overline{\Mor}}
\newcommand{\lInt}{\overline{\Int}}
\newcommand{\rInt}{\widetilde{\Int}}
\newcommand{\ol}{\overline}
\newcommand{\jphitilde}{\widetilde{\mathcal{J}}^d_{g,n}(\phi^+, \phi^-)}
\newcommand{\jphitildeP}{\widetilde{\mathcal{J}}^d_{g,n+1}(\phi^+, \phi^-;P)}
\newcommand{\parts}{\mathcal{P}}
\newcommand{\F}{\mathcal{F}}
\newcommand{\col}{\colon}
\newcommand{\mf}{\mathbf{m}}
\newcommand{\nf}{\mathbf{n}}
\newcommand{\jtilde}
{\widetilde{\mathcal{J}}}
\newcommand{\Div}{\operatorname{Div}}

\newcommand{\bl}{\operatorname{Bl}}

\usepackage{color}
\definecolor{forestgreen}{rgb}{0.13, 0.55, 0.13}

\newcounter{note}

\begin{document}

\def\fcj{fine compactified Jacobian\xspace}
\def\fcuj{fine compactified universal Jacobian\xspace}

\def\fcjs{fine compactified Jacobians\xspace}

\def\fcujs{fine compactified universal Jacobians\xspace}

\def\gf{vine function\xspace}

\def\gfs{vine functions\xspace}

\def\gc{full forest\xspace}

\def\gcs{full forests\xspace}

\date{}
\title{Wall crossing of universal Brill-Noether classes}
\author{Alex Abreu}
\address{Universidade Federal Fluminense,
Instituto de Matemática}
\email{alexbra1@gmail.com}
\author{Nicola Pagani}
\address{Department of Mathematical Sciences, University of Liverpool}
\email{pagani@liverpool.ac.uk}

\begin{abstract}
We give an explicit graph formula, in terms of decorated boundary strata classes, for the wall--crossing of universal Brill-Noether classes $\w_d(\phi)$. 

More precisely, fix $n>0$ and $d<g$, and  two stability conditions $\phi^+$ and $\phi^-$ for degree~$d$ compactified universal (over $\Mmb gn$) Jacobians that lie on opposite sides of a stability hyperplane. Our main result is a formula for the difference between $\w_d(\phi^+)$ and the pullback of $\w_d(\phi^-)$ along the (rational) identity map $\mathsf{Id} \colon \Jmb dgn (\phi^+) \dashrightarrow \Jmb dgn (\phi^-)$. The calculation involves constructing a resolution of the identity map by means of a sequence of blowups. 
\end{abstract}
\maketitle
\setcounter{tocdepth}{2}
\tableofcontents

\section{Introduction}

The  Brill--Noether theory of line bundles on nonsingular algebraic curves is a classical pillar of XIX century algebraic geometry, which has been rediscovered and reused to prove important contemporary results. Broadly speaking, the theory is about studying the space of line bundles of a fixed degree having a  fixed  number of linearly independent global sections (see \cite{acgh} and references therein for a survey of the classical results).

For  fixed integers $g,n$  (we will assume for uniformity of notation, that $g \geq 2$ and $ n \geq 1$), and $d$ there exists a universal Jacobian  $\mathcal{J}_{g,n}^d \to \mathcal{M}_{g,n}$, a moduli space that parameterizes isomorphism classes of degree~$d$ line bundles over smooth, $n$-pointed curves of genus $g$. From now on we assume $d<g$ and  define the universal  Brill--Noether class $\mathsf{w}_d$ as the fundamental class in $\mathcal{J}^d_{g,n}$ of the locus $\mathsf{W}_d$ of line bundles that admit a nonzero global section. This locus has fiberwise codimension $g-d$ over $\mathcal{M}_{g,n}$ and it is empty for $d<0$. In this paper we study extensions of this class to different compactifications of the universal Jacobian. 

The moduli space $\mathcal{M}_{g,n}$ admits a natural, modular and well-studied compactification $\Mmb{g}{n}$ obtained by adding (Deligne--Mumford) \emph{stable} pointed curves. On the other hand, there are several natural compactifications of $\mathcal{J}_{g,n}^d$ over $\Mmb{g}{n}$. In the words of Oda--Seshadri \cite{oda79}, this should not be seen as a drawback of the theory, but rather a merit.

In \cite{kp3} Kass--Pagani constructed an affine space of stability conditions $V_{g,n}^d$ with an explicit hyperplane arrangement, with the property that every $\phi\in V_{g,n}^d$ produces a compactification $\Jmb{d}{g}{n}(\phi)$ of the universal Jacobian, with good properties (it is a nonsingular DM stack) when $\phi$ is not on a hyperplane. This space comes with a natural origin --- a canonical stability --- and so far most of the attention has been devoted to compactified Jacobians corresponding to this particular value (or to its perturbations when the latter belongs to some hyperplanes), see \cite{gruzak}, \cite{hmpps}.

In this paper we study how the Brill--Noether classes $\mathsf{w}_d$, suitably extended to classes $\mathsf{w}_d(\phi)$ on  $\Jmb{d}{g}{n}(\phi)$, vary with $\phi$. What we mean by this is the following: for different stability conditions $\phi_1, \phi_2$, the identity on the common open set $\mathcal{J}^d_{g,n}$ of line bundles on smooth curves defines a rational map \[\mathsf{Id} \colon \Jmb{d}{g}{n}(\phi_1) \dashrightarrow  \Jmb{d}{g}{n}(\phi_2),\] and we can then compute the difference $\mathsf{w}_d(\phi_1) - \mathsf{Id}^* \mathsf{w}_d(\phi_2)$. Recall that a rational map induces a pullback at the level of Chow groups, defined via the correspondence induced by the closure of its graph; that is, considering the diagram
\[
\begin{tikzcd}
     & \Gamma_{\mathrm{Id}} \subseteq  \Jmb{d}{g}{n}(\phi_1)\times   \Jmb{d}{g}{n}(\phi_2) \ar[dl, "p_1"] \ar[dr, "p_2"] & \\
      \Jmb{d}{g}{n}(\phi_1) \ar[rr, dashed, "\mathsf{Id}"]  &&  \Jmb{d}{g}{n}(\phi_2),
\end{tikzcd}
\]
 define $\mathsf{Id}^* \mathsf{w}_d(\phi_2)=p_{1 *} \Big( p_2^*(\mathsf{w}_d(\phi_2)) \cap \big[\overline{ \Gamma_{\mathrm{Id}}}\big] \Big)$, where  $\overline{ \Gamma_{\mathrm{Id}}}$ denotes the closure of the graph of the identity map inside the product.

 By ``compute'', we mean produce an explicit ``graph formula'', as in the case of tautological classes on the moduli space of curves $\Mmb{g}{n}$, where all such classes can be expressed as linear combinations of ``decorated boundary strata'' (see \cite{pandhacalculus}).

While an established theory of a tautological ring for  $\Jmb{d}{g}{n}(\phi)$ is not yet available (though significant steps forward have been recently made, see \cite{yin}, \cite{bhpss}, \cite{hmpps}), there are several natural classes on each compactified universal Jacobian, and ``decorated boundary strata classes'', supported on the boundary of $\Jmb{d}{g}{n}(\phi)$, may be defined in complete analogy to the case of $\Mmb{g}{n}$. In fact, an  underlying motivation for our work is to develop a categorical and wall--crossing framework for a theory of tautological classes over compactified universal Jacobians. The Brill--Noether classes are, in our view, the most natural geometric classes where to start testing these ideas. We will discuss in \ref{subsec: motiv} some more context and motivation for our calculation.

We now explain what we mean by ``a suitable extension'' for the class $\mathsf{w}_d$. One possible approach is to take the class of the Zariski closure of $\mathsf{W}_d$, but this is very hard to control, and it does not have good formal properties (for example, it does not commute with base change). Another approach is to consider sheaves in $\Jmb{d}{g}{n}(\phi)$ that admit a nonzero global section, but the locus of those sheaves is, in general, not of the expected dimension and not equidimensional. Our extension instead is by means of the Thom--Porteous' formula. By the universal property, there is a tautological (or Poincar\'e) sheaf $\mathcal{L}_{\text{tau}}(\phi)$ on the universal curve $\pi \colon \Cmb{g}{n}(\phi) \to \Jmb{d}{g}{n}(\phi)$. We define the extension as the degeneracy class
\begin{equation} \label{classes}
\mathsf{w}_d(\phi):= c_{g-d} (-R^{\bullet} \pi_\ast \mathcal{L}_{\text{tau}}(\phi)),
\end{equation}
as in \cite[Chapter~14]{fulton} (see Subsection~\ref{subsec: BN} for details). By the Thom--Porteous formula (see \emph{loc.cit.}), the  restriction of $\mathsf{w}_d(\phi)$ to $\mathcal{J}^d_{g,n}$ equals the  original Brill--Noether class $\mathsf{w}_d$ (after capping with the fundamental class). We compare \eqref{classes} with the class of the Zariski closure in Proposition~\ref{BNisclosure}. The class~\eqref{classes} is supported on the universal Brill--Noether locus, but in general the latter does not have the expected codimension, hence its fundamental class does not coincide with~\eqref{classes} (more details in  Proposition~\ref{BNisnotexpected}).

In this paper we assume that $\phi^+$ and $\phi^-$ are on opposite sides of a stability hyperplane (Definition~\ref{Def: opposite sides}), and we give an explicit graph formula for the difference \[\mathsf{w}_d(\phi^+) - \mathsf{Id}^* \mathsf{w}_d(\phi^-).\] In order to achieve this, we first produce a nonsingular resolution of the identity
\[
\begin{tikzcd}[column sep=scriptsize]
&\jphitilde \ar[dl, "p"] \ar[dr, "p_-"]& \\
\Jb{d}{g,n}(\phi^+)\ar[rr, dashed, "\mathsf{Id}"] & &\Jb{d}{g,n}(\phi^-)
\end{tikzcd}
\]
by an explicit sequence of blowups of $\Jb{d}{g,n}(\phi^+)$. We use this resolution to give, in Theorem~\ref{maintheorem}, an explicit and closed graph formula for the difference $p^*(\mathsf{w}_d(\phi^+))-p_-^*(\mathsf{w}_d(\phi^-))$ in the cohomology of $\jphitilde$. Finally, we calculate the pushforward of that formula via $p$ to write a formula (again a graph formula, explicit and closed) for the difference $\mathsf{w}_d(\phi^+) - \mathsf{Id}^* \mathsf{w}_d(\phi^-)$.

Our construction of $\jphitilde$ and our formulas are complicated by the fact that, for some of the hyperplanes, the locus where the identity is undefined fails to be irreducible. In those cases, the space $\jphitilde$ is constructed as an explicit sequence of blowups along centers that have transversal self-intersection, and a large part of our paper is devoted to this construction.

In this introduction we describe the particular case of our construction and formula when the indeterminacy locus is irreducible (this  occurs in many cases, and in some sense in most cases, as long as $n>1$).  Then the indeterminacy locus ${\mathcal{J}}_{\beta}'\subset \Jmb dgn (\phi^+)$ generically parameterizes curves with $2$ nonsingular components of genera, say, $g_X$ and $g_Y$, carrying markings $S$ and $S^c$, and joined at a certain number of nodes, say $t$, together with line bundles of some fixed bidegree, say, $(d-d_Y, d_Y)$ over the union of the two curves.   The locus ${\mathcal{J}}_{\beta}'$ can be parameterized by a ``resolved stratum'' \[f_{\beta} \colon {\mathcal{J}}_\beta \rightarrow {\mathcal{J}}_{\beta}' \hookrightarrow \Jb{d}{g,n}(\phi^+)\](which we simply call ``a stratum'' in the main body of the paper), where the $t$ nodes are parameterized: a general point of ${\mathcal{J}}_\beta$ is a triple of a $(|S|+t)$-pointed curve of relative genus $g_X$, a $(|S^c|+t)$-pointed curve of genus $g_Y$, and a line bundle of bidegree $(d-d_Y, d_Y)$.  The  conormal bundle to $f_{\beta}$ has rank $t$ and it splits as a direct sum of line bundles, whose first Chern classes we call $\Psi_1, \ldots, \Psi_t$ (see Remark~\ref{relate-psi} for more details on how these relate to the ``classical'' $\psi$-classes in  $\Mmb gn$). The base change to ${\mathcal{J}}_\beta $ of the universal family $\pi \colon \Cmb{g}{n} \to \Jmb{d}{g}{n}(\phi^+)$ consists of two irreducible components, say $\pi^X \colon X \to {\mathcal{J}}_\beta $ and $\pi^Y \colon Y \to {\mathcal{J}}_\beta $, of relative genera $g_X$ and $g_Y$ respectively, each carrying a tautological sheaf $L_X$ and $L_Y$ (obtained by pulling back $\mathcal{L}_{\text{tau}}(\phi^+)$).

In this particular case, our main result becomes:
\begin{theorem*}  (Corollary~\ref{Cor: wc-goodwalls} with $m=1$.) If $\phi^+$ and $\phi^-$ are on opposites sides of a stability hyperplane (Definition~\ref{Def: opposite sides}) and the indeterminacy locus of the identity morphism $\mathsf{Id} \colon \Jmb{d}{g}{n}(\phi^+) \dashrightarrow  \Jmb{d}{g}{n}(\phi^-)$ is irreducible, the difference $\mathsf{w}_d(\phi^+) - \mathsf{Id}^* \mathsf{w}_d(\phi^-)$ equals
\begin{multline*}
\sum_{\substack{s+j+\lambda=\\g-d-t}} {\small \binom{g_{Y}-d_{Y}-j-1}{g-d-j-s} }\frac{f_{{\beta} *}}{t!} \Bigg(  c_{s}(-R^{\bullet} \pi^X_{ \ast} L_X(-X \cap Y)) \cdot c_{j}(-R^{\bullet} \pi^Y_{\ast} L_Y) \cdot h_{\lambda}(\Psi_{1}, \ldots, \Psi_{t})\Bigg),
\end{multline*}
where $h_{\lambda}$ is the complete homogeneous polynomial of degree $\lambda$ in $t$ variables.
\end{theorem*} 
The special case $d=g-1$ of the above formula, when the Brill--Noether class is a divisor (the theta divisor), was discovered in \cite[Theorem~4.1]{kp2}. In that case the calculation was massively simplified by the fact that the classes have codimension $1$, and therefore, because  $\Jmb{d}{g}{n}(\phi^+)$ is nonsingular, no blowup is required.

We regard Theorem~\ref{maintheorem} as the main result in this paper, and the construction of $\jphitilde$ (Construction~\ref{con: blowup}) as the main technical tool.  The formula in Theorem~\ref{maintheorem} computes the pullback to the resolution $\jphitilde$ of the difference $\mathsf{w}_d(\phi^+) - \mathsf{Id}^* \mathsf{w}_d(\phi^-)$ in terms of decorated boundary strata classes. The formula for the difference in $\Jb{d}{g,n}(\phi^+)$ is obtained by pushing the latter forward along a blowdown morphism, which generates more complicated coefficients.

The starting point to construct the  resolution $\jphitilde$ is the observation that the tautological sheaf for $\phi^+$ is not $\phi^-$-stable, and the locus ${\mathcal{J}}_{\beta}'$  where it fails $\phi^-$-stability generically parameterizes curves with $2$ nonsingular irreducible components (throughout called a ``vine curve'') and line bundles of a fixed bidegree. When the locus ${\mathcal{J}}_{\beta}'$ is irreducible, the resolution is constructed by blowing up $\Jb{d}{g,n}(\phi^+)$ at ${\mathcal{J}}_{\beta}'$. The two components $X'$ and  $Y'$ of the pullback of the universal curve to the exceptional divisor $E$ are now divisors (``universal twistors'') in the blowup of the universal curve, and after suitably tensoring by one of them, the sheaf $\mathcal{L}_{\text{tau}}(\phi^+)$ becomes $\phi^-$-stable. The latter sheaf defines the other morphism $\jphitilde \to \Jb{d}{g,n}(\phi^-)$ by the universal property. The main technical difficulty is to suitably identify  a sequence of blowups at centers that have transversal self-intersection, which allows one to generalize the above reasoning to the case when the base locus ${\mathcal{J}}_{\beta}'$ is not irreducible.

Note that other resolutions of the identity map may also be constructed following existing literature (\cite{apresol}, \cite{holmes}, \cite{marcuswise}, \cite{hmpps}, \cite{holmescomp}), but those constructions yield \emph{singular} spaces.

In {\bf Section~2} we fix the background notation and recall some relevant results from the existing literature. In {\bf Section~3} we introduce the objects we work with, compactified universal Jacobians and Brill--Noether classes. In {\bf Section~4} we write axioms for categories of ``resolved'' strata of normal crossing stratifications, and then prove some general intersection theory results that are valid in this context. The main geometric ideas here are not entirely new, but we could not find a suitable reference that works in the generality that we need, and we believe that this axiomatic point of view will prove helpful in the current research landscape. In {\bf Section~5} we discuss the combinatorial aspects that arise from a wall--crossing situation where there are stability conditions $\phi^\pm$ on opposite sides of a given stability hyperplane (Definition~\ref{Def: opposite sides}).  Our paper is concerned with the case of rank~$1$ sheaves on nodal curves, and the combinatorics of Section~5 should be the special case of a theory for higher dimension and rank. The central definition is that, for each graph $G$ and divisor $D$ on $G$ and choice of stability conditions $\phi^\pm$ on opposite sides of a hyperplane, of a poset $\ext(G,D)$ of ``extremal'' subsets of the vertices of $G$. {\bf Section~6} gives the construction of the resolution $\jphitilde$. Finally, in {\bf Section~7} we are then ready to employ intersection theory techniques and calculate the wall--crossing difference. At the end of Section~7 we explain how the pullback of the wall--crossing term via an Abel--Jacobi section can be explicitly calculated in terms of decorated boundary strata classes in $\Mmb gn$ by employing the main result of \cite{prvZ}.

In the background of this work, we produce  two  results that we believe are of independent interest.  The first is Theorem~\ref{prop: univ-semistable-mod}, where we interpret the universal quasistable family over $\Jmb dgn (\phi)$ (also known in the literature as Caporaso's family from \cite{caporaso}, see also \cite{mmuv} and \cite{estevespacini}) as a fine compactified universal Jacobian $\Jmb dg{n+1} (\phi')$ over marked curves with one extra  point.

Secondly, as part of Proposition~\ref{BNisclosure}, we describe the collection of stability conditions for $d~<~0$ such that $\mathsf{w}_d(\phi)=0$. One can choose a suitable Abel--Jacobi section $\sigma \colon \Mmb gn \to \Jmb dgn (\phi)$ and obtain a zero class $\sigma^* \mathsf{w}_d(\phi)$. A different formula for the latter as a linear combination of standard tautological classes was given in \cite{prvZ} by means of the Grothendieck-Riemann-Roch formula. This gives tautological relations in $\Mmb gn$ (see Remark~\ref{Rem: relinmbgn} for the details).  Note that these relations are in degree larger than $g$ (the degree is $g-d$ for negative $d$), the same range of Pixton's double ramification relations (proven by Clader-Janda in \cite{claderjanda}).

\subsection{Motivation and related work} \label{subsec: motiv}
An important motivation underlying our calculation is its relation with the (possibly twisted) double ramification cycle. The double ramificatin cycle associated to a fixed vector of integers ${\bf d}=(d_1, \ldots, d_n)$ such that $d=0=\sum d_i=0$ is the locus in $\mathcal{M}_{g,n}$ of pointed curves $(C, p_1, \ldots, p_n)$ such that $\mathcal{O}_C(\sum d_i p_i)$ is isomorphic to the trivial line bundle $\mathcal{O}_C$. The morphism $\sigma_{\bf d} \colon \mathcal{M}_{g,n} \to \mathcal{J}_{g,n}$ defined by mapping
\[
(C, p_1, \ldots, p_n) \mapsto (C, p_1, \ldots, p_n; \mathcal{O}_C(d_1p_1 + \ldots + d_n p_n))
\]
is called the \emph{Abel--Jacobi} section.

The first observation made by Pandharipande-Zvonkine (unpublished) was that one can extend the double ramification cycle to $\overline{\mathcal{M}}_{g,n}$ as the class $c_g(-R \pi_* \mathcal{O}(\sum d_i p_i))$, but that extension rarely coincides with the 'interesting' extension (the one with applications in other areas such as symplectic geometry and integrable systems), originally constructed via a virtual fundamental class on a certain moduli space of stable maps. However, the main result of \cite{HKP} shows that such 'interesting' extension can be recovered as the pullback of $\mathsf{w}_0(\phi)$ via the Abel--Jacobi \emph{rational} section $\sigma_{\bf d} \colon \overline{\mathcal{M}}_{g,n} \dashrightarrow \overline{\mathcal{J}}_{g,n}(\phi)$, for all nondegenerate $\phi$'s such that the structure sheaf is $\phi$-stable. 

For a review of the double ramification cycle theory and related literature, we address the reader to \cite{bhpss} and \cite[Section~1.1]{hmpps}. We refer to \cite[Section~4.3]{prvZ} for a more detailed discussion on how the double ramification cycle relates to the Brill--Noether classes $\mathsf{w}_0(\phi)$ discussed here, and on how the wall--crossing approach can be used to produce an alternative formula for the 'interesting' double ramification cycle. 

In this context, a different approach to the problem of extending algebraic classes from the universal Jacobian is to study such extensions to the stack of all line bundles on all nodal curves (see e.g. \cite{bhpss}). This approach bypasses the difficulties of wall-crossing, but it comes with other difficulties related with the fact that the theory of algebraic cycles and their intersections on such large stacks is very delicate. It would be interesting to establish a precise way to compare this different approach with ours. 

The Hodge class $\lambda_g$ on $\Mmb{g}{n}$ can be thought of (up to a sign) as the pullback of $\mathsf{w}_0(\phi)$ (with $\phi$ such that the trivial line bundle is stable) via the Abel--Jacobi section associated with the trivial vector ${\bf d}=(0, \ldots, 0)$. 
More generally, the class $\mathsf{w}_d(\phi)$ is the formal analogue of the $\lambda_{g-d}$ class on $\Mmb{g}{n}$ (with the Hodge bundle $R\pi^\bullet_* (\omega_\pi)$ replaced by $R\pi^\bullet_* (\mathcal{L}_{\text{tau}}(\phi)$). Given the important role that the $\lambda$-classes  have played in the enumerative geometry of curves / intersection theory for moduli of curves, we expect that the same will hold for the Thom-Porteous extensions $\mathsf{w}_d(\phi)$. Also, as observed in Remark~\ref{independent of tauto}, the class~\eqref{classes} is independent of the choice of a tautological line bundle $\mathcal{L}_{\text{tau}}(\phi)$. This is not the case for other natural classes, \emph{e.g.} for the first Chern class of the pullback of $\mathcal{L}_{\text{tau}}(\phi)$ via some section, or for the pushforward of a power of $\mathcal{L}_{\text{tau}}(\phi)$ under a forgetful morphism.

Another application of our work, which we mentioned in the paragraph immediately before the beginning of this subsection, is that it produces a large number of tautological relations in $\Mb{g,n}$. The quest for finding all tautological relations has seen tremendous progress in recent years due to Pixton's work \cite{pixton}, where he produced a list of relations that he conjectures to be complete (we refer to \cite{pandacalculus} for a survey). However, establishing whether a given relation can be written as a linear combination of a given set of relations is computationally very hard. We do not know how our set of 'Brill-Noether tautological relations', \emph{i.e.} those given in Remark~\ref{Rem: relinmbgn}, compares with Pixton's set.



We conclude this subsection by emphasizing that the main technical tool of this paper, the construction of the nonsingular resolution $\jphitilde$ of the identity map opens the way, in principle, to resolve any other wall--crossing problem for algebraic cycles on compactified universal Jacobians that come as extensions from the interior $\mathcal{J}_{g,n}^d$. As it is not needed for our scopes, in this paper we do not discuss a modular description of the space $\jphitilde$, though we expect that one such description should be possible following the recent work \cite{molchorub} by Molcho.


\subsection{Acknowledgments} To be added after the refereeing process.

\section{Notation and preliminaries}

\subsection{Posets}

In this paper we will work with many posets (typically, the one underlying some category of stratifications, and some of its subposets). Here we recollect the relevant notation.

\begin{definition}
Let $P$ be a finite partially ordered set (or a poset).

A subset $C$ of  $P$ is called a \emph{chain}, if the partial order on $C$ induced by $P$ is a total order on $C$.

A poset is \emph{ranked} if for every element $a$, all maximal chains having $a$ as the largest element have the same length (called the \emph{rank} of $a$).

The poset $P$ is called a \emph{forest}, if for every $a\in P$ the lower set $\{b\leq a\}$ is a chain. More generally, we say that a subset $F\subseteq P$ is a forest if $F$ together with the partial order induced by $P$ is a forest.

If $a > b$ and there exists no $c$ such that $a > c>b$, then we say that $a$ \emph{covers} $b$, and write $a \gtrdot b$.
\end{definition}

\subsection{Graphs}
By a \emph{graph}  we mean a finite, connected, undirected multigraph, decorated with a genus function and markings (see for example \cite[Section~3.1]{ccuw} and  \cite[Section~2.1]{mmuv} for a precise definition).

If $G$ is a graph, we write $V(G)$ for its set of vertices and $E(G)$ for its set of edges, we write $g \colon V(G) \to \mathbb{N}$ for the genus function and $\leg \colon \{1, \ldots, n\} \to V(G) $ for the markings function. 

If $S \subseteq V(G)$, we write $G(S)$ for the complete subgraph of $G$ on the vertices $S$, and say that $G(S)$ is the subgraph of $G$ induced by $S$. 

Given $V_1, V_2 \subseteq V(G)$, we write $E(V_1,V_2)$ for the edges that have one endpoint in $V_1$ and the other in $V_2$ (if the edge is a loop, we include it if and only if its adjacent vertex is in both $V_1$ and $V_2$). 

If $f \colon G_1 \to G_2$ is a morphism of graphs, we also denote by $f \colon V(G_1) \to V(G_2)$ the corresponding surjection between vertices, and by $f^* \colon E(G_2) \to E(G_1)$ the corresponding injection at the level of edge sets.

If $G$ is a  graph and $E \subseteq E(G)$, we denote by $G^E$ the graph obtained from $G$ by adding exactly $1$ vertex, denoted $v_e$, in the ``interior'' of each edge $e \in E$. We call each such $v_e$ an  \emph{exceptional vertex} of $G^E$.

A graph $G$ is \emph{stable} if 
\[
2g(v)-2+|E(\{v\},\{v\}^c)|+|\leg^{-1}(v)| > 0
\]
for every vertex $v\in V(G)$.

\subsection{Families of curves and sheaves} \label{Not: curves and sheaves}
A \emph{nodal curve} $C$ is a reduced and connected projective scheme of dimension $1$ over some fixed algebraically closed field, with singularities that are at worst ordinary double points. The (arithmetic) \emph{genus} of $C$ is $p_a(C)=h^1(C, \mathcal{O}_C)$.
A \emph{subcurve} $X$ of $C$ is a  union of irreducible components of $C$. Its complement $X^c$ is the union of the other components of $C$.

A \emph{$n$-pointed curve} is a tuple $(C, p_1, \ldots, p_n)$ where $C$ is a nodal curve, and $p_1, \ldots, p_n$ are pairwise distinct nonsingular points of $C$. Its \emph{dual graph} $G(C)$ has the irreducible components of $C$ as vertices, the nodes of $C$ as edges, the geometric genus (resp. the marked points) of each component as the genus (resp. the markings) decoration.

A morphism $f \colon C'\to C$ of nodal curves is \emph{a semistable modification} if it is obtained by contracting some subcurves, not necessarily irreducible, $E \subset C'$ such that  $p_a(E)=0$ and $|E \cap E^c|=2$. Every subcurve $E \subset C'$ contracted by $f$ is called \emph{an exceptional curve of $f$}. A semistable modification such that every exceptional curve is irreducible is called a \emph{quasistable modification}.

A coherent sheaf on a nodal curve $C$ has \emph{rank~$1$} if its localization at each  generic point of $C$ has length $1$. It is \emph{torsion-free} if it has no embedded components. 
If the stalk of a torsion-free sheaf $F$ over $C$ fails to be locally free at a point $P\in C$, which must necessarily be a node, we will say that $F$ is \emph{singular} at $P$.
If $F$ is a rank~$1$ torsion-free sheaf on $C$ we say that $F$ is \emph{simple} if its automorphism group is $\mathbb{G}_m$ or, equivalently, if removing from $C$ the singular points of $F$ 
 does not disconnect $C$. 
 
A \emph{family of nodal curves} over a $\mathbb{C}$-scheme $S$ is a proper and flat morphism $\mathcal{C} \to S$ whose fibers are nodal curves. (Throughout, all families $\mathcal{C}/S$ will admit a distinguished section in the $S$-smooth locus of $\mathcal{C}$). A \emph{semistable (resp. a quasistable) modification} of the family $\mathcal{C}/S$ is another family $\mathcal{C}'/S$ with a $S$-morphism $f \colon \mathcal{C}'\to \mathcal{C}$ that is a semistable (resp. a quasistable) modification (as defined above) on all geometric points $s \in S$.

If $T$ is a $S$-scheme, a \emph{family of rank~$1$ torsion-free simple sheaves} parameterized by $T$ over a family of curves $\mathcal{C} \to S$ is a  coherent sheaf $F$   on $\mathcal{C}\times_S T$, flat over $T$, whose fibers over the geometric points are rank~$1$, torsion-free and simple.

If $F$ is a rank~$1$ torsion-free sheaf on a nodal curve $C$, the \emph{(total) degree} of $F$ is $\deg_C(F):=\chi(F)-h^0(C,\mathcal{O}_C)+p_a(C)$. 
If $X \subseteq C$ is a subcurve, we denote by $F_{X}$  the maximal torsion-free quotient of $F \otimes \mathcal{O}_X$. The total degree and the degree of $F_X$ and $F_{X^c}$ are related by the formula 
\begin{equation}
\label{rel-degree}
\deg_C(F) = \deg_{X} F + \deg_{X^c} F + \delta_{X \cap X^c} (F),
\end{equation}
where $\delta_S(F)$ is the number of points in $S$ where the stalk of $F$ fails to be locally free.

A line bundle $F'$ on a semistable modification $f \colon C'\to C$ is called \emph{positively admissible} (see \cite{estevespacini}) if $\deg_E(F')$ is either $0$ or $1$ on every exceptional subcurve of $f$. The following results follow from \cite[Section~5]{estevespacini}.

\begin{proposition}
\label{prop:estevespacini}
Let $\pi\col \mathcal{C}\to S$ and $\pi'\col \mathcal{C}'\to S$ be families of nodal curves and $f\col \mathcal{C}'\to \mathcal{C}$ be a semistable modification. Let $F'$ be a positively admissible line bundle on $\mathcal{C}'$ and set $F=f_*(F')$.
 \begin{enumerate}
     \item The sheaf $F$ is a torsion free rank-$1$ sheaf and $R^1f_*(F')=0$, in particular $f_*(F')$ commutes with base change. Moreover, we have that $R^\bullet\pi_*(F)=R^\bullet\pi'_*(F')$.
     \item The sheaf $F$ is invertible if and only if $\deg_E(F')=0$ on every exceptional subcurve of $f$. Moreover, in this case, $F'=f^*f_*(F')$.
     \item If $f$ is a quasistable modification  and $\deg_E(F')=1$ for every exceptional subcurve, then $\mathcal{C}'=\mathbb{P}_\mathcal{C}(F^\vee)$ and $F'$ is isomorphic to the tautological line bundle $\mathcal{O}_{\mathbb{P}_\mathcal{C}(F^\vee)}(1)$.
     \item More generally, we have that $f$ factors as $\mathcal{C}'\xrightarrow{g} \mathbb{P}_\mathcal{C}(F^\vee)\to \mathcal{C}$, and $\mathcal{O}(1)\cong g_*(F')$ and $F'\cong g^*(\mathcal{O}(1))$.
 \end{enumerate}
\end{proposition}


We can now define the multidegree of a sheaf on a nodal curve more precisely, as the multidegree of the unique positively admissible line bundle as in the above corollary. 

A degree $d$ \emph{pseudodivisor} on a graph $G$ is a pair $(E,D)$ where $E \subseteq E(G)$ and $D \in \Div^d(G^E)$ satisfies $D(v')=1$ for each exceptional vertex $v'$. When $E$ is empty, we simply write $D$ in place of the pair $(\emptyset, D)$.

Given a degree-$d$ rank~$1$ torsion free sheaf $F$ on a curve $C$, we define the multidegree $\underline{\deg}(F)$ of $F$ as the pseudodivisor $(E,D)$ on the dual graph $G(C)$ of $C$ as follows. The set $E$ is the set of edges of $G(C)$ that correspond to nodes of $C$ where $F$ is not locally free. The divisor $D$ on $G(C)^E$ is defined by $D(v)=\deg_{C_v}(F_{C_v})$ if $v\in V(G(C))\subseteq V(G(C)^E)$ and $D(v)=1$ for every exceptional vertex $v$. By Equation~\eqref{rel-degree}, we have that $(E,D)$ is a degree-$d$ pseudodivisor. 

Note also that a rank~$1$ torsion free sheaf on $C$ is simple if and only if its multidegree $(E,D)$ has the property that $E$ does not disconnect the graph $G(C)$.

\subsection{Moduli spaces and graphs}
\label{notation - moduli}

Here we discuss some general notation on moduli spaces of curves. We refer the reader to \cite{gac} for more details on nodal curves and their dual graphs.

A $n$-pointed curve $(C, p_1, \ldots, p_n)$ is \emph{stable} if $|\Aut(C, p_i)|< \infty$. 
We will sometimes abuse notation and write $C$ for $(C,p_1, \ldots, p_n)$. For example, we will say that the genus of $(C,p_1, \ldots, p_n)$ equals $g$ to mean that $h^1(C,\mathcal{O}_C)=g$, that is the arithmetic genus of the underlying curve $C$ is $g$. 

We will denote by $\Mmb{g}{n}$ the moduli space of stable $n$-pointed curves of genus~$g$. The moduli space admits a stratification with strata indexed by dual graphs, which we now discuss.

\subsubsection{Stable graphs}

We  denote by $G_{g,n}$ the small category of stable, $n$-marked graphs of genus $g$ (where we have fixed a choice of $1$ object for each isomorphism class).  Morphisms $G\to G'$ are given by the contraction of  some of the edges,  followed by an isomorphism. (More details in \cite{gac} and \cite{mmuv}). There is a natural functor $G_{g,n+1} \to G_{g,n}$ that forgets the last point and stabilizes the graph.

\subsubsection{Stratification of moduli of stable curves.} For $G \in G_{g,n}$ there is a gluing morphism
\[
f_G \colon  \overline{\mathcal{M}}_{G} := \prod_{v \in V(G)} \overline{\mathcal{M}}_{g_G(v),l_G(v)} \rightarrow \Big[ \prod_{v \in V(G)} \overline{\mathcal{M}}_{g_G(v),l_G(v)}/ \Aut(G) \Big]\twoheadrightarrow \overline{\mathcal{M}}_{G}'\hookrightarrow    \Mmb{g}{n},
\]
where $g_G\colon V(G) \to \mathbb{N}$ is the genus function of $G$ and $l_G \colon V(G) \to \mathbb{N}$ counts the number of markings plus gluing points based at $v \in V(G)$ (we refer the reader to \cite[Chapter~X.7, Chapter XII.10]{gac} for the details on how these gluing maps are defined).

We say that $\overline{\mathcal{M}}_G$, or $f_G$, is a (resolved) stratum of $\Mmb{g}{n}$. We regard $\overline{\mathcal{M}}_{G}$ as a ``resolved stratum'' and its image $\overline{\mathcal{M}}_{G}'$ as the corresponding ``embedded stratum''.

The codimension $1$ strata are the following divisors generically parameterizing curves with $1$ node: \begin{enumerate} \item the divisor $\Delta_{\text{irr}}$, generically parameterizing irreducible curves \item  for $0 \leq i \leq g$ and $S \subseteq [n]$ (except $i=0$ and $|S| <2$ and $i=g$ and $|S|> n-2$), the divisor $\Delta_{i,S}=\Delta_{g-i, S^c}$ generically parameterizing curves with $2$ components, of which one of genus $i$ carrying the marked points in $S$.
\end{enumerate}

Recall from \cite[Chapter~XIII, Section~3]{gac} that on the (resolved) stratum the normal bundle to $f_G$ splits as a direct sum of line bundles
\[
N_{f_G} = \bigoplus_{e \in E(G)} \mathbb{L}_e.
\]
We denote by $\Psi_e=-c_1(\mathbb{L}_e)$. Moreover, recall from \emph{loc.cit.} that if $e$ is the edge whose half edges $h(e), h'(e)$ are based at $v,v' \in V(G)$, then the cotangent line bundles to $h(e)$ and $h'(e)$ are denoted by $\mathbb{L}_{h(e)}$ and $\mathbb{L}_{h'(e)}$ and its first Chern classes $\psi_{h(e)}$ and $\psi_{h'(e)}$. We then have $\mathbb{L}_e= \mathbb{L}_{h(e)}^{\vee} \boxtimes \mathbb{L}_{h'(e)}^{\vee}$ and so $\Psi_e=\psi_{h(e)}+\psi_{h'(e)}$.

 In Subsection~\ref{section:nc and blowup} we will define the notion of a category of (resolved) strata induced by normal crossing divisors on a DM stack, and we will interpret the category $G_{g,n}$ as the category of strata of the nonsingular DM-stack $\Mmb{g}{n}$ induced by the normal crossing  divisor $\Delta = \Delta_{\text{irr}}+ \sum_{i,S} \Delta_{i,S}$. 

\section{Compactified Jacobians and Universal Brill--Noether Classes}

In this chapter we introduce the basic objects of study in this paper, compactified universal Jacobians, and extensions of universal Brill--Noether classes by means of Thom--Porteous formula. We also recall the  results on the stability space of compactified universal Jacobians that we will need later.

\subsection{The universal stability space}
Here we recall the definition and first results on the stability space of a single curve and on the universal stability space $V_{g,n}^d$ from \cite{kp3}.

\begin{definition} For a fixed graph $G$, we define the space of polarizations

\[V^d_{\text{stab}}(G):=\left\{\phi\in \mathbb{R}^{ V(G)} :\ \sum_{v \in V(G)} \phi(v) = d\right\} \subset \mathbb{R}^{V(G)}.\] 
For $V \subseteq V(G)$, we write $\phi(V)$ for $\sum_{v\in V} \phi(v)$.

Every morphism $f\colon G\rightarrow G'$ of  graphs induces a morphism $f_*\colon V^d_{\text{stab}}(G)\rightarrow V^d_{\text{stab}}(G')$ by setting
\begin{equation} \label{phicompatible} f_*\phi(v') = \sum \limits_{f(v)=v'} \phi(v)
\end{equation}
and we define the space of universal polarizations as the limit (or inverse limit) 
\[V^d_{g,n}:=\varprojlim_{G\in G_{g,n}} V^d_{\text{stab}}(G),\] i.e. as the space of assignments $\left(\phi(G)\in V^d_{\text{stab}}(G)\colon \ G \in G_{g,n}\right)$ that are compatible with all graph morphisms.
\label{spaceofpolar}
\end{definition}

We now present a simple description of the universal stability space $V_{g,n}^d$ that follows from  \cite[Corollary~4.3]{kp3}. The result requires that we introduce some notation for graphs of ``vine curves''.

\begin{definition} \label{vinecurvenotation}
A \emph{vine curve triple} $(i,t,S)$ consists of two natural numbers $i,t$ and a subset $S\subseteq [n]$, such that $0\leq i \leq g$, $1 \leq t$, $i+t \leq g+1$,  and such that if $(i,t)=(0,1)$ then $|S|\geq2$, if $(i,t)=(0,2)$ then $|S| \geq 1$, if $(i,t)=(g,1)$ then $|S^c| \geq 2$ and if $(i,t)=(g-1,2)$ then $|S^c| \geq 1$.

A \emph{vine graph} is a stable graph $G(i, t, S)$ associated to a vine curve triple, which consists of two vertices of genus $i$ and $g-i$ connected by $t$ edges, and with marking $S$ on the first vertex and $S^c$ on the second vertex. We will always assume that $S$ contains the first marked point.  The \emph{vine stratum} corresponding to a vine graph $G$ is the corresponding stratum $\overline{\mathcal{M}}_G$.

The stability space $V^d_{\text{stab}}(G(i,t,S))$ is an affine subspace of $\mathbb{R}^2$. We can parameterize it by means of one real variable $x_{i,t,S}$ by taking the inverse image under the projection onto the first factor. That is, we describe \begin{equation} \label{eq: coord} V^d_{\text{stab}}(G(i,t,S))=\{(x_{i, t,S}, d-x_{i,t,S}): x_{i, t, S} \in \mathbb{R}\}\subset \mathbb{R}^2.\end{equation}
\end{definition}

We now introduce the stability space of ``vine curves'' using the previous definition.

\begin{definition} \label{CDT} We let
\[
		T_{g,n}^d:=   \prod_{\substack{(i,t,S) \\ \textrm{ a vine curve triple }}} V^d_{\text{stab}}(G(i, t, S)).
	\]
Then we define:

\begin{enumerate}

\item The vector space $C^d_{g,n}$  as the product of all stability spaces of the form $V^d_{\text{stab}}(G(i,1, S))$.

\item The vector space $D^d_{g,n}$ as the product of all stability spaces $V^d_{\text{stab}}(G(0,2,\{j\})$ for $j=1\ldots, n$. 
\end{enumerate}
Throughout this paper we will use the coordinates $x_{i,t,S}$ from Equation~\ref{eq: coord} on the spaces $T_{g,n}^d$ and hence on $C_{g,n}^d$ and $D_{g,n}^d$ (which we view as quotients of $T_{g,n}^d$).
\end{definition}

There are natural restriction affine linear maps:
\[
\tau_d \colon V_{g,n}^d \to T_{g,n}^d, \quad \rho_d \colon V_{g,n}^d \to C_{g,n}^d \times D_{g,n}^d
\]
that, when $d$ equals zero, are homomorphisms of vector spaces. One of the main results of \cite[Section~3]{kp3} is that the universal stability space embeds into the ``vine curves'' stability space, and that $\rho_d$ is an isomorphism.
\begin{proposition} \label{stabilityspace}
 (\cite[Lemma~3.8, Corollary~3.4]{kp3}) The affine linear map $\tau_d$ is injective. Each morphism $\rho_d$ is an isomorphism of affine spaces.
\end{proposition}

\subsection{The stability hyperplanes} \label{subsec: stabilityH}
We will later see in Section~\ref{Section:compjac} that for every universal stability condition $\phi \in V_{g,n}^d$ there exists a compactified universal Jacobian parameterizing $\phi$-stable (rank~$1$, torsion free) sheaves on every (flat) family of $n$-pointed stable curves of genus $g$.  Here we combinatorially introduce  the degenerate locus of $V_{g,n}^d$, which  will later be seen to be the locus of $\phi$'s such that there exist strictly semistable sheaves on some stable curves. We will introduce the degenerate locus as a union of hyperplanes (which one could think of as a finite \emph{toric} hyperplane arrangement). This explicit description is taken from  \cite[Section~5]{kp3}. 

\begin{definition} \label{Def:degenerate}	
 We say that a polarization $\phi \in V^d_{\text{stab}}(G)$ is \emph{degenerate} if for some subset  $\emptyset \subsetneq V_0 \subsetneq V(G)$ the number
	\begin{equation}
	 \frac{\left|E(V_0, V_0^c)\right|}{2}+ \sum \limits_{v \in V_0 } \phi(v)   
		\label{Eqn: defdegenerate}
	\end{equation}
	is an integer.

	We say that a universal stability condition $\phi \in V_{g,n}^d$ is \emph{degenerate} if for some $G \in {G}_{g,n}$, the $G$-component $\phi(G)$ is degenerate in $V^d_{\text{stab}}(G)$.
\end{definition}
The degenerate locus is a locally finite union of affine hyperplanes, and we will soon describe these hyperplanes explicitly. Let us start with a simple example.

\begin{example} \label{vinecurves} (Vine curves). If $G$ is a vine graph, after  identifying 
$V^d_{\text{stab}}(G)=\mathbb{R}$ by projecting onto the first factor (as done at the end of Definition~\ref{vinecurvenotation}), we have that the degenerate locus is a locally finite collection of points that only depends on the parity of the number of nodes $t$.   If $t$ is even, the degenerate locus corresponds to $\mathbb{Z} \subset \mathbb{R}$. If $t$ is odd, the degenerate locus corresponds to the $\frac{1}{2}+\mathbb{Z} \subset \mathbb{R}$.
\end{example}

We now give an explicit description of the degenerate locus in $V_{g,n}^d$, based on \cite[Section~5]{kp3}. By Proposition~\ref{stabilityspace}, we have that $V_{g,n}^d\subset T_{g,n}^d$, where the latter is the stability space of vine curves (one for each topological type), with coordinates $x_{i,t,S}$ for each vine curve triple $(i,t,S)$ (see Definition~\ref{vinecurvenotation}).

For each vine curve triple $(i,t,S)$ and integer $k$, define the (translate of the coordinate) hyperplane
\begin{equation} \label{transcoord}
T_{g,n}^d \supset H(i,t,S;k):= k+\frac{t}{2}.
\end{equation}
One main result of \cite[Section~5]{kp3} is that the degenerate locus in the universal stability space is the union of the pullback from the stability space of vine curves of all coordinate hyperplanes and their integer translates. More precisely:
\begin{proposition} \label{Prop: degen-on-vine} (\cite[Lemma~5.8]{kp3})
 The degenerate locus in $V_{g,n}^d$ is a union of hyperplanes. Each hyperplane is the inverse image via the affine linear embedding $\tau_d\colon V_{g,n}^d \subset T_{g,n}^d$ of a hyperplane of the form $H(i, t, S;k)$.
\end{proposition}

A more explicit description of the degenerate locus can be obtained via the isomorphism $V_{g,n}^d \cong C_{g,n}^d \times D_{g,n}^d$. When expressing the hyperplanes of \eqref{transcoord} in terms of the coordinates $x_{i,1,S}$ and the coordinates $x_j:=x_{0,2,\{j\}}$, by \cite[Theorem~2]{kp3} we have\footnote{Note that the formula in loc.cit is translated by the coordinates of a ``degree-$d$ canonical stability condition'' -- a choice of an origin in $V_{g,n}^d$ that we do not discuss here.}

\begin{equation}
    \label{changeofcoord}
    x_{i,t,S}= \frac{2g-2i-t}{2g-2} \cdot \sum_{j \in S} x_j + \frac{2i-2 +t}{2g-2} \cdot \Bigg(d-\sum_{j \notin S} x_j\Bigg)
\end{equation} whenever $t \geq 2$. Therefore, the stability hyperplanes take the following form
\begin{equation} \label{walls1}
    H(i,1,S;k)=\left\{x_{i,1,S}= k+\frac{1}{2}\right\} 
\end{equation}
for all vine curve triples (Definition~\ref{vinecurvenotation})  of the form $(i,1,S)$ (the boundary divisors in $\Mmb{g}{n}$ that generically parameterize curves with  $2$ components)
\begin{equation} \label{walls2}
    H(i,t,S;k)=\left\{\frac{2g-2i-t}{2g-2} \cdot \sum_{j \in S} x_j + \frac{2i-2 +t}{2g-2} \cdot \Bigg( d- \sum_{j \notin S} x_j\Bigg)= k+ \frac{t}{2} 
    \right\} 
\end{equation}
for all vine curve triples $(i, t, S)$ with $t \geq 2$. 

Note that the  hyperplanes in \eqref{walls1} and \eqref{walls2} may come with multiplicities in the degenerate locus. In other words, there may exist different $(i_1, t_1, S_1; k_1)$, $(i_2, t_2, S_2; k_2)$ such that  $H(i_1, t_1, S_1;k_1) = H(i_2, t_2, S_2;k_2)$. We will now analyze when this may occur. 

It is immediate to observe that a necessary condition for two hyperplanes of this form to coincide is that their corresponding subsets of marked points must also coincide:

\begin{proposition} \label{prop: subsetscoincide}
 Suppose $(i_1, t_1, S_1; k_1)$ and $(i_2, t_2, S_2; k_2)$ are such that  $H(i_1, t_1, S_1;k_1)=H(i_2, t_2, S_2;k_2)$. Then $S_1=S_2$.
\end{proposition}
\begin{proof}
Straightforward.
\end{proof}

First we deal with the hyperplanes of \eqref{walls1}, occurring on compact type vine strata (i.e. those of codimension $1$). Those are all simple, in the following sense: 
\begin{proposition} \label{prop:ctwalls}
 Suppose $(i_1, 1, S_1; k_1)$ and $(i_2, t_2, S_2; k_2)$ are such that  $H(i_1, 1, S_1;k_1)=H(i_2, t_2, S_2;k_2)$. Then $(i_1, 1, S_1; k_1)=(i_2, t_2, S_2; k_2)$.
\end{proposition}
\begin{proof}
Straightforward.
\end{proof}

The next proposition is about hyperplanes of the form \eqref{walls2} with $S \neq [n]$. 
As we shall discuss in Section~\ref{sec:goodwalls}, a stability hyperplane of this type corresponds to a change of stability on loci of vine curves that are \emph{disjoint}.

\begin{proposition} \label{prop:goodwalls}
  Suppose $(i_1, t_1, S; k_1)$ and $(i_2, t_2, S; k_2)$ are such that  $H(i_1, t_1, S;k_1)=H(i_2, t_2, S;k_2)$, and assume $S \neq [n]$. Then $2i_1+t_1 = 2i_2+t_2$.
\end{proposition}
\begin{proof}
Straightforward.
\end{proof}

The most interesting vine strata from the point of view of the stability decomposition  are those  with $S= [n]$. Over those vine strata it can occur that two stability hyperplanes of the form  \eqref{walls2} coincide.  For example, if $d=0$, by fixing $\sum_{j \in [n]} x_j = g-1$ one sees that all  hyperplanes of the form $H(i,t,[n] ;k)$ with $i+ t  +k=g$  coincide (note that this collection of hyperplanes  is finite, because of the constraints $i+t \leq g+1$, $i\geq 0$ and $t \geq 2)$.

\subsection{Compactified Jacobians, universal and semistable families} \label{Section:compjac}

Here we define, for every nondegenerate $\phi \in V_{g,n}^d$, a fine compactified universal Jacobian $\Jmb dgn(\phi)$, parameterizing $\phi$-stable sheaves. The construction is taken from \cite[Section~4]{kp3}, in the language of pseudodivisors from \cite[Section~4]{abreupacini}. Each fine compactified Jacobian will come with a normal crossing stratification category (an abstract definition of this notion will be given and discussed in the next section).

We also describe (Theorem~\ref{prop: univ-semistable-mod}) a quasistable modification of the universal curve $\Cmb gn (\phi) \to \Jmb dgn(\phi)$ as a certain $(n+1)$-universal Jacobian $\Jmb dg{n+1} (\alpha(\phi))$.

\begin{definition} \label{def: Dstable} For $\phi \in V^d_{\text{stab}}(G)$ we say that a degree~$d$ pseudodivisor $(E,D)$ is \emph{$\phi$-semistable} if
\begin{equation} \label{Eqn: SymDefOfStability}
\phi(V_0)-\deg_{V_0}(D)+\frac{|E(V_0,V_0^c)|}{2}\geq   0
\end{equation}
for every $V_0\subseteq V(G^E)$. We say that $(E,D)$ is \emph{$\phi$-stable} if the inequality above is strict for every $V_0$ such that $V_0\neq V(G^E)$ and $V_0$ is not contained in the set of exceptional vertices. Given $v_0\in V(G)$, we say that $(E,D)$ is \emph{$(\phi,v_0)$-quasistable} if the inequality is strict for every $V_0$ such that $V_0\neq V(G^E)$ and $v_0\in V_0$. 
\end{definition}
As stipulated in Subsection~\ref{Not: curves and sheaves}, when $E=\emptyset$, we will simply write $D$ for $(\emptyset, D)$.

\begin{remark}
By \cite[Proposition~4.6]{abreupacini} if a pseudodivisor $(E,D)$ on $G$ is $(\phi,v_0)$-quasistable for some $(\phi,v_0)$, then $E\subseteq E(G)$ does not disconnect $G$.
\end{remark}

\begin{remark} \label{remark-deg}
We have introduced the degenerate locus of $V^d_{\text{stab}}(G(C))$ and of $V_{g,n}^d$ in Definition~\eqref{Def:degenerate}. We claim that, in both cases, an element $\phi$ is nondegenerate if and only if all semistable pseudodivisors are stable. The ``only if'' is immediate. The other implication is proved in~\cite[Section~5]{kp3}.
\end{remark}

We now define stability for rank~$1$ torsion free sheaves on curves.

\begin{definition} (\cite[Definition~4.2]{kp3})\label{phistab}	
 Let $C$ be a nodal curve with dual graph $G(C)$  and let $\phi \in V^d_{\text{stab}}(G(C))$. 
 A rank~$1$ torsion-free  sheaf $F$  of degree $d$ on $C$ is \emph{$\phi$-(semi)stable} if its multidegree $\underline{\deg}(F)$ is a $\phi$-(semi)stable pseudodivisor.
 
If $P \in C^{\text{sm}}$ is a nonsingular point of $C$ in the component $C_{v_0}$, we say that $F$ is \emph{$(\phi,P)$-quasistable} if $\underline{\deg}(F)$ is $(\phi,v_0)$-quasistable.

If $C' \to C$ is a semistable modification of $C$ and $F'$ is a positively admissible line bundle on $C'$, we say that $F'$ is $\phi$-(semi)stable or $(\phi,P)$-quasistable if so is $f_*(F')$.

 For $\phi\in V^d_{\text{stab}}(G(C))$ and $P \in C$,  we define $\overline{\mathcal{J}}^d_{\phi,P}(C)$ to be the scheme  parameterizing  $(\phi,P)$-quasistable sheaves. 
\end{definition}
Note that if $F$ is a rank~$1$ torsion free sheaf on $C$ then (1) if $F$ is $(\phi,P)$-quasistable then it is simple, and (2) the sheaf $F$ is simple if and only if its multidegree $(E,D)$ has the property that  $E\subseteq G(C)$ is nondisconnecting.

\begin{remark} \label{Rem: perturb} Let $\phi \in V^d_{\text{stab}}(G(C))$ and $P \in C^{\text{sm}}$ be as above. Let $\phi'\in V^d_{\text{stab}}(G(C))$ be a small perturbation of $\phi$ obtained by subtracting a small $\epsilon>0$ from $\phi$ on the vertex of $G(C)$ containing $P$, and by adding a small positive amount on all other vertices (so that $\sum \phi'(v)=\sum \phi(v)=d$).  

Then $(\phi,P)$-quasistability coincides with $\phi'$-stability which in turn coincides with $\phi'$-semistability.
\end{remark}

We are now ready to introduce  compactified universal Jacobians. 
For any $1 \leq i \leq n$, we denote by $\sigma_i \colon \Mmb gn \to \Cmb gn$ the $i$-th section.

	\begin{definition} \label{def: semistablecat}
		Let $\phi \in V_{g,n}^d$ be a universal polarization.

  We define $\catJ$ to be the category whose objects are $(G,(E,D))$ where $G$ is an object of $G_{g,n}$ and $(E,D)$ is a  $\phi$-semistable pseudodivisor on $G$. A morphism $(G,(\emptyset,D)) \to (G',(\emptyset, D'))$ in $\catJ$ is a morphism $f\in \Mor_{G_{g,n}}(G,G')$ such that the induced homomorphism $f_{\ast}\colon \Div(G) \to \Div(G')$ on divisors satisfies $f_\ast(D) = D'$. We refer to \cite[Section~2.1]{abreupacini}  for the notion of a morphism $(G,(E,D)) \to (G',(E', D'))$ when $E, E'$ are nonempty.
  
  Similarly, we define $\mathfrak{C}_{g,n}(\phi, \sigma_i)$ to be the category whose objects are $(G,(E,D))$ and $(E,D)$ is $(\phi, \sigma_i)$-quasistable. (By abuse of notation, we also denote by $\sigma_i$ the vertex of $G$ that corresponds to the component of $C$ containing the image of the section $\sigma_i$).
  
		We say that a family of rank~$1$ torsion-free  sheaves of degree $d$ on a family of stable curves is \emph{$\phi$-(semi)stable} or $(\phi, \sigma_i)$-quasistable if that property holds for all geometric fibers. We define $\Jmb dgn(\phi)$ to be the moduli stack  parameterizing  $\phi$-semistable sheaves on families of stable curves. We define $\Jmb dgn(\phi,\sigma_i)$ to be the moduli stack  parameterizing  $(\phi,\sigma_i)$-quasistable sheaves on families of stable curves. \label{polarised}
\end{definition}
As anticipated in Section~\ref{Not: curves and sheaves}, we will shorten the notation and write $(G,D)$ instead of $(G, (\emptyset, D))$.

\begin{notation}
If $\phi \in V_{g,n}^d$ is nondegenerate then, by Remark~\ref{remark-deg} we have that all semistable sheaves are stable. It follows that for every $1 \leq i \leq n$ we have the equalities $\catJ=\mathfrak{C}_{g,n}(\phi, \sigma_i)$ and $\Jmb dgn(\phi)=\Jmb dgn(\phi,\sigma_i)$, where $\sigma_i \colon \Mmb gn \to \Cmb gn$ is the $i$-th section.
\end{notation}
\begin{remark}
For $\phi \in V_{g,n}^d$ degenerate and for all $1 \leq i \leq n$ we can describe $\mathfrak{C}_{g,n}(\phi, \sigma_i)$ (resp. $\Jmb dgn(\phi,\sigma_i)$) as $\mathfrak{C}_{g,n}(\phi'_i)$ (resp. $\Jmb dgn(\phi'_i)$) for some nondegenerate perturbation $\phi_i'$ of $\phi$. (As done in Remark~\ref{Rem: perturb} for a single curve). 

In order to achieve this we  employ Proposition~\ref{stabilityspace}. We  define $\phi_i'$ by subtracting  from $\phi \in C_{g,n}^d \times D_{g,n}^d$ some sufficiently small $\epsilon>0$ on the $i$-th component of the factor $D_{g,n}$, and by adding  $\epsilon/(n-1)$ on all other $n-1$ components of $D_{g,n}$ (so for all curves, the sum over all irreducible components of the values of $\phi$ and of $\phi_i'$ coincide).
\end{remark}

The following guarantees the existence of  universal moduli spaces.

\begin{theorem} (\cite[Corollary~4.4]{kp3} and \cite{esteves}/\cite{melouniversal}) \label{existenceofMS}
For all $\phi \in V_{g,n}^d$ and $1 \leq i \leq n$ the stack $\Jmb dgn (\phi, \sigma_i)$ is a nonsingular Deligne--Mumford stack, and  the forgetful morphism $\Jmb dgn(\phi, \sigma_i) \to \Mmb gn$ is representable, proper and flat. 
\end{theorem}

The moduli stacks of Theorem~\ref{existenceofMS} are called \emph{fine compactified universal Jacobians}. 

As observed in \cite[Remark~4.6]{kp3}, the fine compactified (universal) Jacobians produced by this construction are the same as those defined by Esteves and Melo \cite{esteves,melouniversal}.

By the universal property of $ \Jmb dgn (\phi)$, the universal family $\pi_n \colon \Cmb gn(\phi) \to \Jmb dgn (\phi)$ carries some tautological (or Poincar\'e) sheaf  $\mathcal{F}_{\text{tau}}(\phi)$.  These are of fiberwise total degree $d$ and $\phi$-stable. If $\mathcal{F}_{\text{tau},1}(\phi)$ and $\mathcal{F}_{\text{tau},2}(\phi)$ are any two tautological sheaves, there exists a line bundle $M$ on $\Jmb dgn (\phi)$ such that \[\mathcal{F}_{\text{tau},2} = \mathcal{F}_{\text{tau},1} \otimes \pi_n^*M.\] One way to make a definite choice of a tautological sheaf is to assume that it is trivial along a given section.

Note that, as described in \cite{cmkv1}, the total space $\Cmb gn(\phi)$ is singular. A natural desingularization of $\Cmb gn(\phi)$, carrying a tautological line bundle $\mathcal{L}_{\text{tau}}(\phi)$, was provided by Esteves--Pacini in \cite{estevespacini} by using a semistable modification of the universal family. Here we will give an alternative description of it using a compactified universal Jacobian with one extra point.

\begin{remark} We observe that there is a natural map $\alpha \colon V_{g,n}^d \to V_{g,n+1}^d$, with image in the degenerate locus, defined as follows. 

If $G$ is the stable graph obtained as the  stabilization of the $(n+1)$-pointed graph $G'$ after the $n+1$ marking is removed, then there is a natural bijection between the vertices of $G'$ and those of $G$, except possibly for  $1$ extra genus $0$ vertex of $G'$. Then define $\alpha(\phi)$ as the assignment on $G'$ that is defined by this bijection, and that is $0$ on the extra genus $0$ vertex of $G'$ (when that exists). The extra genus $0$ vertex could be a tail (when it is connected to the complement by $1$ edge) or a bridge (when it is connected to the complement by $2$ edges). The fact that $\phi$ is compatible for graph morphisms implies the same property for $\alpha(\phi)$. \end{remark}  

\begin{notation}
We will slightly abuse the notation and,  for $\phi \in V_{g,n}^d$, we will simply write $\phi \in V_{g,n+1}^d$ in place of $\alpha(\phi) \in V_{g,n+1}^d$.
\end{notation}
 
 We now show that a quasistable modification of the universal curve can be described as the  morphism  $\pi' \colon \Jb{d}{g,n+1}(\phi, \sigma_i) \to \Jb{d}{g,n}(\phi,\sigma_i)$ that forgets the last point and stabilizes, thus mapping each $(C',p_1, \ldots, p_{n+1},F)$ to $(C,p_1, \ldots, p_n,f_*F)$, where $f \colon C' \to C$ is the stabilization of $(C', p_1, \ldots, p_n)$. In order to do that, for each fixed $i=1,\ldots, n$, we define a morphism $\psi\colon \Jb{d}{g,n+1}(\phi,\sigma_i) \to \Cmb g{n}(\phi, \sigma_i)$  by 
\[
(C',p_1, \ldots, p_{n+1},L)  \mapsto (C,p_1, \ldots, p_n,f_*L, f(p_{n+1})).
\]
 Then we show that $\psi$ is the stabilization of $\pi'$.
\begin{theorem} \label{prop: univ-semistable-mod}
For each $1 \leq i \leq n$, there is an isomorphism $\Jmb dg{n+1}(\phi, \sigma_i)\to \mathbb{P}(F_{\text{tau}}^\vee)$ such that the following diagram commutes
\begin{equation} \label{eq: commdiag}
\begin{tikzcd}
    \Jmb{d}{g}{n+1} (\phi, \sigma_i) \ar[rr] \ar[dr, "\psi"] & &\mathbb{P}(F_{\text{tau}}^\vee) \ar[ld]\\
    &  \Cmb g{n}(\phi, \sigma_i). & \\
\end{tikzcd}.
\end{equation}
Moreover, a tautological line bundle (\emph{i.e.} a line bundle isomorphic to the pullback of $\mathcal{O}(1)$) on $\Jb{d}{g,n+1}(\phi, \sigma_i)$ is
\begin{equation} \label{deftau}
\mathcal{L}_{\textnormal{tau}}:=\sigma_{n+1}^*(\mathcal{F}'_{\textnormal{tau}}) \otimes \sigma_{1}^*(\mathcal{F}'_{\textnormal{tau}})^{-1},
\end{equation}
where $\mathcal{F}'_{\textnormal{tau}}=\mathcal{F}'_{\textnormal{tau}}(\phi)$ is a tautological sheaf on the universal curve 
\[
\pi_{n+1}\colon \Cmb g{n+1}(\phi, \sigma_i) \to \Jmb dg{n+1} (\phi,\sigma_i).
\]
\end{theorem}

\begin{proof}
Firstly, we apply Proposition~\ref{prop:estevespacini} to show that the morphism $\psi$ is a quasistable modification of the universal curve  $\overline{\mathcal{C}}_{g,n}(\phi, \sigma_i)$. If $C=C'$, then $f$ is an isomorphism, so $\psi$ is an isomorphism locally around $(C',p_1,\ldots, p_{n+1},L)$. If $C\neq C'$, then we have two cases. Either $p_{n+1}$ lies on a rational tail and, in this case, $L=f^*f_*(L)$ so $\psi$ is again an isomorphism locally around $(C',p_1,\ldots, p_{n+1},L)$. Or, $p_{n+1}$ is on a bridge on $E\subset C'$ such that no other marked points are in $E$. We will now focus on this case.

If $\deg_E(L)=0$, then by Proposition~\ref{prop:estevespacini}, we have that $L=f^*f_*(L)$ and again the map $\psi$ is an isomorphism locally around $(C',p_1,\ldots, p_{n+1},L)$. We are left with the case where $\deg_E(L)=1$. In this case, we have that $f_*(L)$ is not locally free around $f(p_{n+1})$, which is a node. More so, we have that $\psi^{-1}(C,p_1,\ldots, p_n,f_*(L), f(p_{n+1}))$ is isomorphic to $\mathbb{P}^1$. Indeed, every $L'$ obtained from gluing $L|_{E^c}$ and $\mathcal{O}_E(1)$ has the property that \[\psi(C',p_1,\ldots, p_{n+1}, L') = (C,p_1,\ldots, p_n, f_*(L), f(p_{n+1})).\] The possible gluings are set-theoretically parameterized by a $\mathbb{P}^1$. Moreover, since the local rings of the Jacobians over such points are known via deformation theory (see \cite[Equation~8.17]{CMKV17}), we can see that the fiber $\psi^{-1}(C,p_1,\ldots, p_n,f_*(L), f(p_{n+1}))$ is actually reduced and transversal to the closure in $\Jb{d}{g,n+1}(\phi, \sigma_i)$ of
\[
(\pi')^{-1}(C,p_1,\ldots, p_n,f_*(L))\setminus \psi^{-1}(C,p_1,\ldots, p_n,f_*(L), f(p_{n+1})).
\]

Secondly, we observe that $\deg_E(\mathcal{L}_{\text{tau}})=1$ on each exceptional component $E$ contracted by $\psi$. In order to show this, it suffices to construct a nonconstant map $\delta \colon \mathbb{P}^1\to E$ such that $\delta^*(\mathcal{L}_{\text{tau}})= \mathcal{O}(1)$. Let $(C,p_1, \ldots, p_n, L, p)$ correspond to the point in the universal curve that is contracted by $E$. That means that $p$ is a node of $C$ and that $L$ fails to be locally free at $p$. We construct a family $X/\mathbb{P}^1$ by gluing two sections on two families $X_1/\mathbb{P}^1$ and $X_2/\mathbb{P}^1$. The family $X_1$ is the trivial family $C^{\nu_p} \times \mathbb{P}^1$ (where $\nu_p$ denotes the normalization at $p$) carrying the $n$ trivial sections whose images are $p_i\times \mathbb{P}^1$, for $i=1, \ldots, n$, and the gluing sections are those whose images are $q_1\times \mathbb{P}^1$ and $q_2\times \mathbb{P}^1$ such that  $\nu_p(q_i)=p$. The family $X_2$ is the blowup of the trivial family $\mathbb{P}^1 \times \mathbb{P}^1$ at $[0:1] \times [0:1]$ and $[1:0] \times [1:0]$ with a further section $p_{n+1}$ defined as the inverse image of a constant section (different from $[0:1]$ and $[1:0]$) and the two gluing sections are the strict transforms of the sections $[0:1]$ and $[1:0]$. Then we choose any line bundle $F$ on $X$ with the property that $F|_{X_1}=\xi^*(L)/\text{Torsion}$, where $\xi\colon X_1\to C$ is the projection $X_1\to C^{\nu_p}$ composed with the partial normalization $C^{\nu_p}\to C$,  and $F|_{X_2}=\mathcal{O}(\widetilde{\Delta})$, for $\widetilde{\Delta}$ the strict transform of the diagonal in $X_2$. Then $F$ is $(\phi, \sigma_i)$-quasistable, and so the datum of $(X,F)$ defines a morphism $\delta$. By construction, we have that $\delta^*(L_{\text{tau}})=\sigma_{n+1}^*(F)= \sigma_{n+1}^*(\mathcal{O}(\widetilde{\Delta}))$, which equals $\mathcal{O}(1)$ because the image of $\sigma_{n+1}$ intersects $\widetilde{\Delta}$ at $1$ reduced point. 

As before, let $\mathcal{F}_{\text{tau}}$ be a tautological sheaf over $\mathcal{C}_{g,n}(\phi)$.
To conclude, we prove that the direct image $\psi_* \mathcal{L}_{\text{tau}}$ of the line bundle defined in \eqref{deftau} equals  $\mathcal{F}_{\text{tau}} \otimes \pi_n^*(M)$ for some line bundle $M$ on $ \Jmb dg{n} (\phi,\sigma_i)$. By the previous part combined with Proposition~\ref{prop:estevespacini}, we conclude that $\psi_*\mathcal{L}_{\text{tau}}$ is rank~$1$ and torsion-free. By \cite[Appendix~7]{kp3}, it is enough to prove that this equality occurs on an open subset $U$ of $\Cmb g{n}(\phi)$ whose complement has codimension at least $2$. It is easy to show equality over the open set $U$ that is the universal curve over the line bundle locus in $\Jb{d}{g,n}(\phi)$ (this follows because $\psi_{|\psi^{-1}(U)}$ is an isomorphism over $U$, and because the restriction of $\mathcal{F}'_{\text{tau}}$ to the open set $\pi_{n+1}^{-1}(\psi^{-1}(U)) \subset \Cmb g{n+1}(\phi, \sigma_i)$ is a line bundle). This concludes the proof that $\psi$ is a positively admissible quasistable modification.

Since the degree of $\mathcal{L}_{\mathrm{tau}}$ restricted to any exceptional curve is $1$, the existence of the isomorphism  $\Jmb dg{n+1}(\phi, \sigma_i)\to \mathbb{P}(F_{\text{tau}}^\vee)$, and the fact that Diagram~\eqref{eq: commdiag} commutes, follow from Proposition~\ref{prop:estevespacini} Item~(3) and (4).
\end{proof}

\subsection{Brill--Noether classes} \label{subsec: BN}

To the data of a flat family $\pi \colon \mathcal{C}\to S$ of nodal curves of arithmetic genus $g$ over a nonsingular scheme $S$ and a rank $1$ torsion-free $\mathcal{F}$ on $\mathcal{C}$ of fiberwise degree $d$, we can associate the Brill--Noether (or Thom--Porteous) class
\begin{equation} \label{thomporteous}
\mathsf{w}_d(\mathcal{C}/S,\mathcal{F}):= c_{g-d}(-R^{\bullet} \pi_\ast \mathcal{F}).
\end{equation}
This class is supported on the subscheme
\[
\mathsf{W}_d(\mathcal{C}/S,\mathcal{F})=\{s \in S: \ h^0(\mathcal{C}_s,\mathcal{F}_s)>0\} \subset S
\]
and when the latter is of the expected codimension $g-d$, it coincides with its fundamental class (with a suitably defined scheme structure-- see \cite[Chapter~14]{fulton}). If  $h^0(\mathcal{C}_s,\mathcal{F}_s)=0$  for all $s \in S$, then the complex $-R^{\bullet} \pi_\ast \mathcal{F}=R^{1} \pi_\ast \mathcal{F}$ is  a vector bundle of rank $g-1-d$, so the Chern class $\mathsf{w}_d(\mathcal{C}/S,\mathcal{F})$ equals zero. 

Here are a couple of further basic remarks on these classes.

\begin{remark} \label{independent of tauto}
If $I$ is a line bundle on $S$, then we have
\[\mathsf{w}_d(\mathcal{C}/S,\mathcal{F})= \mathsf{w}_d(\mathcal{C}/S,\mathcal{F}\otimes \pi^* I ).\]

Indeed, by \cite[Example~3.2.2]{fulton} we have 
\begin{equation} \label{tensorbylinebundle}
c_j(-R^{\bullet} \pi_\ast \mathcal{F} \otimes I) = \sum_{i=0}^j \binom{g-d-1-i}{j-i} c_i(-R^{\bullet} \pi_\ast \mathcal{F})  \cdot c_1(I)^{j-i}
\end{equation}
for all $j \geq 0$.  The result follows because, for $j=g-d$, the binomial coefficient vanishes unless when $i=j$.
\end{remark}

\begin{remark}
Let $f \colon \mathcal{C}' \to \mathcal{C}$ be a semistable modification of the family  of nodal curves $\pi \colon \mathcal{C} \to S$, and let $\mathcal{L}$ be a positively admissible line bundle on $\mathcal{C}$ (see Subsection~\ref{Not: curves and sheaves}). By Proposition~\ref{prop:estevespacini}, we deduce
\begin{equation} \label{same BN}
R^{\bullet} (\pi \circ f)_{\ast} \mathcal{L} =R^{\bullet} \pi_\ast ( f_\ast \mathcal{L}).\end{equation}
Conversely, if $\mathcal{F}$ is a rank~$1$ torsion free simple sheaf on a family of stable curves $\mathcal{C}/S$, there exists a quasistable modification $\mathcal{C}'$ of $\mathcal{C}$ and a line bundle $\mathcal{L}$ on $\mathcal{C}$ such that $R^{\bullet} f_\ast (  \mathcal{L})=f_\ast  \mathcal{L} = \mathcal{F}$ and thus \eqref{same BN} occurs.
\end{remark}

The same construction and remarks apply to the case of the semistable modification $\pi \colon \Cb{g,n}'(\phi) \to \Jmb{d}{g}{n}(\phi)$ of the universal family, and its tautological line bundle $\mathcal{L}=\mathcal{L}_{\text{tau}}$ (see \cite{prvZ}). We will denote by $\mathsf{w}_d(\phi)$ the corresponding universal class in $A^d(\Jb{d}{g,n}(\phi))$ and with $\mathsf{W}_d(\phi)$ the subscheme over which it is supported.

\begin{remark} When restricted to smooth curves, the scheme $\mathsf{W}_d(\phi)_{|\J{d}{g,n}}$ is reduced, irreducible, and of relative codimension $g-d$ (it is the image of the $d$-th symmetric product via the Abel map). The closure of $\mathsf{W}_d(\phi)_{|\J{d}{g,n}}$ in $\Jmb{d}{g}{n}(\phi)$ is contained in $\mathsf{W}_d(\phi)$ and when the two coincide, we have $\mathsf{w}_d(\phi)=[\mathsf{W}_d(\phi)]$. \label{closureBN}
\end{remark}

\begin{remark} In general, the scheme $\mathsf{W}_d(\phi)$ fails to be irreducible and of the expected dimension. For example, if $\phi$ is a stability condition such that the line bundles of bidegree $(d_1, d_2)$ are $\phi$-stable on curves in the boundary divisor $\Delta_{i,S}$, and either $d_1>i$ or $d_2>g-i$, then $\mathsf{W}_d(\phi)$ contains the pullback of $\Delta_{i,S}$ in $\Jb{d}{g,n}(\phi)$.

We will discuss more on this matter in Proposition~\ref{BNisclosure} and Remark~\ref{BNisnotexpected}.
\end{remark}

We conclude this Subsection by providing sufficient conditions for the Brill--Noether class \ref{thomporteous} defined by the Thom--Porteous Formula to coincide with the class of the Brill--Noether locus.  Some parts of the proof of the next two propositions will require to employ the fact that $\catJ$ is a stratification of $\Jmb{d}{g}{n}(\phi)$, which we will discuss in the next section. For this reason, we postpone the proof of the next result to Section~\ref{proofofpropositions}.

As in Subsection~\ref{subsec: stabilityH}, we fix coordinates for $V_{g,n}^d \cong C_{g,n}^d \times D_{g,n}^d$, and let $V_{g,n}^d \ni \phi=((x_{i,1,S})_{(i,S)}, (x_1, \ldots, x_n))$ where $x_j=x_{0,2,j}$ for each $1 \leq j \leq n$.

\begin{proposition} \label{BNisclosure}
We have that $\mathsf{W}_d(\phi)$ is the closure of $\mathsf{W}_d(\phi)_{|\J{d}{g,n}}$ (in particular it is reduced, irreducible and of the expected codimension, and $\mathsf{w}_d(\phi)=\mathsf{W}_d(\phi)$) if and only if $\phi=((x_{i,1,S})_{(i,S)}, (x_1, \ldots, x_n))$ is as follows
\begin{enumerate}
\item If $d=g-1$, for $i-3/2<x_{i,1,S}<i+1/2$ for all $(i,S)$.
\item If $d=g-2>0$, for $i-3/2<x_{i,1,S}<i-1/2$ for all $(i\geq 1,S)$ and $-3/2<x_{0,1,S}<1/2$ for all $S$, and 
\[i-2 <\frac{2g-2i-t}{2g-2} \cdot \sum_{j \in S} x_j + \Bigg(d-\sum_{j \notin S} x_j\Bigg) \cdot\frac{2i-2 +t}{2g-2}   <i+1 \] for all vine curve triples $(i,t,S)$ with $t \geq 2$.
\item Never if $0<d\leq g-3$.

\item If $d <0$, for $d-1/2<x_{i,1,S}<1/2$ for all $(i,S)$, and the coordinates $x_1, \ldots, x_n$ satisfy
\begin{equation}\label{condition} d-1 <\frac{2g-2i-t}{2g-2} \cdot \sum_{j \in S} x_j + \Bigg(d- \sum_{j \notin S} x_j\Bigg) \cdot \frac{2i-2 +t}{2g-2}   <1 \end{equation}
for all vine curve triples $(i,t,S)$ with $t \geq 2$.
\eqref{condition}.

\item If $d=0$, for $-1/2<x_{i,1,S}<1/2$ for all $(i \geq 1,S)$, for $-3/2<x_{0,1,S}<1/2$ for all $S$, and when  the coordinates $x_1, \ldots, x_n$ satisfy \eqref{condition} for all vine curve triples $(i,t,S)$ with $t \geq 2$.

\end{enumerate}
\end{proposition}
 We now explicitly define a stability condition that, in each of the cases (1),(2),(4),(5) listed above, belongs to the ranges that we have identified for $\mathsf{W}_d(\phi)$ to equal the closure of its restriction to the open part. (This shows, in particular, that these ranges are not empty).
\begin{definition} \label{stabilisedcan} For $G\in G_{g,n}$, define the stabilized-canonical divisor $K^s_G$ to equal zero at every vertex contained in some rational tail\footnote{A rational tail is a complete subgraph whose genus is $0$ and that is connected to its complement by exactly $1$ edge.}, and for every other $v$ to equal  $K^s(v)=2g(v)-2+\operatorname{val}'(v)$, where $\operatorname{val}'(v)$ is the number of edges at $v$ (counting each loop twice),  except the edges that are contained in some rational tail.

Then define the stabilized canonical element $\phi^d_{\text{scan}}(G)=\frac{d}{2g-2} \cdot K^s_G \in V_{g,n}^d$.
\end{definition}
Note that the above is different from the canonical stability $\phi^d_{\text{can}} \in V_{g,n}^d$ chosen as the origin in \cite{kp2}.

\subsubsection{Pullback via Abel--Jacobi sections} \label{Sec: AJ} Fix integers ${\bf d}=(k;d_1, \ldots, d_n)$ such that $d= d_1+ \ldots+ d_n$ and integers ${\bf f}=(f_{i,S})_{i,S}$ for every boundary divisor $G(g-i,1,S) \in G_{g,n}$.
Define  the universal line bundle
\begin{equation} \label{eq: aj}
\mathcal{L}=\mathcal{L}_{{\bf d}, {\bf f}}=\omega^k_{\overline{\mathcal{C}}_{g,n}/\Mmb gn}\left(\sum_{j=1}^n d_j x_j+ \sum_{i,S} f_{i,S^c} \cdot  C_{i,S^c} \right),
\end{equation}
where $C_{i,S^c} \subset \Cmb gn$ is the component\footnote{These unnatural conventions will simplify the formulas in Remark~\ref{PB-AJ} and Example~\ref{ex-AJ}.} over the boundary divisor $\Delta_{g-i,S}=\overline{\mathcal{M}}_{G(g-i,1,S)} \subset \Mmb gn$ that contains the sections in $S^c$. Then define $\phi= \phi_{{\bf d}, {\bf f}}\in V_{g,n}^d$ to be the multidegree of $\mathcal{L}$.

If $\phi^+=\phi^+_{{\bf d}, {\bf f}}$ is a nondegenerate small perturbation of $\phi$, then we have that the universal line bundle $\mathcal{L}$ is  $\phi^+$-stable, and it defines a (Abel--Jacobi) section \[\sigma=\sigma^+_{{\bf d}, {\bf f}} \colon \Mmb gn \to \Jmb dgn (\phi^+).\]

\begin{remark} \label{Rem: relinmbgn} Assume $d<0$, $k=0$, and ${\bf d}$ satisfies $d_i \leq 1$ for all $i$, and at most one of the $d_i$'s equals $1$. Assume that ${\bf f}$ satisfies $d \leq f_{i,S^c}+\sum_{j \in S} d_j \leq 0$ for all $(i,S)$. Then $\phi_{{\bf d}, {\bf f}}$ satisfies the conditions of Item~(4) in Proposition~\ref{BNisclosure}.

 By Proposition~\ref{BNisclosure} the pullback
\begin{equation} \label{Eq: rel} \sigma_{{\bf d}, {\bf f}}^*(\mathsf{w}_d)= 0 \quad \in \operatorname{A}^d(\Mmb gn),\end{equation}
gives a tautological relation. The LHS of \eqref{Eq: rel} can be explicitly written as a linear combination of standard generators of the tautological ring of $\Mmb gn$ by means of \cite[Theorem~1]{prvZ} (see also Corollary~3.7 and Equation~3.9 of \emph{loc.cit}).
\end{remark}

\section{Normal crossing stratification categories and blowups}
\label{section:nc and blowup}
In this section we define the axioms needed for a category of (resolved) strata of a space stratified by normal crossing divisors that are not necessarily simple normal crossing. We use this formalism to write some intersection theoretic formulas (the excess intersection formula and the Grothendieck-Riemann-Roch formula for the total Chern class) that we will use to derive our main result Theorem~\ref{maintheorem}. Then we define the blowup category at a stratum with transversal self-intersection. A construction of such strata categories starting from a normal crossing divisor, and more generally from a toroidal embedding, is given in \cite[Definition~3.5]{mmuv}.

The main examples we are generalizing are (a) the poset obtained by intersecting the components of a simple normal crossing divisor and (b) the stratification of $\Mmb{g}{n}$ by topological type, induced by the boundary divisors $\Delta= \Delta_{\text{irr}} \cup \bigcup_{i,S} \Delta_{i,S}$ (see Subsection~\ref{notation - moduli}). In the latter case, the relevant category is the category $G_{g,n}$ of stable $n$-pointed graphs of genus $g$ with morphisms given by graph  contractions. \ 

\subsection{Categories of resolved strata for a normal crossing stratification}
\label{subsec:normal_strat}

We start with the main defintion.

\begin{definition} \label{def: stratcat} Let $\cat$ be a finite skeletal category with a terminal object $\bullet$ such that every morphism is an epimorphism.   We say that such $\cat$ is a \emph{(normal crossing) stratification category} if its underlying poset is  ranked by some rank function $\codim$ whose minimum element is the terminal object $\bullet$, and with $\codim(\bullet)=0$ and if the following axiom is satisfied:
    \begin{axiom}
    \label{item:Sf} For each $f\colon \alpha\to \beta$ there exist exactly $\codim(f):=\codim(\alpha)-\codim(\beta)$ pairs \[(\beta', \Aut(\beta')g) \ \in \ \Obj(\cat)\times (\Aut(\beta')\backslash \Mor(\alpha,\beta'))\] such that, for each such pair, there exists $i\colon \beta'\to \beta$ with $\codim(i)=1$ and $f=i\circ g$. 
\end{axiom}
\end{definition}

\begin{remark} In Axiom~\ref{item:Sf} (a) the existence of such $i$ is independent of the choice of the representative $g$ in the left coset $\overline{g}:=\Aut(\beta')g$; and (b) since $g$ is an epimorphism, the morphism $i$ is necessarily unique.
\end{remark}

\begin{remark} If $\cat$ is any category with finite sets of morphisms and where every morphism is an epimorphism, then we have that $\Mor_{\cat}(\alpha,\alpha)=\Aut_{\cat}(\alpha)$. Indeed, if $f\in \Mor(\alpha,\alpha)$, then there exist natural numbers $a>b$ such that $f^a=f^b$, and since $f$ is an epimorphism we have that $f^{a-b}=\operatorname{Id}_A$, which proves that $f$ is an isomorphism.\par 
Moreover, if $\alpha$ and $\beta$ are distinct elements, we also have that if $\Mor(\alpha,\beta)\neq \emptyset$ then $\Mor(\beta,\alpha)=\emptyset$. Indeed, assume that there exists $f\colon \alpha\to \beta$ and $g\colon \beta\to \alpha$. By the observation we would have that both $f\circ g$ and $g\circ f$ are automorphisms, which implies that both $f$ and $g$ are isomorphisms, this contradicts the fact that $\cat$ is skeletal. This means that the set $\Obj(\cat)$ has a natural poset structure given by $\alpha\geq \beta$ if $\Mor(\alpha,\beta)\neq \emptyset$. \par
\end{remark}

When dealing with a stratification category $\cat$ we will use the following notation.
\begin{enumerate}
    \item We write $f_\alpha$ for the unique element of $\Mor(\alpha,\bullet)$.
    \item If $f_i\colon \alpha \to \beta_i$ are morphisms for $i=1,\ldots, m$, we define 
    \[
    \Aut(f_1,\ldots, f_m):=\{\tau \in \Aut(\alpha); \ f_i\circ \tau =f_i\text{ for every }i=1,\ldots, m\}. 
    \]
    Note that $\Aut(f_\alpha)=\Aut(\alpha)$.
    \item  For each morphism $f\colon \beta\to \gamma$ and  object $\alpha\in \Obj(\cat)$, we define $\oMor(\alpha,f):=\Aut(f)\backslash\Mor(\alpha,\beta)$. When $f=f_\beta$, we simply write $\oMor(\alpha,\beta):=\oMor(\alpha,f_\beta)=\Aut(\beta)\backslash \Mor(\alpha,\beta)$. 
    \item For each morphism $f\col \alpha\to \beta$, we let $S_f$ denote the set of all pairs $(\beta', \ol{g})$ satisfying the condition in Axiom~\eqref{item:Sf}. Moreover, for each $(\beta',\ol{g})\in S_f$ we denote by $i_{\overline{g},f}:=i\col \beta'\to \beta$ the morphism defined in Axiom~\ref{item:Sf}. We define $S_\alpha:=S_{f_\alpha}$.
\end{enumerate}

Here are the most relevant examples of stratification categories in this paper.

\begin{example} \label{ex: snc}
    (Simple normal crossing). Let $X$ be a nonsingular variety and $D= D_1+\ldots+D_k$ be a simple normal crossing divisor. To this we can associate a category $\cat$ whose objects are the strata and morphisms are the inclusions. This category $\cat$ is finite, skeletal, has a terminal element, every morphism is an epimorphism, it is ranked by codimension,  and it satisfies Axiom~\ref{item:Sf}.
\end{example}

The category $\cat$ is \emph{simple normal crossing} if, in addition to Axiom~\ref{item:Sf}, it satisfies:

  \begin{axiom}
  \label{axi:snc}
  For every $\alpha,\beta\in \Obj(\cat)$ the set $\oMor(\alpha,\beta)$ contains at most $1$ element.
    \end{axiom}

\begin{example} \label{cat: Mbgn} The second example is $\cat=G_{g,n}$ introduced in Subsection~\ref{notation - moduli}. The terminal object here is the trivial graph with $1$ vertex of genus $g$ carrying all the markings, and no edges. The rank function is the number of edges. The set $S_f$ of a morphism  $f \colon G \to G'$  is naturally identified with the set of edges of $G$ that are contracted by $f$. In particular, $S_{G}$ equals the edge set $E(G)$.

The rank~$1$ objects, the boundary divisors, are either graphs with two vertices connected by one edge (corresponding to the divisors $\Delta_{i,S}$, see \ref{notation - moduli}), or the graph consisting of $1$ vertex of genus $g-1$ with $1$ loop.
\end{example}

\begin{example} \label{cat: Jbgn} The  main example in this paper is the category $\mathfrak{C}=\catJ$ that we introduced in Definition~\ref{def: semistablecat}, an enhancement of the category $G_{g,n}$ discussed above. The terminal object 
 is the trivial graph endowed with the unique function that maps its unique vertex to the integer $d$. The rank of an object $(G,(E,D))$ equals $|\operatorname{Edges}(G)| + |E|$. For $f \colon (G', (E',D')) \to (G, (E, D))$, the set $S_f$ is naturally identified with the set of edges contracted by $f$.
 
 The rank $1$ objects are $(G,(E,D))$ with $G$ a rank $1$ object of $G_{g,n}$ and $(E,D)$ a $\phi$-stable pseudodivisor (which implies $E=\emptyset$). \end{example}

\begin{example}
    To a nonsingular variety $X$ (or DM stack) endowed with a normal crossing divisor $D$, \cite[Definition~3.5]{mmuv} associates a stratification category that respects Axiom~\ref{item:Sf} above.

    Note that the construction of \emph{loc.cit} in the case $X=\Mmb{g}{n}$ and $D= \Delta_{irr} + \sum_{(i,S)} \Delta_{i,S}$, which we discussed in Example~\ref{cat: Mbgn}, produces the quotient of the category $G_{g,n}$ of stable graphs where $2$ morphisms are identified whenever they are the same on the corresponding edge sets. (See \cite[Figure~2]{mmuv} for examples of automorphisms that are identified to the identity). A similar phenomenon happens for the category of Example~\ref{cat: Jbgn}.

    In this paper we prefer instead to work with the usual category of stable graphs (and its enhancements).
\end{example}

The next proposition is the analogue of the fact that the set of strata that contain a given stratum is in natural bijection with the subsets of the divisors that define that stratum.
\begin{proposition}
\label{prop:bijection_parts}
Given a morphism $f\colon \alpha\to \beta$ and $1 \leq k \leq \codim(\alpha) - \codim(\beta)$, there is a natural bijection between the set of pairs
\[
\{(\gamma, \ol{j})  \in \Obj(\cat) \times \oMor(\alpha, \gamma): \  \codim(\gamma)-\codim(\beta)=k \textrm{ and } \exists h\colon \gamma\to \beta, f=h\circ j\}
\]
(note that $h$ above is unique) and the set $\parts(k, S_f)$ of subsets of $S_f$ containing $k$ elements.
\end{proposition}

We start by observing the following:
\begin{remark} \label{Rem: inclusion}
Given a factorization $f\colon\alpha\xrightarrow{j} \gamma \xrightarrow{h}\beta$, there is a natural inclusion $j^*\colon S_{h}\hookrightarrow S_f$ given by $(\beta', \ol{g'})\mapsto (\beta', \ol{g'\circ j})$.

 Moreover, we claim that for each $\overline{j}\in \oMor(\alpha, h)$ we have a well-defined $\overline{j}^*(S_h)$.  Indeed, the map $j^*$ is the same as $(j \circ \tau)^*$ for every $\tau \in \Aut(\gamma \to \beta)$.
\end{remark}

\begin{proof}
We first observe that by Remark~\ref{Rem: inclusion}, there is a natural map $\lambda_{k,f}$ from the set of pairs, call it $X_{k,f}$, to the set $\mathcal{P}(k,S_f)$ of $k$-elements subsets of $S_f$, obtained by  $\lambda_{k,f}((\gamma,\overline{j})):= j^*(S_h)$. 

Then we prove that the cardinality of $X_{k,f}$ equals $\binom{\codim(\alpha)-\codim(\beta)}{k}$, which is also the cardinality of $\mathcal{P}(k,S_f)$. This is achieved by induction on $\codim(\alpha)-\codim(\beta)$ and double counting. For each $c \in S_f$, let $X_{k,f,c}$ be the subset of $X_{k,f}$ of elements whose image via $\lambda_f$ contains $c$. By induction hypothesis, we have that $|X_{k,f,c}|=\binom{\codim(\alpha)-\codim(\beta)-1}{k-1}$. By Axiom~\ref{item:Sf} the number of elements of $\{(a,b): \ a \in X_{k,f}, b \in \lambda_{k,f}(a)\}$ equals $k \cdot |X_{k,f}|$ and, it also equals $(\codim(\alpha)- \codim(\beta)) \cdot \binom{\codim(\alpha)-\codim(\beta)-1}{k-1}$. These two equalities prove that $|X_{k,f}|=\binom{\codim(\alpha)- \codim(\beta)}{k}$.

Finally we prove that each $\lambda_{k,f}$ is surjective. First we prove that this is the case for $k=\codim(\alpha)-\codim(\beta)-1$ (or equivalently when $\codim(\gamma)=\codim(\alpha)-1$). By the previous paragraph, $\lambda_{k,f}$ is, for this $k$, a function between sets of the same cardinality, so it is equivalent to prove that it is injective.  Let $a_1,a_2\in X_{k,f}$ be such that $\lambda_{k,f}(a_1)=\lambda_{k,f}(a_2)$ and let $c$ be the only element of $S_f \setminus \lambda_{k,f}(a_1)$. If $a_1 \neq a_2$, then $\lambda_{k,f}^{-1}(\{c\})$  contains at most $\codim(\alpha)-\codim(\beta)-2$ elements, but in the previous paragraph we have established that $|X_{k,f,c}|=\codim(\alpha)-\codim(\beta)-1$; this contradicts the assumption $a_1 \neq a_2$.

To prove surjectivity of each $\lambda_{k,f}$ we argue by induction on $\codim(\alpha)-\codim(\beta)$. Let $S \in \mathcal{P}(k,S_f)$. Choose $T \supset S$ with $T \in \mathcal{P}(\codim(\alpha)-\codim(\beta)-1,S_f)$. By the previous paragraph, there exists $d=(\delta, \ol{h})$ such that $\lambda_{k,f}(d)=T$ and a factorization of $f=h \circ g$ through $\delta$, so $T=g^*(T')$ and $S= g^*(S')$ for some $S' \subset T'$ subsets of $S_h$. We have  $\codim(\delta)-\codim(\beta)=\codim(\alpha)-\codim(\beta)-1$. By applying the induction hypothesis to $\lambda_{k,h}$, we find $c \in X_{k, h}$ such that $\lambda_{k,h}(c)=S'$, and so $\lambda_{k,f}(g^*c)=S$. This concludes the proof of surjectivity.
\end{proof}

We will now define some important geometric notions in the stratification category.

\begin{definition} \label{def: generic}
Fix $f_i\colon \alpha_i\to \beta$  for $i=1, \ldots, m$, and $f\colon \gamma \to \beta$. Let  $g_i\colon \gamma\to \alpha_i$ for $i=1,\ldots, m$ be a collection of morphisms such that $f_i\circ g_i=f$.\par 
We say that the collection $(g_i)$ is \emph{generic} with respect to the tuple $(f, (f_i))$ if  \[S_f=\bigcup_{i=1}^m g_i^*(S_{f_i}).\]\par We say that the collection $(f,(f_i))$ is \emph{transversal} at $(g_i)$ if \[g_i^*(S_{f_i})\cap g_j^*(S_{f_j}) = f^*(S_\beta)\] for every $i\neq j$.
\end{definition}

Following the above definition, we write $\Int(f_1,\ldots, f_m)_f$ to denote the set of all generic tuples $(g_1,\ldots, g_m)$.

\begin{remark} Fix the same data of the above definition. Let $(\tau_1,\ldots, \tau_m)$ be a tuple in $\prod \Aut(f_i)$ and let  $(g_1,\ldots, g_m)$ be a generic collection, then  $(\tau_1\circ g_1, \ldots, \tau_m \circ g_m)$ is also generic. A similar result holds for an automorphism $\tau\in \Aut(f)$. That is: $(g_i)$ is generic if and only if $(g_i\circ \tau)$ is generic. This gives a natural left action of $\prod\Aut(f_i)$ and  right action of $\Aut(f)$ on $\Int((f_i))$.
\end{remark}
Following the remark, we define \[\lInt(f_1,\ldots, f_m)_f:=\prod \Aut(f_i) \backslash \Int(f_1,\ldots, f_m)_f\] and \[\rInt(f_1,\ldots, f_m)_f:= \Int(f_1,\ldots, f_m)_f/\Aut(f).\] 
Elements of $\lInt(f_1,\ldots,f_m)$ will be denoted by $(\ol{g}_1, \ldots, \ol{g}_m)$, while the elements of $\rInt(f_1,\ldots,f_m)$ will be denoted by $(g_1,\ldots, g_m) \Aut(f)$.

When $f_1=\ldots=f_m=f'$, we write $\sInt((f')^m)_{f}$ to denote the set of sets (not tuples) $\{\ol{g}_1,\ldots, \ol{g}_m\}$ (here the $g_i$ must be pairwise distinct) such that $(g_1,\ldots, g_m)\in \lInt((f')^m)_f$.

\subsection{Normal crossing stratifications}
We start the subsection with its main definition.
 
 \begin{definition} \label{def: ncs} We say that a stratificaton category $\cat$ as in Definition~\ref{def: stratcat} is the category of strata of a nonsingular DM-stack $X_\bullet$ if there exists a functor 

\begin{align*}
    \cat&\to \text{nonsingular DM-stacks} \\
    \alpha&\mapsto X_\alpha \\
    f\colon\alpha\to \beta&\mapsto X_f\colon X_\alpha\to X_\beta
\end{align*}
such that
\begin{enumerate}
    \item The morphisms $X_f\col X_\alpha\to X_\beta$ are proper and local complete intersection of codimension $\codim(f)$.
    \item The quotient stack $\left[\frac{X_\alpha}{\Aut(f)}\right]$ is the normalization of the image of $X_f$.
    \item The normal bundle $N_{f}$ of $X_f$ can be written as $N_f=\bigoplus_{e\in S_f}\LL_e$, where, for a pair $e=(\beta',\ol{g})\in S_f$, we define $\LL_e:=g^*(N_{i_{g,f}})$.

\item If $f_i\colon \alpha_i\to \beta$ for $i=1,2$ are two morphisms, then the following diagram
\[
\begin{tikzcd}
\displaystyle\bigsqcup_{\substack{f\colon \gamma\to \beta \\ (g_1,g_2)\Aut(f)\in \rInt(f_1,f_2)_f}} [X_{\gamma}/\Aut(g_1,g_2)] \ar["X_{g_1}"]{r} \ar["X_{g_2}"]{d} & X_{\alpha_1} \ar["X_{f_1}"]{d}\\
X_{\alpha_2} \ar["X_{f_2}"]{r}& X_\beta\\
\end{tikzcd}
\]
is a fiber diagram.
\end{enumerate}
\end{definition}
From now on we will abuse the notation and, for $f \in \Mor(\alpha, \beta)$, we simply write $f \colon X_{\alpha} \to X_{\beta}$ in place of $X_f \colon X_{\alpha} \to X_{\beta}$.

\begin{notation}
\label{not:prime_image}
We will use a prime to denote the image of a morphism $f\col X_\alpha \to X_\beta$. In other words, $X'_f:=\im(f) \subseteq X_\beta$ and, in particular, $X'_\alpha:=\im(f_\alpha)\subseteq X_\bullet$.    

We also say that the objects $X_{\alpha}'$ are the embedded strata and the objects $X_{\alpha}$ are the (resolved) strata.
\end{notation}

The two main examples in this paper are that of $\Mmb{g}{n}$ and that of $\Jmb{d}{g}{n}(\phi)$:

\begin{example} \label{Ex: stratM} The category $G_{g,n}$ is a category of strata of the nonsingular DM-stack $\Mmb{g}{n}$. If $G \in G_{g,n}$, we write $\overline{\mathcal{M}}_G$ for the corresponding stratum and $\overline{\mathcal{M}}'_G$ for its image in $\Mmb{g}{n}$  (see Subsection~\ref{notation - moduli}, \cite[Chapter~XII.10]{gac}). \end{example}

\begin{example}  \label{Ex: stratJ} The category $\catJ$ is a category of strata of the nonsingular DM-stack $\Jmb{d}{g}{n}(\phi)$ (\cite[Section~3]{mmuv}). If $(G,(E,D)) \in \catJ$, we write $\mathcal{J}_{G,(E,D)}$ for\footnote{Each such stratum also depends on $\phi$, but we do not include this in the notation to not overburden it.} the corresponding stratum and $\mathcal{J}_{G,(E,D)}'$ for its image in $\Jmb dgn(\phi)$.

 If $G$ is a vine graph (Definition~\ref{vinecurvenotation}) and $D$ is a divisor on $D$, by abuse of notation we will also refer to $(G,D)$ as a vine graph, and to $\mathcal{J}_{G,(E,D)}$ as a vine stratum, as in the case of $\overline{\mathcal{M}}_{g,n}$.
\end{example}

The point made in Example~\ref{Ex: stratJ} allows us to complete the proof of Proposition~\ref{BNisclosure}. The next subsection is devoted to completing that proof.

\subsubsection{Proof of Proposition~\ref{BNisclosure}}
\label{proofofpropositions}

\begin{proof}


\hfill \begin{enumerate}

\item If $d=g-1$, the result of \ref{BNisclosure} follows from \cite[Theorem~4.1]{kp2}. 

\item Assume that $d=g-2$. In order to reach our conclusion, we  prove that $\phi$ is  in the claimed range if and only if the intersection of $\mathsf{W}_d(\phi)$ with the boundary of $\Jmb{d}{g}{n}(\phi)$ has codimension larger than the expected codimension $2$. Also, the strata that generically parameterize curves whose irreducible components are singular can be excluded, because the existence of a nonzero global section is equivalent if those components are smoothened.

Firstly, we analyze the boundary divisors, which have the form $\mathcal{J}_{(G(i,1,S),D)}$. The range of $\phi$ in the claim is equivalent to constraining the divisor $D$ to equal $(i-1, g-i-1)$. It is straightforward to verify that the locus cut out in $\mathcal{J}_{(G(i,1,S),D)}$ by the condition of admitting a global section has codimension at least $2$ in $\mathcal{J}_{(G(i,1,S),D)}$, hence it has codimension at least $3$ in $\Jmb{d}{g}{n}(\phi)$. 

If, on the other hand, the divisor $D$ is of the form $(k-1, g-k-1)$ for $k \neq i$, then either $\mathcal{J}_{(G(i,1,S),D)}$ is contained in  $\mathsf{W}_d(\phi)$, or their intersection has codimension $1$ in $\mathcal{J}_{(G(i,1,S),D)}$. In both cases, their  intersection has codimension smaller than or equal to $2$ in $\Jmb{d}{g}{n}(\phi)$.

Then we analyze the codimension $2$ strata $\mathcal{J}_{G,D}$. If $G$ is a tree, the stability condition is uniquely determined by the stability condition on the boundary divisors, and so is the stable degree $D$ -- the problem has been resolved in the previous paragraph. We assume therefore that $G=G(i, 2, S )$ is a vine graph with $2$ nodes. Using the change of coordinates \eqref{changeofcoord}, the range identified in our statement is equivalent to requesting that the stable divisor $D$ on $G(i,2,S)$ equals one of $(i-2, g-i)$, $(i-1,g-i-1)$, $(i,g-i-2)$ or $(i+1, g-i-3)$. In all these cases, one can check that the generic element of $\mathcal{J}_{G,D}$ does not admit a global section. Conversely, if $D$ is not one of those $4$ cases, the stratum $\mathcal{J}_{G,D}$ is contained in $\mathsf{W}_d(\phi)$. This concludes our proof.

\item Assume that $1 \leq d \leq g-3$. In order to reach our conclusion, it is enough to prove that for every $\phi$, the intersection of $\mathsf{W}_d(\phi)$ with some boundary divisor has codimension smaller than or equal to the expected codimension $g-d$.

We take $i = \lfloor \frac{g}{2}\rfloor$ and pick any $S \subseteq [n]$, and show that the intersection of $\mathsf{W}_d(\phi)$ with the preimage of $\Delta_{i,S}$ in  $\Jmb{d}{g}{n}(\phi)$ contains a locus of codimension  smaller than or equal to $g-d$ in $\Jmb{d}{g}{n}(\phi)$ .

The stable bidegree $D$ such that the intersection $\mathcal{J}_{G(i,1,S),D} \cap \mathsf{W}_d(\phi)$ has  largest codimension is $D=(\frac{d}{2},\frac{d}{2})$ for $d$ even (resp $D=(\frac{d-1}{2}, \frac{d+1}{2})$ for $d$ odd). The intersection has codimension $\lceil \frac{g-d-1}{2}\rceil+2$ in $\Jmb{d}{g}{n}(\phi)$ and, for $d \leq g-3$, this number is smaller than or equal to  $g-d$.

\item[(4-5)] Assume that $d\leq 0$.

Assume first that $\phi$ is not in the given range. Then, arguing as in the case $d=g-2$ above, one can check that in some boundary divisor or in some codimension $2$ vine stratum (depending on which inequality $\phi$ fails to satisfy) the intersection with $\mathsf{W}_d(\phi)$ has codimension larger than the expected one (which is $g$ for $d=0$, and by which we mean that the locus is not empty when $d<0$).

Assume now that $\phi$ is in the given range. We will use the following result:
 \begin{proposition} \label{cord<0} If $ \phi \in V_{g,n}^d$ is nondegenerate and such that the inequality 
 \begin{equation} \label{closetocan}
\phi_{C_0} \leq \frac{\left|C_0 \cap \overline{C_0^c}\right|}{2}.
\end{equation}
holds for all $(C,p_1,\ldots, p_n) \in \Mmb{g}{n}$ and for all subcurves $C_0 \subseteq C$, then

\begin{enumerate}
    \item if $d=0$, then $F \in \mathsf{W}_d(\phi)$ if and only $F$ is the trivial line bundle;
    \item if $d<0$, then $\mathsf{W}_d(\phi)=\emptyset$.
    \end{enumerate}

    \end{proposition}
\begin{proof}
Part~(a) is \cite[Lemma~8, Lemma~9]{HKP}. Part~(b) follows from Lemma~\ref{lemmad<0} below.
    
\end{proof}
By applying Proposition~\ref{cord<0}, and observing that both $\phi$ and the multidegree $D$ of line bundles are stable for graph morphisms, and arguing as in the proof of Proposition~\ref{stabilityspace}, we conclude that Inequality~\eqref{closetocan} is satisfied for all curves $C$ and subcurves $C_0$ if and only if it is satisfied for all vine curves $C$ (and taking $C_0$ to be one of its irreducible components). After applying the change of coordinates \eqref{changeofcoord}, this is equivalent to the given range.  

The only remaining case to consider is when $d=0$ and $-3/2<x_{0,1,S}<-1/2$ for some $S$. In that case, the intersection of the component $\mathcal{J}_{G(0,1,S),(-1,1)}$ with $\mathsf{W}_d(\phi)$ has codimension $g+1$ in $\Jmb{d}{g}{n}(\phi)$, hence the intersection is in the closure of the restriction of $\mathsf{W}_d(\phi)$ to the open part.


\end{enumerate}
\end{proof}


In the proof of Proposition~\ref{cord<0} we used the following. 
\begin{lemma} \label{lemmad<0} Assume $d<0$. Let $C$ be a nodal curve, and let $\phi \in V^d_{\textnormal{stab}}(C)$ be such that Inequality~\eqref{closetocan}
holds for all subcurves $C_0 \subseteq C$. Then every $\phi$-stable rank-$1$ torsion-free sheaf $F$ on $C$ satisfies $H^0(C,F)=0$.
\end{lemma}

The lemma generalizes to the case $d<0$ the argument given in \cite[Lemma~3.1]{dudin} and \cite[Lemma~8]{HKP}.
\begin{proof}
Let $F$ be one such sheaf. If $F$ is $\phi$-stable, the inequality
\begin{equation}
\deg_{C_0}(F)< \frac{\left|C_0 \cap \overline{C_0^c}\right|}{2} - \delta_{C_0}(F)+ \phi_{C_0}
\end{equation}
holds for all subcurves $\emptyset \neq C_0 \subsetneqq C$.
The latter, combined with \eqref{closetocan}, implies the inequality
\begin{equation} \label{section-contrad}
\deg_{C_0}(F)< \left|C_0 \cap \overline{C_0^c}\right| - \delta_{C_0}(F)
\end{equation}
for all subcurves $\emptyset \neq C_0 \subsetneqq C$.

The fact that the latter inequality holds on all subcurves $C_0$ implies that $F$ admits no nonzero global sections. Indeed, if such a section $s$ existed denote by $C'$ its support. Note that $C'\neq C$ because the degree of $F$ is negative. Hence, $C'\neq C$ and we have the inequality
\[
\deg_{C'}(F)\geq \left|C' \cap \overline{C'^c}\right| - \delta_{C'}(F),
\]
contradicting \eqref{section-contrad}. 
\end{proof}


We conclude this interlude by observing that for all degrees ``in the middle'', the Brill--Noether cycle cannot be of the expected codimension.

\begin{remark} \label{BNisnotexpected}
Assume that $1 \leq d \leq g-5$. We claim that there exist no $\phi$ such that $\mathsf{W}_d(\phi)$ has the expected codimension $g-d$. To show this, we argue in a very similar way to the case $1 \leq d \leq g-3$ of the proof of Proposition~\ref{BNisclosure}. We let $i = \lfloor \frac{g}{2}\rfloor$ and pick any $S \subseteq [n]$. In the same way as discussed in loc.cit., for $1 \leq d \leq g-5$, the intersection of $\mathsf{W}_d(\phi)$ with the preimage of $\Delta_{i,S}$ in $\Jmb{d}{g}{n}(\phi)$ contains a locus of codimension strictly smaller than  $g-d$  in $\Jmb{d}{g}{n}(\phi)$, and this proves our claim.

\end{remark}

\subsection{Intersection theory formulas}
In this subsection we enunciate and prove some results concerning the intersection theory of this stratification. From now on in this subsection we fix $X_\bullet$ and its stratification functor. 
 
 Since most of our computations are done using Chern classes, we will abuse the notation as we now explain. 
   Let $p_1, p_2, q$ be polynomials in variables $x_{i,j}$ such that $p_1=qp_2$. Assume that $L_i$ are elements in the $K$-theory of $X$, and that $A=p_1(c_j(L_i))$ and $B=p_2(c_j(L_i))$ are formal polynomials in the Chern classes of $L_i$. We will write $\frac{A}{B}$ to mean the class $q(c_j(L_i))\cap [X]$ in the Chow group of $X$.
   
More generally, we will write $\frac{A}{B}\in A^*(X)$ to mean that there exists polynomials $p_1, p_2, q$ and $K$-theory elements $L_i$ satisfying the conditions in the previous paragraph.

The main motivation for this is \cite[]{fulton}, which states that, for a vector bundle $N$
\[
\frac{c(L\otimes \wedge^\bullet N)}{c_{\rk{N}}(N)}
\]
is a polynomial in the Chern classes of $L$ and of $N$.

In this language, we have the following excess intersection formula.


\begin{proposition}
\label{prop:intersection}
Let $f_i\colon \alpha_i\to \beta$ for $i=1,2$ be two morphisms in $\cat$ and  fix classes $A_i/c_{\codim(f_i)}(N_{f_i})\in A^*(X_{\alpha_i})$, then
\[
f_{1*}\Big(\frac{A_1}{c_{\codim(f_1)}(N_{f_1})}\Big) f_{2*}\big(\frac{A_2}{c_{\codim(f_2)}(N_{f_2})}\Big) = \sum_{\substack{f\colon \gamma\to \beta \\ (g_1,g_2)\Aut(f)\in \\ \rInt(f_1,f_2)_f}}\frac{f_{*}}{|\Aut(g_1, g_2)|}\Big(\frac{g_1^*A_1g_2^*A_2}{c_{\codim(f)}(N_f)}\Big)
\]
where $g_i=1,2$ are the base change morphisms (as in Definition~\ref{def: ncs}, Item~(4)).
\end{proposition}
\begin{proof}
This follows directly from Item~(4) in Definition~\ref{def: ncs} and from the excess intersection formula (see \cite[Proposition~17.4.1]{fulton}).
\end{proof}
We will also be using the following corollary of the above formula.

\begin{corollary}
\label{cor:intersection}
Let $f_i\colon \alpha_i\to \beta$ be two morphisms in $\cat$ and let \[\frac{A_i}{c_{\codim(f_i)}(N_{f_i})}\in A^*(X_{\alpha_i})\] be such that $A_i$ is invariant under $\Aut(f_i)$, then the following holds
\begin{multline*}
\frac{f_{1*}(\frac{A_1}{c_{\codim(f_1)}(N_{f_1})})}{|\Aut(f_1)|}\frac{f_{2*}(\frac{A_2}{c_{\codim(f_2)}(N_{f_2})})}{|\Aut(f_2)|} = \sum_{f\colon \gamma \to \beta}\frac{f_*}{|\Aut(f)|}\Big(\sum_{(\ol{g}_1,\ol{g}_2)\in \lInt(f_1, f_2)_f } \frac{g_1^*A_1g_2^*A_2}{c_{\codim(f)}(N_{f})}\Big)
\end{multline*}
\end{corollary}
\begin{proof}
We  expand the formula in Proposition~\ref{prop:intersection} to obtain
\[
f_{1*}\Big(\frac{A_1}{c_{\codim(f_1)}(N_{f_1})}\Big)f_{2*}\big(\frac{A_2}{c_{\codim(f_2)}(N_{f_2})}\Big) = \sum_{\substack{f\colon \gamma\to \beta \\(g_1,g_2)\in \Int(f_1,f_2)_f}}\frac{f_{*}}{|\Aut(f)|}\Big(\frac{g_1^*A_1g_2^*A_2}{c_{\codim(f)}(N_{f})}\Big)
\]
because $|(f\colon \gamma\to \beta,(g_1,g_2)\Aut(f))| = |\Aut(f)|/|\Aut(g_1,g_2)|$. From there, we have that 
\begin{multline*}
f_{1*}\Big(\frac{A_1}{c_{\codim(f_1)}(N_{f_1})}\Big)f_{2*}\Big(\frac{A_2}{c_{\codim(f_2)}(N_{f_2})}\Big) = \\ = \sum_{\substack{f\colon \gamma\to \beta \\ (\overline{g}_1,\overline{g}_2)\in \lInt(f_1,f_2)_f}}\frac{|\Aut(f_1)||\Aut(f_2)|}{|\Aut(f)|}f_{*}\Big(\frac{g_1^*A_1g_2^*A_2}{c_{\codim(f)}(N_{f})}\Big)
\end{multline*}
and  the result follows.
\end{proof}
Next, we apply the above to obtain a self-intersection formula.
\begin{corollary}
Let $f\colon \alpha\to \beta$, then
\[
\bigg(\frac{f_*}{|\Aut(f)|}\Big(\frac{A}{c_{\codim(f)}(N_f)}\Big)\bigg)^k = \sum_{f'\colon \gamma \to \beta}\frac{f'_*}{|\Aut(f')|}\Bigg(\sum_{\substack{(\ol{g}_1,\ldots, \ol{g}_k)\in\\ \lInt(f,\ldots, f)_f}}\frac{\prod_{j=1}^kg_i^*(A)}{c_{\codim(f')}(N_{f'})} \Bigg)
\]
\end{corollary}
The latter will be used to prove the following Grothendieck-Riemann-Roch formula for the total Chern class (deduced from the usual one, involving the Chern character).
\begin{proposition}[Grothendieck-Riemann-Roch for the total Chern class]
Let $f\colon \alpha\to \beta$ be a morphism and let  $\F$ be an element   in the $K$-theory of $   X_\alpha$ with rational coefficients. Then
\[
c(\frac{f_*(\F)}{|\Aut(f)|}) = 1+\sum_{\substack{m\geq 1 \\ f'\colon \gamma\to \beta}}\frac{f'_*}{|\Aut(f')|}\bigg(\sum_{\substack{\{\ol{g}_1,\ldots, \ol{g}_m\}\in \\ \sInt( (f)^m)_{f'} }} \frac{\prod_{j=1}^m g_i^*(c(\bigwedge^\bullet N_f^\vee \otimes \F) - 1)}{c_{\codim(f')}(N_{f'})}\bigg)
\]
\label{prop: Grothendieck-Riemann-Roch}
\end{proposition} 
(This is inspired by \cite[Theorem~15.3]{fulton}).
\begin{proof}
We begin with the usual Grothendieck-Riemann-Roch formula
\[
\ch\big(\frac{f_*(\F)}{|\Aut(f)|}\big) = f_*\big(\frac{\ch(\F)}{|\Aut(f)|}\td(N_f)^{-1}\big),
\]
which, combining with the formula for the Todd class,  implies
\[
\ch_ n\big(\frac{f_*(\F)}{|\Aut(f)|}\big) = \frac{f_*}{|\Aut(f)|}\bigg(\frac{\ch_ n(\bigwedge^\bullet N_f^\vee\otimes \F)}{c_{\codim(f)}(N_f)}\bigg).
\]
By the inversion formula to express the total Chern class in terms of the Chern character (see e.g. \cite[Equation~3.9]{prvZ}), we deduce

\[
c\big(\frac{f_*(\F)}{|\Aut(f)|}\big)= \exp\big(\frac{f_*}{|\Aut(f)|}\bigg(\frac{\sum_{n \geq 1} \frac{(-1)^{n-1}}{n} \ch_n(\bigwedge^\bullet N_f^\vee \otimes \F)}{c_{\codim(f)}(N_f)}\bigg) \big).
\]
Setting $A:=\sum_{n \geq 1} \frac{(-1)^{n-1}}{n} \ch_n(\bigwedge^\bullet N_f^\vee \otimes \F)$, we will then compute
\[
\star : =\exp\bigg(\frac{f_*}{|\Aut(f)|}\Big(\frac{A}{c_{\codim(f)}(N_f)}\Big)\bigg)
\]

\begin{align*}
\star & =1 +\sum_{k\geq 1} \bigg(\frac{f_*}{|\Aut(f)|}\Big(\frac{A}{c_{\codim(f)}(N_f)}\Big)\bigg)^k\frac{1}{k!}\\
 &=1 + \sum_{k\geq 1}\sum_{f'\colon \gamma\to \beta}\frac{f'_*}{|\Aut(f')|}\Big(\sum_{(\ol{g}_1,\ldots, \ol{g}_k)\in \lInt((f)^k)_{f'}}\frac{\prod_{i=1}^kg_i^*(A)}{c_{\codim(f')}(N_{f'})}\cdot\frac{1}{k!}  \Big)\\
 & = 1 + \sum_{f'\colon \gamma\to \beta}\frac{f'_*}{|\Aut(f')|}\Big(\sum_{\substack{m\geq 1 \\ \{\ol{g}_1,\ldots, \ol{g}_m\}\in \sInt((f)^m)_{f'}\\ k_1,\ldots, k_m\geq 1}} \frac{\prod_{i=1}^mg_i^*(A)^{k_i}}{c_{\codim(f')}(N_{f'})}\cdot\frac{1}{(\sum_{i=1}^mk_i)!}  \Big)\\
 &= 1 + \sum_{f'\colon \gamma\to \beta}\frac{f'_*}{|\Aut(f')|}\Big(\sum_{\substack{m\geq 1 \\ \{\ol{g}_1,\ldots, \ol{g}_m\}\in \sInt((f)^m)_{f'}}}\frac{\prod_{i=1}^m(\exp(g_i^*(A))- 1)}{c_{\codim(f')}(N_{f'})}  \Big)\\ 
 \end{align*}
 
The claim is then obtained by applying again the inversion formula in the form
 \[
 \exp(A) = c(\bigwedge^\bullet N_f\otimes F).
 \]
\end{proof}

\subsection{Blowup}  \label{Sec: blowup}

Starting from a stratification category $\cat$ (see Definition~\ref{def: stratcat}), here we define the blowup category at a stratum with transversal self-intersection. Then for a fixed stratification functor $X$, we interpret the blowup category as the stratification of the blowup of the nonsingular DM-stack $X$ at that stratum.

\begin{definition}
\label{def:transversal_self}
We say that an object  $\delta\in \Obj(\cat)$ \emph{has transversal self-intersection} if for every pair $g_1, g_2\colon \gamma \to \delta$, the sets  $g_1^*(S_{\delta})$, $g_2^*(S_{\delta})$ are  either equal or disjoint.

\end{definition}

\begin{remark}
If $g_1^*(S_\delta)=g_2^*(S_\delta)$, then $\overline{g}_1=\overline{g}_2\in \oMor(\gamma,\delta)$. See Proposition~\ref{prop:bijection_parts}.
\end{remark}
 
\begin{example} \label{notransversalself} Let $\cat=G_{2k+1,1}$ for some $k\geq 1$. We claim that the vine  graph $G= G(k,2,\{1\})$ does not have transversal self-intersection. 

Indeed, let $G'$ be the ``triangle'' graph with $2$ vertices of genus $k$ and a third vertex of genus $0$ carrying the marking. There are two different morphisms $g_1,g_2 \colon G' \to G$ and $g_1^*(S_G) \cap g_2^*(S_G)$ consists of $1$ edge.
\end{example}

\begin{definition} \label{def: blowup cat} Let $\delta$ be an object of $\cat$ with transversal self-intersection. We define \emph{the blowup category $\bl_\delta\cat$ of $\cat$ at $\delta$} as follows. 

First consider the following category. Its set of objects consists of pairs  $(\gamma, \mf) $ where $\gamma$ is an object of $\cat$ and $\mf$ is a function  $\oMor(\gamma,\delta)\to \parts(S_{\gamma})$ such that \begin{equation} \label{eq:mf_empty} \emptyset \neq \mf(\overline{g})\subseteq g^*(S_\delta)\ \text{  for every  }\  \overline{g}\in \oMor(\gamma,\delta).\end{equation} 

Its morphisms $(\gamma_1,\mf_1)\to (\gamma_2,\mf_2)$ are morphisms $f\colon \gamma_1\to \gamma_2$ such that for every $\overline{g}_1\in \oMor(\gamma_1,\delta)$ we have that one of the conditions hold
\begin{enumerate}
\item there exists $\overline{g}_2\in \oMor(\gamma_2,\delta)$ such that $\overline{g}_1=\overline{g_2\circ f}$ and  $\mf_1(\overline{g}_1) \subseteq f^*(\mf_2(\overline{g}_2))$,
\item or $\mf_1(\overline{g}_1)\cap f^*(S_{\gamma_2})=\emptyset$.
\end{enumerate}
We then define $\bl_\delta\cat$ as a skeleton of the above category.
\end{definition}

\begin{proposition}
The category $\bl_\delta\cat$ is naturally ranked, and it satisfies  Axiom~(1) from Definition~\ref{def: stratcat}
.
\end{proposition}
\begin{proof}
Straightforward.
\end{proof}

 \begin{remark} \label{Rem: blowup} The rank of $(\gamma,\mf)$ is 
    \[
    \rk(\gamma)-\sum_{\overline{g}\in\oMor(\gamma,\delta)}|\mf(\overline{g})|.
    \]

    Moreover, the set $S_{(\gamma,\mf)}$ (the codimension~$1$ strata that contain a fixed stratum $(\gamma,\mf)$) is naturally identified with \[S_{\gamma}\setminus\Bigg(\bigcup_{\overline{g}\in \oMor(\gamma,\delta)}\mf(\overline{g})\Bigg) \cup \oMor(\gamma,\delta).\]
\end{remark}


Recall Notation~\ref{not:prime_image}. We define $h\col \widetilde{X}_\beta\to X_\beta$ to be the blowup of $X_\beta$ at the union of the images $X'_{g_1} \subseteq X_{\beta}$ for every $g_1\colon \gamma\to \beta$ such that there exists $g_2\col \gamma \to \delta$ satisfying  $(g_1,g_2)\in \Int(f_\beta, f_\delta)_{f_{\gamma}}$. We define $X_{\beta, \mf}$ to be  
\[
\prod_{\overline{g}\in\oMor(\gamma,\delta)} \mathbb{P}\Big(\bigoplus_{e\in \mf(\overline{g}) }\LL_e\Big).
\]

\begin{proposition} \label{prop :blowupstrat}
The functor 
\begin{align*}
\bl_\delta \cat &\to \textnormal{nonsingular DM stacks }\\ 
 (\gamma,\mf) & \mapsto X_{\gamma,\mf}
\end{align*}
is a stratification of $\bl_{X'_\delta}X_\bullet$.
\end{proposition}
\begin{proof}
This follows from \cite[Section~4.5]{mps} (see also \cite[Theorem 6, p.90]{kkms}) where the nonsingular DM stack is constructed as the star subdivision of the cone stack associated to the stratification.
\end{proof}

\begin{remark}
    When there  exists no morphism $\gamma \to \delta$, there exists a unique $\mf$  such that the pair $(\gamma, \mf) \in \bl_{\delta} \cat$.
    The latter is the stratum that corresponds to the strict transform of the image $X_{\gamma}'\subset X_\bullet$. 
\end{remark}

\begin{remark}
\label{rem:lifting}
   Suppose that $\delta$ is a stratum with transversal self intersection and $f\col \gamma \to \beta$ is a morphism such that $\Mor(\beta,\delta)=\emptyset$. 
   Let $(\gamma,\mf)$ be an object in $\bl_\delta\cat$. Then the morphism $f$ lifts to a morphism $(\gamma, \mf) \to (\beta, \emptyset)$ in $\bl_{\delta} \cat$ if and only if \[f^*S_{\beta}\cap \bigcup_{\overline{g}\in\oMor(\gamma,\delta)}\mf(g)  = \emptyset.\] (That is, when  $X'_{\gamma, \mf}$ is contained in the strict transform of $X'_{\beta}$ in $\bl_{X'_{\delta}} X_{\bullet}$).
\end{remark}

In this paper, the main example of the above construction is going to be the case where $\mathfrak{C}$ is the category $\catJ$ of Example~\ref{Ex: stratJ}, or a blowup of the latter. We now describe the example of $1$ blowup of $\catJ$ at one of the centers that will be relevant for our main result.

\begin{example} \label{oneblowup} Let $\phi \in V_{g,n}^d$ and $(G,D) \in \catJ$ be the lift of a vine graph $G=G(i,t,S)$ by some $\phi$-stable divisor $D$.

Morphisms $f \colon (G', (E',D')) \to (G,D)$  correspond to subsets $T_f \subset V(G'^{E'})$ such that the complete subgraphs $G(T_f), G(T_f^c)$ in $G'^{E'}$ are connected and of genus $i, g-i-t+1$,  the markings $S$ are on $G(T_f)$ and the markings $S^c$ are on $G(T_f^c)$, and $D'(G(T_f))=D(v_1)$ and $D'(G(T_f^c))=D(v_2)$ (for $v_1,v_2$ the two vertices of $G$). We let $E(T_f) \subseteq E(G'^{E'})$ be the subset of $t$ edges that separate $G(T_f)$ from $G(T_f^c)$.

Assume that $(G,D)$ is a stratum\footnote{Because of Lemma~\ref{lem: noE}, in this paper we will never need to blowup any strata of the form $(G,E,D)$ with $E \neq \emptyset$.} with transversal self-intersection. The category $\bl_{(G,D)} \catJ$ defined above stratifies the blowup $\bl_{{\mathcal{J}'}_{(G,D)}}\Jmb{d}{g}{n}(\phi)$. We can describe more explicitly its objects as tuples $(G',E', D', \alpha)$ such that $(G',(E', D')) \in \operatorname{Obj}(\catJ)$ and $\alpha$ is a choice, for each morphism  $f \colon (G',D') \to (G,D)$ (up to automorphisms of $(G,D)$), of a subset $\emptyset \neq \alpha(\Aut(G,D)f) \subseteq E(T_f)$.

\end{example}

We now define the psi-classes associated to a given stratum $\gamma \in \cat$. Recall that each $e \in S_\gamma$ corresponds to a morphism $j_e \colon \gamma \to \beta_e$ where the latter has rank~$1$.  Then define the psi-classes 
\begin{equation} \label{def: psi}
    \Psi_{\beta_e} := -c_1(\mathbb{L}_e)=- c_1 (N_{X_{\beta_e}} X_{\bullet}), \quad \psi_{\gamma, e}:= j_e^* \Psi_{\beta_e}
\end{equation}
(see Item~(3) of Definition~\ref{def: ncs}) for $\mathbb{L}_e$).

We will now state and prove a pushforward formula for monomials in psi-classes under the blowdown morphism. We begin by introducing some notation. 

Recall Remark~\ref{Rem: blowup}. For an object $(\gamma,\nf)$ in $\bl_\delta\cat$ we define the sets
\begin{align*}
    S_{\gamma\setminus \delta}&:= S_\gamma\setminus \bigcup_{j\col \gamma \to \delta}j^*(S_\delta),\\
    \fu_\delta(\gamma,\nf) &:= \bigcup_{j\col \gamma \to \delta}j^*(S_\delta)\setminus \nf(\overline{j}), \\
    \cu_\delta(\gamma,\nf) &:= \bigcup_{j\col \gamma \to \delta}\nf(\overline{j}).
\end{align*}
(The symbols $\fu$ and $\cu$ will acquire some meaning in Section~\ref{Sec: wc formula} as certain collection of edges, see Equation~\eqref{def: CU FU}.)  Note  that the unions can equivalently be taken over $\oMor(\gamma,\delta)$ instead of over all morphisms. (See Remark~\ref{Rem: inclusion}).

We define $H^\delta_{\gamma,\nf} ( (g'_{e'})_{e'\in S_{\gamma,\nf}})$ as the set of tuples  $((a_e)_{e\in S_\gamma},(g_e)_{e\in S_{\gamma}})$ of non-negative integers satisfying $a_e=0$ for every $e\notin \fu_\delta(\gamma,\nf)$, 
\[
\sum_{e\in \nf(\overline{j})} (g_e + 1) = g'_{\overline{j}}+1+\sum_{e\in j^*(S_\delta)\setminus \nf(\overline{j})}a_e
\]
for every $\overline{j}\in \oMor(\gamma,\delta)\subseteq S_{\gamma,\nf}$, and $g_e=g'_e$ for every $e\in S_{\gamma}\setminus \cu_\delta(\gamma,\nf) \subseteq S_{\gamma,\nf}$.

For a morphism $h \colon (\gamma, \nf) \to (\beta, \mf)$ and a tuple $(g'_{e'})_{e'\in S_{\beta,\mf}}$ we define $h^*(g'_{e'})$ as the tuple 
\[
(h^*(g'_{e'}))_{\widetilde{e}}:= \begin{cases}
    g'_{e'} &\text{ if }\widetilde{e}=h^*(e')\text{ for some }e'\\
    -1 & \text{ if } \widetilde{e} \in S_{\gamma, \nf} \setminus h^{*} (S_{\beta, \mf}).
\end{cases}
\]

We define $M_{\delta}(\gamma)$ to be the set of all function $\mf\col \oMor(\gamma,\delta)\to \parts(S_\gamma)$ satisfying Equation~\eqref{eq:mf_empty}.

\begin{corollary} \label{prop: comparison-psi} Let $p \colon \bl_{X_{\delta'}}X_{\bullet} \to X_{\bullet}$ be the blowdown morphism, and fix integers $(g'_{e'} \geq 0)_{e' \in S_{\beta, \mf}}$. Then the pushforward
\[
p_* \frac{f_{(\beta, \mf) *}}{|\Aut(\beta, \mf)|} \Big( \prod_{e' \in S_{\beta, \mf}} \Psi_{e'}^{g'_e}\Big)\] 
equals

\[\sum_{\gamma \in \cat} \frac{f_{\gamma *}}{|\Aut(\gamma)|}\Bigg( \sum_{\substack{ \nf \in M_{\delta}(\gamma) \\ h \in \oMor((\gamma, \nf), (\beta, \mf))}} \sum_{ \substack{  (a_e,g_e)\in \\ H^\delta_{\gamma,\nf} (h^*(g'_{e'})) }}(-1)^{a_e}\binom{g_e}{a_e} \Psi_e^{g_e-a_e}  \Bigg).
\]

\end{corollary}
 \begin{proof}
  Follows from \cite[Theorem~4.8]{borisov} (or  \cite[Theorem~4.2]{aluffi}). 
 \end{proof}

\section{Combinatorial aspects of wall Crossing}

\label{sec: combo-of-wc}
In this section we fix two stability conditions $\phi^+$ and $\phi^-$ ``on opposite sides of a stability hyperplane $H$'' (Definition~\ref{Def: opposite sides}), and give a description of what will turn out to be the stratification category $\catJtilde=\catJtilde(\phi^+,\phi^-)$, recall Definition \ref{def: stratcat}), of the resolution $\jphitilde$ of the rational map $\id \colon \jtilde(\phi^+)\dashrightarrow \jtilde(\phi^-)$. This resolution will be defined in the next section, in Construction~\ref{con: blowup}.




\begin{definition} \label{Def: opposite sides}
    Let $\phi^+, \phi^- \in V_{g,n}^d$ be  nondegenerate, and let $H$ be a stability hyperplane (see Subsection~\ref{subsec: stabilityH}). We say that \emph{the polarizations $\phi^\pm$ are on opposite sides of the hyperplane $H$} if $\phi^0 = \frac{\phi^-+\phi^+}{2} \in H$ is the only degenerate point of the segment $[\phi^+, \phi^-] \subset V_{g,n}^d$.
\end{definition}
In other words, $H \ni \phi_0$-semistability implies $\phi^+$ or $\phi^-$ stability, and  $\phi^\pm$ are small perturbations of $\phi_0$. Throughout we fix $H, \phi^\pm$ and $\phi_0$ as in the above definition.

\subsection{Extremal sets, \gfs and \gcs}

\label{Sec:combowc} 
In this subsection we prove some of our bulk combinatorial results that have to do with wall crossing, focusing on the ``extremal'' multidegrees, i.e. those multidegrees that are $\phi^+$-stable but are not $\phi^-$-stable. \underline{We shall fix  a $n$ pointed graph $G$ of genus~$g$ and a divisor $D$ on $G$ throughout}. (We are not imposing that $G$ is stable, but the cases we are interested in are either $G$ equals $G'^{E'}$ for some stable graph $G'$ and some $E' \subseteq E(G')$, or $G$ is obtained by forgetting one marking from a stable marked graph.) 

 For a subset $V\subseteq V(G)$  we define\footnote{We note that the existing definition of $\beta(V)$ in \cite{esteves} and \cite{abreupacini} are based on a different sign convention, however all relevant properties remain the same.}
\[
\beta^\star(V):=-\deg(D|_{V})+\phi^\star(V)+\frac{\left|E(V, V^c)\right|}{2},\; \ \ \  \text{ for }\; \star = +, -, 0.
\]
Note that $D$ is $\phi^\star$-semistable (Definition~\ref{def: Dstable}) if and only if  $\beta^\star(V)\geq 0$ for every $V\subset V(G)$. Moreover, we have the following relation 
\begin{equation}
    \label{eq:beta_sum}
    \beta^\star(V)+\beta^\star(W)-|E(V\setminus W,W\setminus V)|=\beta^\star(V\cap W)+\beta^\star(V\cup W).
\end{equation}
for all $V, W \subseteq V(G)$ (see \cite[Lemma~4.1]{abreupacini})

\underline{From now on in this subsection we will assume that $D$ is $\phi^+$stable on $G$}.

\begin{definition} A subset $V\subsetneqq V(G)$ is called \emph{extremal} (with respect to $\phi^+,\phi^-$ and $D$) if
\begin{equation}
\label{eq:extremal}
\beta^+(V) > 0\text{ and } \beta^{-}(V)<0.
\end{equation}
\end{definition}
Note that the latter implies, in particular,
\[
\phi^+(V) > \phi^-(V);
\]
and also $\beta^0(V) = 0$.

We are now ready to define the main object of study in this section.
\begin{definition}
 We define the poset $\ext(G,D)=\ext_{\phi^+,\phi^-}(G,D)$  as 
\[
\{V\subseteq V(G); \ V\text{ is extremal, connected and with connected complement}\}
\]
with the ordering given by inclusion. 
\end{definition}

\begin{remark}
\label{rem:pullback_ext}
If $\iota\colon (G,D)\to (G',D')$ is a specialization and $V'\in \ext(G',D')$, we have that $\iota^{-1}(V')\in \ext(G,D)$.
\end{remark}

\begin{remark} \label{Rem: ext-vine} Each element $V$ of $\ext(G,D)$ corresponds to a morphism from $G$ to an extremal vine  stratum (Example~\ref{Ex: stratJ}) $(G',D')$ (up to automorphisms of $(G',D')$) obtained by contracting $E(V,V)$ and $E(V^c, V^c)$. \end{remark}

We have the following results for extremal subsets.

\begin{proposition}
\label{prop:extremal_cap_cup}
Let $V_1$ and $V_2$ be extremal subsets. Then $E(V_1\setminus V_2, V_2\setminus V_1)=\emptyset$. 

If furthermore $V_1\cap V_2 \neq \emptyset$ and $V_1\cup V_2 \neq V(G)$, then $V_1\cap V_2$ and $V_1 \cup V_2$ are extremal.
\end{proposition}
\begin{proof}
The idea is to use Equation \eqref{eq:beta_sum} to compute $\beta^\star(V_1\cap V_2)$ and $\beta^\star(V_1\cup V_2)$. In order to do that, we define
\[
H_0 := V_1\cap V_2,\, H_1:=V_1\setminus H_0,\, H_2:=V_2\setminus H_0\text{ and }H_3:=V_1^c\cap V_2^c,
\]
see Figure~\ref{fig:intersection_hemispheres}, and define 
\[
\alpha:=|E(H_1, H_2)|=|E(V_1\setminus V_2, V_2\setminus V_1)|.
\]
\begin{figure}[ht]
    \centering
    \begin{tikzpicture}[scale=1.2]
    \draw  (0,0) circle (0.5cm);
    \node at (0,0) {$H_0$};
    \draw (2,0) circle (0.5cm);
    \node at (2,0) {$H_1$};
    \draw (0,2) circle (0.5cm);
    \node at (0,2) {$H_2$};
    \draw (2,2) circle (0.5cm);
    \node at (2,2) {$H_3$};
    \draw[fill] (0.4,0) circle (0.02cm);
    \draw[fill] (1.6,0) circle (0.02cm);
    \draw  (0.4,0) -- (1.6,0);
    \draw[fill] (0.4,2) circle (0.02cm);
    \draw[fill] (1.6,2) circle (0.02cm);
    \draw  (0.4,2) -- (1.6,2);
    \draw[fill] (0,0.4) circle (0.02cm);
    \draw[fill] (0,1.6) circle (0.02cm);
    \draw  (0,0.4) -- (0,1.6);
    \draw[fill] (2,0.4) circle (0.02cm);
    \draw[fill] (2,1.6) circle (0.02cm);
    \draw  (2,0.4) -- (2,1.6);
    \draw[fill] (0.3,0.3) circle (0.02cm);
    \draw[fill] (1.7,1.7) circle (0.02cm);
    \draw  (0.3,0.3) -- (1.7,1.7);
    \draw[fill] (0.3,1.7) circle (0.02cm);
    \draw[fill] (1.7,0.3) circle (0.02cm);
    \draw  (0.3,1.7) -- (1.7,0.3);
    \node[left] at (0.7,1.3) {$\alpha$};
    \end{tikzpicture}
    \caption{}
    \label{fig:intersection_hemispheres}
\end{figure}

Since 
\[
\beta^-(H_0\cup H_1), \beta^-(H_0\cup H_2) < 0
\]
we have that $\beta^0(H_0\cup H_1) = \beta^0(H_0\cup H_2) = 0$. By \eqref{eq:beta_sum} we have
\[
\beta^0(H_0\cup H_1)+\beta^0(H_0\cup H_2) -\alpha= \beta^0(H_0)+\beta^0(H_0\cup H_1\cup H_2)
\]
and, because $\beta^0(H)\geq 0$ for every $H$, we deduce that $\alpha=0$ (which proves the second sentence), and that
\[
\beta^0(H_0)=\beta^0(H_0\cup H_1\cup H_2)=0.
\]

If $H_0\neq \emptyset$, then $\beta^+(H_0)>0$ and if   $H_0\cup H_1\cup H_2\neq V(G)$, then $\beta^+(H_0\cup H_1\cup H_2)>0$. Since $\beta^0 = \frac{\beta^++\beta^-}{2}$,  we have that $\beta^-(H_0)<0$ (respectively, $\beta^-(H_0\cup H_1\cup H_2) < 0$) if $H_0\neq \emptyset$ (respectively, if $H_0\cup H_1\cup H_2\neq V(G)$). Thus, if $H_0 \neq \emptyset$ and $H_0 \cup H_1 \cup H_2 \neq V(G)$, then $H_0$ and $H_0 \cup H_1 \cup H_2$ are extremal, concluding our proof. 
\end{proof}

\begin{proposition}
\label{prop:extremal_disconnected}
Let $V$ be an extremal set. 

If $V=V_1\sqcup V_2$ with $E(V_1,V_2)=\emptyset$ and $V_1, V_2\neq \emptyset$, then $V_1$ and $V_2$ are extremal. Similarly, if $V^c= W_1\sqcup W_2$ with $E(W_1,W_2)=\emptyset$ and $W_1,W_2 \neq \emptyset$, then $W_1^c$ and $W_2^c$ are extremal.
\end{proposition}
\begin{proof}
For the first part, we have that $\beta^0(V)=\beta^0(V_1)+\beta^0(V_2)$, since $\beta^0(V) = 0$, then $\beta^0(V_1)=\beta^0(V_2)=0$ as well. So, $\beta^-(V_i), \beta^-(V_2)<0$. The second part is proven similarly.
\end{proof}

In what follows we will also need an additional hypothesis. 

\begin{hypothesis}
\label{hyp:ext_connected}
    If $V \subseteq V(G)$ is extremal, then $\leg(1)\in V$.
\end{hypothesis}
 In particular, by Proposition~\ref{prop:extremal_disconnected}, we have that if $V\subseteq V(G)$ is extremal, then $G(V)$ is connected. From now on in this chapter we will always assume Hypothesis~\ref{hyp:ext_connected}.  

Hypothesis~\eqref{hyp:ext_connected} fixes the following convention on $\phi^+, \phi^-$:

\begin{remark} \label{Rem: minusplusconv}
  If $V$ is an extremal set in $\ext(G,D)$ then we have $\phi^+(V)>\phi^-(V)$. If we set $S:= \leg^{-1}(V)$ and $i:=g(V)$ and $t=|E(V,V^c)|$, then by Hypothesis~\eqref{hyp:ext_connected} we have:
  \[
  x_{i,t,S}^{+} >  x_{i,t,S}^{-}.
  \]
  for the coordinates $\phi^\pm=(x_{i,t,S}^\pm)_{(i,t,S)}\in V_{g,n}^d\cong C_{g,n}^d \times D_{g,n}^d$ introduced in Definition~\ref{CDT} via the isomorphism of Proposition~\ref{stabilityspace}. 
\end{remark}
 Moreover, since $\phi^+$ and $\phi^-$ are on opposite sides of a hyperplane $H=H(i_0,t_0, S_0)$, Hypothesis~\ref{hyp:ext_connected} is always satisfied upon possibly switching $\phi^+$ and $\phi^-$.

\begin{remark}
    The results in this section hold more generally when '$\phi^+$ and $\phi^-$ lie on opposite sides of a higher codimension stability plane $H$ (not necessarilly a hyperplane'. By this we mean that there exist stability hyperplanes $H_1, \ldots, H_k$ such that, setting $H:= H_1 \cap \ldots \cap H_k$, we have that $\phi^0:=\frac{\phi^++\phi^-}{2}\in H$ is the only degenerate point of the segment $[\phi^+, \phi^-]$ (cf Definition~\ref{Def: opposite sides}). In this more general case, Hypothesis~\ref{hyp:ext_connected}  becomes more restrictive.
\end{remark}

Here are some important properties of $\ext(G,D)$ that follow  from Hypothesis~\ref{hyp:ext_connected}:
\begin{corollary}
\label{cor:V1V2_union}
 Let $V_1,V_2\in \ext(G,D)$, then either $V_1\cup V_2=V(G)$ or there exists $V\in \ext(G,D)$ such that $V_1\cup V_2\subseteq V$.
\end{corollary}
\begin{proof}
    Assume that $V_1\cup V_2\neq V(G)$. Then, by Proposition~\ref{prop:extremal_cap_cup}, we have that $V_1\cup V_2$ is extremal. By Hypothesis~\ref{hyp:ext_connected}, we have that $V_1\cup V_2$ is also connected. Let $W$ be a connected component of $(V_1\cup V_2)^c$. By Proposition~\ref{prop:extremal_disconnected}, we have that $W^c$ is extremal. Moreover, since $V_1\cup V_2$ is connected (and $G$ is connected), so is $W^c$. This proves that $W^c\in \ext(G,D)$ and $V_1\cup V_2\subseteq W^c$.
\end{proof}

\begin{corollary}
\label{cor:V1V2_intersection} 
Let $V_1, V_2$ be elements of $\ext(G,D)$ such that $ V_1, V_2\subseteq V$ for some $V\in \ext(G,D)$. Then $V_1\cap V_2\in \ext(G,D)$. 
\end{corollary}
\begin{proof}
    By Proposition~\ref{prop:extremal_cap_cup}, we have that $V_1\cap V_2$ is extremal and by Hypothesis~\ref{hyp:ext_connected} we have that $V_1\cap V_2$ is nonempty and connected. All that is left is to prove that \[(V_1\cap V_2)^c= V^c \cup (V\setminus V_1)\cup (V\setminus V_2)\] is connected. But this follows from the fact that \[V^c, \quad  V_1^c= V^c\cup (V\setminus V_1),  \textrm{ and} \quad V_2^c=V^c\cup (V\setminus V_2)\] are connected.
\end{proof}

We are now ready to introduce a key notion to describe the blowup category of $\catJ$ at some vine strata.

\begin{definition}
\label{def:good_function}
For each $(G,D)$ and lower set $L\subseteq \ext(G,D)$, we say that a function $\alpha\colon L\to \mathcal{P}(E(G))$ is a \emph{\gf} if the following conditions hold.
\begin{enumerate}
    \item $\alpha(V)\subseteq E(V,V^c)$ for every $V \in L$.
    \item For all $V\in L$ we have $\alpha(V) = \emptyset$ if and only if there exists $V'\subsetneqq V$ with $V'\in \ext(G,D)$ such that $\alpha(V')\cap E(V,V^c)\neq \emptyset$.
\end{enumerate}
We usually think of $L$ as part of the datum of $\alpha$, and write $L_{\alpha}$ for the domain of the \gf $\alpha$. We also define $|\alpha|=\bigcup_{V\in L_{\alpha}}\alpha(V) \subseteq E(G)$.
\end{definition}
\begin{definition}
\label{def:good_function_compatible}
Given a specialization $\iota\colon (G,D)\to (G', D')$ we say that the \gfs $\alpha$ and $\alpha'$ are \emph{compatible with  $\iota$} if 
\begin{enumerate}
    \item $\iota^{-1}(V')\in L_{\alpha}$ for every $V'\in L_{\alpha'}$.
    \item if $\alpha'(V')\neq \emptyset$, then $\alpha(\iota^{-1}(V'))\neq \emptyset$.
    \item if $e'\notin |\alpha'|$ then $\iota_{E}(e)\notin |\alpha|$.
\end{enumerate}
\end{definition}
In that case we write $\iota\colon (G,D,\alpha)\to (G', D', \alpha')$ and say that the first triple specializes to the second.

We will also need to introduce some subcategories of the stratification category of $\catJ$, which will use only some \gfs which we call ``full''. We now introduce those, and then discuss how this notion is equivalent to the combinatorial notion of a \gc.

\begin{definition} \label{def: alpha-full}
We say that $\alpha$ is \emph{full} if 
$|\alpha| = E(G)$.
\end{definition}

We will show in Proposition~\ref{prop: bijection full} how full \gfs are equivalent to the following notion. 

\begin{definition} \label{def: fullfor}
A forest $V_{\bullet} \subseteq \ext(G,D)$ is a \emph{\gc} in $\ext(G,D)$ if 
\begin{enumerate}
    \item it contains all maximal elements of $\ext(G,D)$, and
    \item the edge set satisfies $E(G)=\bigcup_{V\in V_\bullet} E(V, V^c)$.
\end{enumerate} 
\end{definition}

We  first prove some intermediate results in that direction. For a forest $V_\bullet\subseteq \ext(G,D)$, and for each $V'\in \ext(G,D)$, we define \begin{equation} \label{def: next} \nex(V')= \nex_{V_\bullet}(V') :=  \bigcap_{V'\subsetneqq V\in V_\bullet} V  \end{equation}
(with the usual convention that the intersection over the empty set equals $V(G)$).

\begin{lemma} \label{lemma: incomp}
    Let $V_\bullet\subseteq\ext(G,D)$ be a full forest and let $V_1$ and $V_2$ be two incomparable elements in $V_\bullet$. Then $V_1\cup V_2=V(G)$ and $E(V_1^c, V_2^c)=\emptyset$.
\end{lemma}
\begin{proof}
    By Corollary \ref{cor:V1V2_union} we have that either $V_1\cup V_2= V(G)$, or there exists $V\in \ext(G,D)$ such that $V_1, V_2\subseteq V$. If the latter holds, Part~(1) of Definition~\ref{def: fullfor} implies that $V_1$ and $V_2$ are comparable, a contradiction. So $V_1\cup V_2=V(G)$. The fact that $E(V_1^c, V_2^c)=\emptyset$ follows from Proposition~\ref{prop:extremal_cap_cup}.
\end{proof}

\begin{proposition}
\label{prop:properties_next}
    Let $V_\bullet\subseteq \ext(G,D)$ be a full forest. Let $V'\in \ext(G,D)$ be a nonmaximal element and let $V_1,\ldots, V_m$ be the elements of $V_\bullet$ that are minimal among those containing $V'$. Then:
    \begin{enumerate}
        \item For every $i=1,\ldots, m$, \[
E(\nex(V')\setminus V', V_i^c)\neq \emptyset.
\]
\item $\nex(V')\neq V'$. 
\item $G(\nex(V'))$ is connected. 
\item If $V'\in V_\bullet$, then $E(\nex(V')\setminus V', \nex(V')\setminus V' ) =  \emptyset$.
\item If $V\in \ext(G,D)$ is such that $\nex(V')\subseteq V$, then there exists $i \in \{1,\ldots, m\}$ such that $V_i\subseteq V$. 
        \item  All minimal elements of $\ext(G,D)$ belong to $V_{\bullet}$.
\item We have $E(G)=\bigsqcup_{V\in V_\bullet}E(V,\nex(V)\setminus V)$.
    \end{enumerate}
\end{proposition}
\begin{proof}
    
For~(1), we first notice that $V_i,V_j$ are inconparable if $i\neq j$, because of the minimality condition. By Lemma \ref{lemma: incomp}, we have that $E(V_i^c, V_j^c)=\emptyset$ for all $i\neq j$. Since \[V'^c = (\nex(V')\setminus V')\cup\bigcup_{i=1}^m V_i^c\] and $V'^c$ is connected, we must have that $E(\nex(V')\setminus V', V_i^c)\neq \emptyset$. Item~(2) follows then immediately from Item~(1).

For (3), just notice that $\nex(V')$ is an extremal element by Proposition~\ref{prop:extremal_cap_cup}, then it is connected by Hypothesis~\ref{hyp:ext_connected}. 

We now prove Item~(4). We deduce this by combining the fact that $V_\bullet$ is a full forest, with the fact that
\begin{equation}\label{Eq: intVbullet}
E(\nex(V')\setminus V', \nex(V')\setminus V' ))\cap E(V,V^c)=\emptyset
\end{equation}
for all $V \in V_\bullet$. To prove the latter fact, let $V$ be an element of $V_\bullet$. We distinguish three cases. If $V$ is incomparable with $V'$, by Lemma~\ref{lemma: incomp} we have that $V\cup V'=V(G)$, which implies  that $\nex(V')\setminus V'\subseteq V$ and hence Equation~\eqref{Eq: intVbullet}. If $V'\subseteq V$, then $V_i\subseteq V$ for some $i$, and hence $\nex(V')\subseteq V$, which again implies Equation~\ref{Eq: intVbullet} holds in this case. Finally, in the case  $V\subseteq V'$, Equation~\ref{Eq: intVbullet} follows immediately from the definition of $\nex(V')$. This completes the proof of Item~(4).

Item~(5). This follows from the fact that \[V(G)=\nex(V')\cup\bigcup_{j=1}^m V_j^c\] and hence \[V^c=\bigcup_{j=1}^m V^c\cap V_j^c.\] Since $G(V^c)$ is connected and $E(V_i^c,V_j^c)=\emptyset$ for $j \neq i$ (Lemma~\ref{lemma: incomp}), so we have that there exists $i\in \{1,\ldots, m\}$ such that $V^c\cap V_j^c=\emptyset$ for every $j\neq i$. This means that $V_i\subset V$.

     Item~(6). Let $V_0$ be a minimal element of $\ext(G,D)$. If $V_0$ is also maximal, there is nothing to do. So we can assume that $V_0$ is nonmaximal. Assume by contradiction that $V_0\notin V_\bullet$. From Items (2) and (3) we deduce that \[E(V_0,\nex(V_0)\setminus V_0)\neq \emptyset.\]  Let $V\in V_\bullet$ we will  prove that \[E(V_0, \nex(V_0)\setminus V_0)\cap E(V,V^c)=\emptyset\] and get a contradiction. If $V_0\subseteq V$, then $\nex(V_0)\subset V$ (recall that $V_0\neq V$ because $V_0\notin V_\bullet$), then \[E(V_0, \nex(V_0)\setminus V_0)\cap E(V,V^c)=\emptyset.\] If $V_0\not\subseteq V$, we have that $V_0\cap V\subsetneqq V_0$ is extremal by Proposition~\ref{prop:extremal_cap_cup}; by the minimality of $V_0$ in $\ext(G,D)$ we have that $(V_0\cap V)^c = V_0^c\cup V^c$ is not connected. Since $V_0^c$ and $V^c$ are connected, we have that $V_0^c\cap V^c =\emptyset$ and hence $V_0\cup V=V(G)$. By Proposition~\ref{prop:extremal_cap_cup} we deduce that $E(V_0^c, V^c)=\emptyset$. Hence \[E(V_0,V_0^c)\cap E(V,V^c)=\emptyset,\] because $E(V_0,V_0^c)=E(V_0\cap V, V_0^c)$ and $E(V,V^c)=E(V_0\cap V, V^c)$. In particular \[E(V_0,\nex(V_0)\setminus V_0)\cap E(V,V^c)=\emptyset.\] 
     
   Item~(7). Let $e$ be an edge of $G$ and let $V'$ be an element in $V_\bullet$ that is maximal among those with the property that $e\in E(V',V'^c)$. For each $V'\subsetneqq V\in V_\bullet$, we must have that $e\in E(V,V)$, hence that $e\in E(V', \nex(V')\setminus V')$.

\end{proof}


\begin{proposition}
\label{prop:alpha_to_forest}
Let $\alpha$ be a \gf and define \[V_\bullet:=\{V\in \ext(G,D): \ \alpha(V)\neq \emptyset\}\subseteq \ext(G,D).\] Then the following hold.

\begin{enumerate}
    \item For every $V_0\in \ext(G,D)$, we have that $ \{V \in V_\bullet;\ V \subseteq V_0\}$
    is a chain. 
    \item The poset $V_\bullet$ is a forest. 
    \item $\alpha(V)\subseteq E(V,\nex(V)\setminus V)$ for every $V\in V_\bullet$.
\end{enumerate}
\end{proposition}
\begin{proof}
Let $V_1, V_2\in V_\bullet$ be such that $V_1,V_2\subseteq V_0$ and $V_1$ and $V_2$ are incomparable. By Corollary \ref{cor:V1V2_intersection}, we have that $V_1\cap V_2 \in \ext(G,D)$. Moreover, we have that 
\[
E(V_1\cap V_2, (V_1\cap V_2)^c)\subseteq E(V_1,V_1^c)\cup E(V_2, V_2^c).
\]
By the fact that $\alpha$ is a \gf, we have that there exists $V'\subseteq V_1\cap V_2$ such that \[\alpha(V')\cap E(V_1\cap V_2, (V_1\cap V_2)^c)\neq \emptyset.\] 
This means that either $\alpha(V')\cap E(V_1, V_1^c)\neq \emptyset$, or $\alpha(V')\cap E(V_2,V_2^c)\neq \emptyset$, which contradicts the fact that $\alpha$ is a \gf and $\alpha(V_1), \alpha(V_2)\neq \emptyset$. This concludes the proof of Item~(1). Item~(2) follows directly.


We now prove~(3). If $\nex(V)=V(G)$ there is nothing to prove. Otherwise, we have
\[E(V, V^c)\setminus E(V,\nex(V)\setminus V)\subseteq \bigcup_{V\subsetneqq V'\in V_\bullet} E(V', V'^c),\] so  \[\alpha(V)\cap(E(V, V^c)\setminus E(V,\nex(V)\setminus V))\neq \emptyset\] would imply $\alpha(V)\cap E(V', V'^c)\neq \emptyset$ for some $V'\in V_\bullet$ with $V\subsetneqq V'$, thus contradicting the assumption that $\alpha$ is a \gf.
\end{proof}

\begin{corollary} \label{cor: alphafull}
Let $\alpha$ be a vine function on $(G,D)$. If $\alpha$ is full, then $L_\alpha=\ext(G,D)$.
\end{corollary}
\begin{proof}
    Assume, by contradiction, that $\ext(G,D)\setminus L_\alpha \neq \emptyset$. Because $L_\alpha$ is a lower set, there exists a maximal element $V_0$ of $\ext(G,D)$ in $\ext(G,D)\setminus L_\alpha$.  
    
    Consider the set \[\{V\subseteq V_0; \ \alpha(V)\neq \emptyset\}.\] By Item~(1) of Proposition~\ref{prop:alpha_to_forest} above, we have that this set has a unique maximal element $V_1\subsetneqq V_0$.

    We claim that $E(V_0\setminus V_1, V_0^c)\cap |\alpha|=\emptyset$. Indeed, assume for a  contradiction that there exists an edge $e\in E(V_0\setminus V_1, V_0^c)\cap |\alpha|$ and let $v\in V_0\setminus V_1$ and $w\in V_0^c$ be its endpoints. Since $e\in |\alpha|$, there exists an extremal set $V_2$ such that $e\in E(V_2, V_2^c)$. 
    \begin{itemize}
        
    \item If $v\in V_2$ and $w\in V_2^c$, then $V_2\cup V_0\neq V(G)$ and by Corollary \ref{cor:V1V2_union}, $V_2$ and $V_0$ are contained in a comment element of $\ext(G,D)$, so $V_2 \subseteq V_0$ by the maximality of $V_0$. Hence $V_2\subset V_1$ by the maximality of $V_1$, which implies that $v\in V_1$, a contradiction. 
    
    \item If instead $v\in V_2^c$ and $w\in V_2$, then $v\in V_0\setminus V_2$ and $w\in V_2\setminus V_0$ and hence $e\in E(V_0\setminus V_2, V_2\setminus V_0)=\emptyset$, which contradicts Proposition~\ref{prop:extremal_cap_cup}.  
    \end{itemize}
    
    Now we observe that $E(V_0\setminus V_1, V_0^c)\neq \emptyset$. This follows from the fact that $V_1^c = V_0^c \sqcup (V_0\setminus V_1)$ and $V_1^c$ induces a connected subgraph.

    By combining $E(V_0\setminus V_1, V_0^c)\cap |\alpha|=\emptyset$ with $E(V_0\setminus V_1, V_0^c)\neq \emptyset$, we contradict the assumption that $\alpha$ is full.
\end{proof}

We are now ready to prove the equivalence of full \gfs and \gcs.

\begin{proposition} \label{prop: bijection full}
For each $(G,D)$, the mapping defined in  Proposition~\ref{prop:alpha_to_forest} induces  a natural bijection between full \gfs and \gcs in $\ext(G,D)$. 
\end{proposition}
\begin{proof}
Assume that $\alpha$ is a full \gf and define $V_{\bullet}= V_{\bullet}^{\alpha}$ as given by Proposition~\ref{prop:alpha_to_forest}. 
By loc.cit we have that $V_\bullet$ is a forest and that $\alpha(V)\subseteq E(V,\nex(V)\setminus V)$ for every $V$. The condition  \[\bigcup_{V\in V_\bullet} \alpha(V) = \bigcup_{V\in \ext(G,D)} \alpha(V)=E(G)\]implies
that $\bigcup E(V,\nex(V)\setminus V) = E(G)$, which in turn implies $\bigcup E(V,V^c)= E(G)$. This proves that $V_{\bullet}$ is a full forest.

We now study the inverse mapping. For a \gc $V_\bullet$, define 
\[
\alpha:= \alpha_{V_\bullet}(V_0) = \begin{cases} \emptyset & \textrm{if } V_0\notin V_{\bullet}; \\ E(V_0, \nex(V_0) \setminus V_0) & \textrm{if } V_0\in V_\bullet. \end{cases}
\]
We claim that $\alpha$ is a \gf.  Condition~(1)  in Definition~\ref{def:good_function} follows from the fact that $E(V_0, \nex(V_0)\setminus V_0)\subseteq E(V_0, V_0^c)$ for every $V\in V_\bullet$. 

Let us prove Condition~(2). First, we see that $\alpha_{V_\bullet}(V_0)=\emptyset$ if and only if $V_0\notin V_\bullet$. By the definition of $\alpha$ it is clear that if $V_0\notin V_\bullet$, then $\alpha(V_0)=\emptyset$. On the other hand, if $\alpha(V_0)=\emptyset$ and $V_0\in V_\bullet$, then $E(V_0, \nex(V_0)\setminus V_0)=\emptyset $, but this is a contradiction with the fact that $\nex(V_0)$ induces a connected subgraph of $G$ and $\nex(V_0)\setminus V_0\neq \emptyset$ (see Proposition~\ref{prop:properties_next}).

Now we show that if $\alpha(V_0)=\emptyset$, then we can find $V'\in \ext(G,D)$ with $V'\subseteq V_0$,  such that $\alpha(V')\cap E(V_0, V_0^c)\neq \emptyset$. By the previous paragraph, we have that $V_0\notin V_\bullet$.  Since $V_\bullet$ contains all minimal elements of $\ext(G,D)$ (by Proposition~\ref{prop:properties_next}), we have that there exists $V\in V_\bullet$ contained in $V_0$. Let $V'$ be the maximum such element. This maximum exists because $V_\bullet$ is a forest that contains all maximal elements of $\ext(G,D)$. 
By Item~(5) of Proposition~\ref{prop:properties_next} and the maximality of $V'$, we have that $\nex(V')\not\subseteq V_0$. Moreover, by Item~(4) of Proposition~\ref{prop:properties_next} we have that \[E(\nex(V')\setminus V_0, \nex(V')\cap V_0\setminus V')=\emptyset.\] Since $\nex(V')$ induces a connected subgraph (this is Item~(3) of Proposition~\ref{prop:properties_next}), we have that $E(V', \nex(V')\setminus V_0)\neq \emptyset$, and since 
\[
E(V', \nex(V')\setminus V_0)\subseteq E(V',\nex(V')\setminus V')\cap E(V_0, V_0^c) = \alpha(V')\cap E(V_0, V_0')
\]
we have that $\alpha(V')\cap E(V_0, V_0')\neq \emptyset$ as needed.

We now prove that if there exists a $V'\in \ext(G,D)$ with $V'\subsetneqq V_0$ such that $\alpha(V')\cap E(V_0,V_0^c)\neq \emptyset$, then $\alpha(V_0)=\emptyset$. Assume by contradiction that there exist $V_0, V'\in V_\bullet$ such that $\alpha(V_0)\neq \emptyset$, and that $V'\subsetneqq V_0$ and $\alpha(V')\cap E(V_0, V_0^c)\neq \emptyset$. Since $\alpha(V_0),\alpha(V')\neq \emptyset$, we have that $V_0,V'\in V_\bullet$. Since $V'\subsetneqq V_0$, we have that $\nex(V')\subseteq V_0$ as well, so this implies that  \[E(V', \nex(V')\setminus V')\cap E(V_0, V_0^c) = \emptyset,\] a contradiction (recall that $E(V',\nex(V')\setminus V')=\alpha(V')$). This concludes the proof that $\alpha$ is a \gf.

The fact  that $\alpha$ is full follows from Item~(7) of Proposition~\ref{prop:properties_next}.

\end{proof}

\begin{definition}
\label{def:morphism_Vbullet}
    A morphism  $f \col (G,D,V_\bullet)\to (G',D',V'_\bullet)$ is a morphism $f \col(G,D)\to (G',D')$ such that $f^{-1}(V')\in V_\bullet$ for every $V'\in V'_\bullet$. This is equivalent to requiring that $f$ is compatible with $\alpha_{V_\bullet}$ and $\alpha_{V'_\bullet}$.

    \end{definition}

Given morphisms $f_i\colon (G,D,\alpha)\to (G_i, D_i, \alpha_i)$ for $i=1,\ldots,k$, we say that the collection of morphisms $(f_1, \ldots, f_k)$ is \emph{generic} if 
\begin{enumerate}
    \item for every edge $e\in E(G)\setminus|\alpha|$ there exists some $i\in \{1,\ldots, k\}$ and some $e'\in E(G_i)\setminus |\alpha_i|$ such that $e=f_{i}^*(e')$;
    
    \item for every $V\in L_\alpha$ such that $\alpha(V)\neq \emptyset$, there exists $i$ and $V'\in L_{\alpha_i}$ with $\alpha_i(V')\neq \emptyset$ such that $f^{-1}_i(V')=V$.

\end{enumerate}

\begin{remark}
In the next section, we will see how this definition of ``generic'' matches the one given in Definition~\ref{def: generic}.
\end{remark}

\begin{proposition}
\label{prop:intersection_full}
Let $f_i\colon (G,D,\alpha)\to (G_i, D_i,\alpha_i)$ for $i=1, \ldots, k$ be a generic collection of morphisms. If all $\alpha_i$ are full, then $\alpha$ is also full.
\end{proposition}
\begin{proof}
Since the collection is generic, we have that for every $e\in E(G)\setminus |\alpha|$ there exists some $i\in\{1,\ldots, k\}$ and some $e'\in E(G_i)\setminus |\alpha_i|$ such that $e=f_{i}^*(e')$. Because we have $|\alpha_i|=E(G_i)$ for every $i=1,\ldots, k$, we conclude that $|\alpha|=E(G)$, hence that $\alpha$ is full.
\end{proof}
The next result will imply that all the strata that we blowup have transversal self-intersection.
\begin{proposition}\label{prop: transversal-self}
Assume $(G, D,\alpha)$ is such that $G$ is a vine graph  and $L_{\alpha}=\emptyset$. Let $f_i\colon (G',D', \alpha') \to (G,D,\alpha)$ for $i=1,2$ be generic. Assume that for all $V \in \ext(G', D')$ such that $V\subseteq f_1^{-1}(\{v_0\})\cap f_2^{-1}(\{v_0\})$, we have $V\in L_{\alpha'}$. Then $(f_1, f_2)$ is transversal.
\end{proposition}
\begin{proof}
Denote by $V_i=f_i^{-1}(\{v_0\})$. Assume for a contradiction that $V_1\cup V_2\neq V(G)$. This means that $V_1\cap V_2 \in \ext(G,D)$ and hence in $L_{\alpha'}$. Since \[E(V_1\cap V_2, (V_1\cap V_2)^c)\subseteq E(V_1,V_1^c)\cup E(V_2, V_2^c),\] that would imply that there exists $V'\in L_{\alpha'}$ such that \[\alpha(V')\cap E(V_1\cap V_2, (V_1\cap V_2)^c)\neq \emptyset,\] and in turn, if $e\in \alpha(V')\subseteq E(V_1,V_1^c)$ we would have a contradiction with the fact that $e\notin |\alpha|$ but $f_i^*(e)\in |\alpha'|$. 

So $V_1\cup V_2=V(G)$, which implies that $E(V_1,V_1^c)\cap E(V_2,V_2^c)=\emptyset$.
\end{proof}


The next proposition will be used to prove that the category $\catJtildeY$, defined later, is a \emph{simple} normal crossing stratification category (as in Example~\ref{ex: snc}).  Recall that, as stipulated in Subsection~\ref{Not: curves and sheaves}, when $E$ is empty, we simply write $(G,D)$ in place of $(G, (\emptyset,D))$. 

\begin{proposition}
\label{prop:simple_vine_curve_crossing}
Let $(G',D')$ be $\phi^+$-stable, $(n+1)$-marked vine graph such that  $v_{n+1}\notin V$ for every $V\in \ext(G',D')$. Let $V_\bullet$ be a full forest, and $f_1,f_2\colon (G,D,V_\bullet) \to (G',D', V'_\bullet)$ be morphisms. Then $f_1\in \Aut(G',D',V'_\bullet)f_2$.
\end{proposition}

\begin{proof}
Upon further contraction of $(G, D, V_{\bullet})$, we can assume that $f_1, f_2 $ are generic. This means that either $(G',D',V_\bullet')=(G,D,V_\bullet)$ or $V_\bullet=\{V_1, V_2\}$. In the first case our statement follows from the fact that every endomorphism is an automorphism. In order to conclude, all we have to do is to exclude the second case.

In the second case, there are two subcases  (up to swapping $V_1$ and $V_2$): either $V_1\subseteq V_2$ or $V_1$ and $V_2$ are incomparable. In the latter subcase, we have that $V_1\cup V_2=V(G)$ by Corollary~\ref{cor:V1V2_union}, but that contradicts the fact that $v_{n+1}\notin V_1\cup V_2$.

In the former subcase, we have that \[g(V_1)=g(V_2), \ |E(V_1,V_1^c)|=|E(V_2, V_2^c)|, \ \leg^{-1}(V_1)=\leg^{-1}(V_2) \textrm{ and } D(V_1)=D(V_2).\]  This implies that for each $v \in V_2\setminus V_1$ we have $g(v)=0$, $\leg^{-1}(\{v\})=\emptyset$, $|E(v)|=2$, and $D(v)=0$. This contradicts the assumption that $(G, D)$ is $\phi^+$-stable.
\end{proof}

We conclude this section by observing that the existence of a full function/forest rules out the presence of exceptional vertices.

\begin{lemma} \label{lem: noE} If $\ext(G, D)$ admits a \gc, then $G$ is stable.
\end{lemma}
\begin{proof}
Because $D$ is $\phi^+$ semistable, if $G$ fails to be stable, it contains an exceptional vertex $v$ (meaning that $v$ has genus $0$, no marked points, it has valence $2$ and $D(v)=1$). Let $e_1,e_2$ be the two edges of $G$ that contain $v$. 

If $(G,D)$ admits a \gc $V_{\bullet}$, then there are $V_1,V_2$ such that $e_i \in E(V_i, V_i^c)$ for $i=1,2$. If $v\notin V_i$ for some $i$, then $\beta^\star_D(V_i)=\beta^\star(V_i\cup\{v\})+1$, in particular $\beta^\star_D(V_i)>0$, a contradiction. This means that $v\in V_1\cap V_2$. On the other hand, we have that $e_1,e_2\in E(V_1\cap V_2,(V_1\cap V_2)^c)$, that $V_1\cap V_2$ is extremal (by~\ref{prop:extremal_cap_cup}),  hence connected, and that $\leg(1)\in V_1\cap V_2$. This contradicts the assumption that $v$ is exceptional.
\end{proof}

\subsection{The stratification categories}
In light of the results of the previous subsection, in this subsection we are now ready to define the stratification category $\catJtilde$, and 
 some other  categories $\catJtildeE$ and $\catJtildeY$ that will play an important role in our proof of Theorem~\ref{maintheorem}. 

    \begin{definition} \label{Def: Jcategories}

We define $\catJtilde= \catJtilde(\phi^+, \phi^-)$ as a skeleton of the category whose objects are triples $(G,D, \alpha)$ such that $(G,D)$ is an object of $\catJ$ and $\alpha$ is a \gf with the property that $L_\alpha= \ext(G,D)$. Morphisms are given as in Definition~\ref{def:good_function_compatible}.

    We define $\catJtildeE$ as the full subcategory\footnote{By Corollary~\ref{cor: alphafull}.} of $\catJtilde$ whose objects $(G,D, \alpha)$ with $\alpha$ \emph{full}. \footnote{By Proposition~\ref{prop:alpha_to_forest}, objects are, equivalently, triples $(G,D, V_{\bullet})$ with $V_{\bullet}$ a full forest in $\ext(G,D)$.}

    We then define the category $\catJtildeY$ as a skeleton of the category whose objects are triples $(G,D,V_\bullet)$ where $G$ is a $n+1$ pointed stable graph of genus $g$, the divisor $D$ is $(\phi^+,\leg(1))$-quasistable, and $V_\bullet$ is a full forest such that $\leg(n+1)\notin V$ for every $V\in V_\bullet$ \footnote{By Lemma~\ref{lemma: incomp}, this implies that $V_{\bullet}$ is a chain.}. Morphisms are specializations as in Definition~\ref{def:good_function_compatible}.
\end{definition}
 
These categories will be interpreted geometrically in Remark~\ref{Rem: alpha-gives-blowup}. Note that the rank~$1$ objects (the divisors) of $\catJtildeE$ are the triples $(G,D,V_{\bullet})$ such that $G$ is a vine graph and $V_{\bullet}$ contains a single element (by Hypothesis~\ref{hyp:ext_connected}, the vertex containing the first marking).

\subsection{The case of ``good'' hyperplanes.} \label{sec:goodwalls}

In this subsection, we fix a stability hyperplane $H=H(i,t,S;k)$ that satisfies   $S^c\neq \emptyset$. In this case, we prove that the corresponding exceptional vine curves loci in the compactified universal Jacobian have pairwise empty intersections. The main result here is:

\begin{proposition} 
\label{prop:good_disjoint}
The objects of $\catJtildeE$ are triples $(G,D,V_\bullet)$ 
satisfying either
\begin{enumerate}
 \item $G$ has no edges and $V_\bullet$ is empty. This is the terminal object.
 \item $G$ is a vine graph and $V_\bullet$ has a single element $V=\{\leg(1)\}$.
\end{enumerate}
\end{proposition}
Note that by Proposition~\ref{prop:goodwalls}, each vine graph as in (2) above is necessarily of the form $G(i-j,t+2j,S)$ for some $i,t \geq 0$ and for all $j\geq 0$ satisfying $-t/2 < j \leq \min(i, g+1-t-i)$.

\begin{proposition}
\label{prop:L1V}
 Let $V\in \ext(G,D)$, then $\leg^{-1}(V)=S$.
\end{proposition}
\begin{proof}
Let \[G':=G/(E(V,V)\cup E(V^c,V^c))\] be the vine graph associated to $V$. If $\leg^{-1}(V)\neq S$, then $(\phi_0)_{G'}$ is nondegenerate by Proposition~\ref{prop: subsetscoincide}. This contradicts the assumption that $V$ is extremal. 
\end{proof}

For our next result, recall that the canonical divisor $K^{\log}_G$ of a graph $G$ is defined by $K^{\log}_G(v)= 2g(v)-2+|E(v)|+|\leg^{-1}(v)|$ for all $v \in V(G)$.

\begin{lemma}
\label{lem:Ext_Klog}
If $V_1,V_2\in \ext(G,D)$, then $K_G^{\log}(V_1)=K_{G}^{\log}(V_2)$.
\end{lemma}
\begin{proof}
We have that \[K_G^{\log}(V)=2g(V)-2+|E(V,V^c)|+|\leg^{-1}(V)|.\] By Proposition~\ref{prop:L1V}, we have that $|\leg^{-1}(V_1)|=|\leg^{-1}(V_2)|$. By Proposition~\ref{prop:goodwalls} we conclude that \[2g(V_1)-2+|E(V_1,V_1^c)|=2g(V_2)-2+|E(V_2,V_2^c)|.\]
\end{proof}

\begin{proposition}
\label{prop:ext_at_most_1}
  If $(G,D)$ is  $\phi^+$-stable, then $|\ext(G,D)|\leq 1$.
\end{proposition}
\begin{proof}
Assume there are different elements $V_1\neq V_2$ of $\ext(G,D)$. By Proposition~\ref{prop:L1V}, we have  $\leg^{-1}(V_1)=\leg^{-1}(V_2)\neq\emptyset$, so $V_1\cap V_2\neq \emptyset$. Moreover, $\leg^{-1}(V_1^{c})=\leg^{-1}(V_2^{c})\neq \emptyset$, so $V_1\cup V_2\neq V(G)$. By Propositions \ref{prop:extremal_cap_cup}, \ref{prop:extremal_disconnected} and \ref{prop:L1V} we have that $V_1\cap V_2 \in \ext(G,D)$. \par 

So we may assume $V_1\subseteq V_2$. By Lemma~\ref{lem:Ext_Klog} we have that $K_G^{\log}(V_1) = K^{\log}_G(V_2)$, which implies that $K_G^{\log}(V_2\setminus V_1) = 0$, which is a contradiction with the fact that $G$ must be stable by Lemma~\ref{lem: noE}. 
\end{proof}

\begin{corollary}
We have that $(G,D)$ has a nonempty \gc $V_\bullet$ if and only if $G$ is a vine graph and $\beta^-_D(\{\leg(1)\})<0$.
\end{corollary}
\begin{proof}
The 'if' part is immediate, so we focus on the other implication. By Proposition~\ref{prop:ext_at_most_1} we have that $\ext(G,D)$ has at most one element. Since $V_\bullet$ must be nonempty, we have that $V_\bullet=\ext(G,D)$. Because $V_\bullet=\{V\}$ is a \gc, we deduce that $E(G)=E(V,V^c)$, and that  $G(V)$ and $G(V^c)$ are connected. This means that both $V$ and $V^c$ are singletons and hence that $G$ is a vine graph.
\end{proof}

\begin{proof}[Proof of Proposition~\ref{prop:good_disjoint}]
Let $(G',D')$ be a pair with different specializations
\[
f_1\colon (G',D')\to (G_1,D_1)\,\text{ and }\, f_2\colon (G',D')\to (G_2,D_2)
\]
to extremal pairs. By Remark~\ref{rem:pullback_ext} we have that  \[\ext(G',D')\supseteq f_1^{-1}(\ext(G_1,D_1))\cup f_2^{-1}(\ext(G_2,D_2)),\] for $i=1,2$, which means that $\ext(G',D')$ has at least $2$ elements, thus contradicting Proposition~\ref{prop:ext_at_most_1}. 
\end{proof}
    
\section{Nonsingular resolution of the identity}
\label{Sec: blowup construction}


Let $\phi^-, \phi^+ \in V_{g,n}^d$ be on opposite sides of a stability hyperplane $H$ (Definition~\ref{Def: opposite sides}).  In this section we construct a nonsingular resolution $\jphitilde$ of the identity map $\mathsf{Id} \colon \Jb{d}{g,n} (\phi^+)\dashrightarrow \Jb{d}{g,n} (\phi^-)$. We construct $\jphitilde$ as an iterated blowup of $\Jb{d}{g,n}(\phi^+)$ at certain strata indexed by vine graphs with extremal bidegrees $(G^i,D^i)$ (recall from the previous section that a multidegree is 'extremal' if it is $\phi^+$-stable but not $\phi^-$-stable). 

To define the order in which we blowup the vine curve strata, we first introduce a partial order. Let $(G_i,D_i)$, for $i=1,2$, be a pair where $G_i$ is a vine graph and $D_i$ is an extremal bidegree, we also set $v_i$ to be the vertex of $G_i$ such that $\beta_{D_i}^-(\{v_i\})<0$ (i.e., it is the vertex with the first marked point, in particular $\phi^+$ ans $\phi^-$ satisfy Hypothesis \ref{hyp:ext_connected}). 

We say that $(G_1,D_1) \leq (G_2,D_2)$ if there exists   $(G,D) \in \mathfrak{C}_{g,n}(\phi^+)$ and morphisms $f_i\colon (G,D)\to (G_i,D_i)$ such that $f_1^{-1}(\{v_1\})\subseteq f_2^{-1}(\{v_2\})$. Note that, in particular, $f_i^{-1}(\{v_i\})\in \ext(G,D)$ (see Remark~\ref{rem:pullback_ext}). 

The next proposition guarantees that this  preorder is indeed a partial order.
\begin{proposition} \label{prop: partial order}
 Assume $(G_1,D_1),(G_2,D_2)$ are vine graphs with extremal bidegrees for $\phi^+, \phi^-$. Then $(G_1,D_1)\leq (G_2,D_2)$ if and only if $\leg^{-1}_{G_1}(\{v_1\}) \subseteq \leg_{G_2}^{-1}(\{v_2\})$ and $g_{G_1}(v_1) \leq g_{G_2}(v_2)$ and $g_{G_1}(v_1) + |E(G_1)| \leq g_{G_2}(v_2)+|E(G_2)|$.
\end{proposition}
\begin{proof}
Let us begin proving the ``only if'' part. If $(G_1,D_1)\leq (G_2,D_2)$, then there are morphisms $(G,D)\to(G_i,D_i)$, $i=1,2$, such that 
\[
V_1:= f_1^{-1}(\{v_1\})\subseteq f_2^{-1}(\{v_2\})=: V_2\text{ and }V_1, V_2\in \ext(G,D).
\]
 This proves that $\leg^{-1}_{G_1}(\{v_1\})\subseteq \leg^{-1}_{G_2}(\{v_2\})$ and that $g_{G_1}(v_1)=g(V_1)\leq g(V_2)=g_{G_2}(v_2)$. Since $V_2^c\subset V_1^c$, it is also true that \[g - g_{G_2}(v_2)-|E(G_2)| = g(V_2^c)\leq g(V_1^c) = g - g_{G_1}(v_1)-|E(G_1)|.\] Note that the inequalities $g(V_1)\leq g(V_2)$ and $g(V_2^c)\leq g(V_1^c)$ hold because $V_1, V_2, V_1^c, V_2^c$ all induce connected subgraphs.
 
For the ``if'' part, consider the graph $G$ with $3$ vertices $w_1$, $w_2$, $w_3$, all markings on $w1$, with $|E_G(w_1,w_3)|=\lambda$  and $|E_G(w_1,w_2)|=|E(G_1)|-\lambda$ and $|E_G(w_2,w_3)|=|E(G_2)|-\lambda$. Set $g_G(w_1)=g_{G_1}(v_1)$ and \[g_G(w_2)=g_{G_2}(v_2)-g_{G_1}(v_1)+\lambda -|E(G_1)| +1,\] and then set $g_G(w_3)$ so that $g(G)=g$. The numerical assumptions in the statement guarantee the existence of $\lambda$ such that $|E_G(w_i,w_j)| \geq 1$ for all $i \neq j$ and $g_G(w_2), g_G(w_3) \geq 0$. It is then straightforward to check that the given graph $G$ admits a morphism to $G_1$ (by contracting $E(w_2,w_3)$) and to $G_2$  (by contracting $E(w_1,w_2)$). The `if'' part of the implication follows by  defining $D$ on $G$ by $D(w_1)=D_1(v_1)$ and $D(w_2)=D_2(v_2)-D_1(v_1)$ and $D(w_3)=d-D_2(v_2)$.
\end{proof}

We are now ready to construct our resolution of the identity map.
\begin{construction} \label{con: blowup} Take any  extension to a total order of the partial order defined above on the set of pairs consisting of a vine graph with an extremal bidegree, and denote this extension by $(G^1,D^1)<(G^2,D^2)<\ldots < (G^m,D^m)$. 

Define $J_i$ inductively as follows: $J_0=\overline{\mathcal{J}}(\phi^+)$ and
\[
J_{i} = \bl_{J_{G^{i},D^i, \alpha^i}}(J_{i-1})
\]
where $\alpha^i$ is the only \gf on $(G^i,D^i)$ with $L_{\alpha^i}=\emptyset$, and $J_{G,D,\alpha^i}$ is the strict transform of $J_{G^i,D^i}\subset \Jb{d}{g,n}(\phi^+)$. Following this, let $\jphitilde:= J_m$.

Similarly, let $G^i(P)$ be the same vine graph $G^i$ with an additional marked point $P$ on the vertex that is not $\leg_{G_i}(1)$, and denote by $D^i(P)$ and $\alpha^i(P)$ the obvious lifts.  Then define  $J_i(P)$ inductively, starting from $J_0(P)=\overline{\mathcal{J}}_{g,n+1}(\phi^+,P)$,
and then
\[
J_{i}(P) = \bl_{J_{G^{i}(P),D^i(P), \alpha^i(P)}}(J_{i-1}),
\]
and finally $\jphitildeP:= J_m(P)$.
\end{construction} 

The first point to observe is that this blowup does not depend upon the chosen extension to a total order.

\begin{proposition} \label{prop: transv}
 Let $(G_i,D_i)$, for $i=1,2$, be two pairs where $G_i$ is a vine graph and $D_i$ is an extremal bidegree such that $(G_1,D_1)$ and $(G_2, D_2)$ are incomparable. Set $\alpha_i$ to be the unique, empty, \gf on $(G_i,D_i)$ such that $L_{\alpha_i}=\emptyset$. If $f_1,f_2$ are morphisms $f_i\colon (G,D,\alpha)\to (G_i,D_i,\alpha_i)$ such that $L_\alpha$ contains 
 \[
 \{V\in \ext(G,D); \ V\subseteq f_1^{-1}(\{v_1\})\cap f_2^{-1}(\{v_2\}) \},
 \]
 then $f_1^*(E(G_1))\cap f_2^*(E(G_2))=\emptyset$.
\end{proposition}

\begin{proof}

Set $V_1:=f_1^{-1}(\{v_1\})$, $V_2:=f_2^{-1}(\{v_2\})$, and $V_0:=V_1\cap V_2$. We have then that 
\[
E(V_0, V_0^c)\subseteq E(V_1, V_1^c)\cup E(V_2, V_2^c) = f_{1}^*(E(G_1))\cup f_{2}^*(E(G_2)). 
\]

We observe that  $V_0\notin \ext(G,D)$. Otherwise we would  have that $|\alpha|\cap E(V_0, V_0^c)\neq \emptyset$ and then $|\alpha| \cap (f_{1}^*(E(G_1))\cup f_{2}^*(E(G_2))\neq \emptyset$. This would contradict the assumption that the morphisms $f_1$ and $f_2$ are compatible with the vine functions, and that $L_{\alpha_1}=L_{\alpha_2}=\emptyset$.

We now claim that $V_1\cup V_2=V(G)$.    This follows by applying Corollaries \ref{cor:V1V2_union}, \ref{cor:V1V2_intersection} combined with the facts that  $V_0\notin \ext(G,D)$, and $V_1 \not\subseteq V_2$ and $V_2 \not\subseteq V_1$ (the last two because $(G_1,D_1)$ and $(G_2, D_2)$ are by assumption incomparable).

By Proposition \ref{prop:extremal_cap_cup}, we  have that 
\[
E(V_1\setminus V_2, V_2\setminus V_1)=\emptyset,
\]
By using the above, we deduce
\begin{align*}
f_1^*(E(G_1)) &= E\left(V_1, V(G) \setminus V_1)\right)\\
              & = E(V_1, V_2\setminus V_1)\\
              & = E(V_0, V_2\setminus V_1)\\
              &\subseteq E\left(V_2, V_2\right),
\end{align*}
 from which we conclude that $f_1^*(E(G_1))\cap f_2^*(E(G_2))=\emptyset$.

\end{proof}

We immediately deduce:
\begin{corollary}
The blowup $\jphitilde \to \Jb{d}{g,n}(\phi^+)$ is independent of the chosen extension to a total order. (It only depends on the partial order). The same is true of  $\jphitildeP$.
\end{corollary}
\begin{proof} By Proposition~\ref{prop: transv}, if two vine graphs are incomparable under the partial order, their intersection is transversal, hence swapping the order of the two blowups does not change the result. \end{proof}

We let $\catJtilde$ be the category whose objects are $(G, D, \alpha)$, where $\alpha$ is a \gf with $L_\alpha=\ext(G,D)$. The morphisms of $\catJtilde$ are given in Definition~\ref{def:good_function_compatible}.

\begin{remark} \label{Rem: alpha-gives-blowup}
    The category $\catJtilde$ defined in \ref{Def: Jcategories} is the category obtained by blowing up $\cat_{g,n}(\phi^+)$  (as in Definition~\ref{def: blowup cat}) at $(G^1, D^1)$, then at $(G^2, D^2)$, \ldots, and finally at $(G^m, D^m)$. The case $m=1$ was discussed in Example~\ref{oneblowup}, and the general case follows in the same way.  The category $\catJtildeE$ is the subcategory of $\catJtilde$ generated by the exceptional divisors only.

The category $\catJtildeY$ is the subcategory of the stratification category of the stack $\jphitildeP$ (same as $\catJtilde$, but with an extra marking), whose elements are the intersection of the components over the exceptional divisors that do not contain the first marking.
\end{remark}

Our main result in this section is then

\begin{theorem}  \label{Thm: blowup} The stack $\jphitilde$ is nonsingular and the category $\catJtilde=\catJtilde(\phi^+,\phi^-)$ is its blowup stratification from $\Jb{d}{g,n}(\phi^+)$. The same result holds for $\jphitildeP$. Moreover, the forgetful morphism $\jphitildeP \to \jphitilde$ is the quasistable modification of the universal curve.
\end{theorem}

This follows as a combination of results in Subsection~\ref{Sec: blowup} and Subsection~\ref{Sec:combowc}.

\begin{proof}
 To prove the first statement we apply Proposition~\ref{prop :blowupstrat} to the category $\catJ$ of Example~\ref{Ex: stratJ}. The fact that each stratum $(G^i, D^i, \alpha^i)$ has transversal self intersection in $J_{i-1}$ is Proposition~\ref{prop: transversal-self} (see also Remark~\ref{Rem: alpha-gives-blowup}).

The second part follows from Theorem~\ref{prop: univ-semistable-mod}, and the fact that blowup commutes with flat base change.
\end{proof}

\begin{remark} 
Note that the vine curve strata indexed by $G^i$ that are part of the datum of our blowup,  do not necessarily have themselves transversal self-intersection (see Example~\ref{notransversalself}) in $\Mmb {g}{n}$, so the procedure of Subsection~\ref{Sec: blowup} cannot be applied to blowup the strata $G^1, \ldots, G^m$ in $\Mmb{g}{n}$ to produce a nonsingular DM stack with a stratification.

Moreover, similarly to Example~\ref{notransversalself}, one can also see that the strata $(G^i,D^i)$ do not themselves have transversal self-intersection in $\Jmb{d}{g}{n}(\phi^+)$. In our construction each stratum $(G^i, D^i)$ only acquires a transversal self-intersection  once lifted to a stratum of $J_{i-1}$ by means of the function $\alpha^i$.
\end{remark}

Let $Y'$ be the Cartier divisor in $\jphitildeP$ given by the sum of all strata that correspond to $(G,D,\alpha)$, where $(G,D)$ is a  vine graph:
\[
Y'=\sum_{i=1}^m J'_{G^i,D^i,\alpha^i}.
\]

Let $\mathcal{L}$ be the sheaf in $\jphitildeP$ obtained by pulling back a tautological sheaf in $\Jb{d}{g,n+1}(\phi^+,P)$ (see Theorem~\ref{prop: univ-semistable-mod}). We have the following result.

\begin{theorem} \label{Thm: resolution}
The line bundle $\mathcal{L}(-Y')$ is $\phi^-$-stable. In particular, the stack $\jphitilde$ comes with two  morphisms  that resolve the identity map $\Jb{d}{g,n} (\phi^+)\dashrightarrow \Jb{d}{g,n} (\phi^-)$. The first is the blowdown morphism (also defined by $\mathcal{L}$), and the second is the morphism defined by $\mathcal{L}(-Y')$.
\end{theorem}
\begin{proof}
By Proposition~\ref{Prop: degen-on-vine} and Remark~\ref{remark-deg}, it is enough to check that $\mathcal{L}(-Y')$ is $\phi^-$ stable on all vine curves.

This follows from  Construction~\ref{con: blowup}. The divisor $Y'\subset  \jphitildeP$ is supported on the  strata $(G^i, D^i, \alpha^i)$ of $\Jb{d}{g,n}(\phi^+)$, which are exactly the vine curves where $\mathcal{L}$ fails to be $\phi^+$-stable. Moreover, over each $J'_{G^i, D^i}$, the divisor $Y'$ fiberwise intersects its complement at $t_i=|E(G^i)|$ points. Thus tensoring by $\mathcal{O}(-Y')$ has the effect of modifying the bidegree of $\mathcal{L}$ on each stratum $(G^i,D^i, \alpha^i)$ by adding the bidegree $(-t_i, +t_i)$ (where the first element of the pair is the degree on the component of the vine curve that contains the first marking).  Therefore, because the bidegree $D^i$ is extremal on $G^i$, the line bundle $\mathcal{L}(-Y')$ is $\phi^-$-stable on $(G^i,D^i,\alpha^i)$ for all $i=1,\ldots,m$.
\end{proof}


We conclude with the following observation.

\begin{corollary} \label{Cor:Y snc}
 The Cartier divisor $Y' \subset \jphitildeP$ is simple normal crossing, and the stratification category it generates (as in Example~\ref{ex: snc}) is $\catJtildeY$.
\end{corollary}
\begin{proof}
     The first part follows from Proposition~\ref{prop:simple_vine_curve_crossing}. The second part follows directly from the definition of $Y'$.
\end{proof}

\section{Wall--Crossing Formulas}

\label{Sec: wc formula}
Let $\phi^+, \phi^- \in V_{g,n}^d$ be   on opposite sides of a stability hyperplane $H$ (Definition~\ref{Def: opposite sides}). In this section we find a formula for the wall crossing along $H$ of Brill-Noether classes in terms of pushforward of boundary strata classes. We first give a formula in Theorem~\ref{maintheorem} on the nonsingular resolution $\jphitilde$ of the identity map \[\mathsf{Id} \colon \Jb{d}{g,n} (\phi^+)\dashrightarrow \Jb{d}{g,n} (\phi^-)\] that we defined in Section~\ref{Sec: blowup construction}. Then we write a second formula in Corollary~\ref{differenceinJbar} by taking the pushforward of that difference along the blowdown morphism $p \colon \jphitilde \to \Jb{d}{g,n} (\phi^+)$.

The universal quasistable family $ \jphitildeP \to \jphitilde$ carries two line bundles: the pullback $\mathcal{L}$ of a tautological line bundle on $\Jb{d}{g,n} (\phi^+)$ and its modification $\mathcal{L}(-Y')$, the pullback of a tautological line bundle on $\Jb{d}{g,n} (\phi^-)$ (see Theorem~\ref{Thm: resolution}).

Our main result in Theorem~\ref{maintheorem} is a formula for the difference of the total Chern classes of the derived pushforward of $-\mathcal{L}(-Y')$ and that of $-\mathcal{L}$ as an explicit pushforward of classes supported on the boundary. Because on the (unprimed) ``resolved'' strata the normal bundles split as a direct sum of line bundles, our formula is better written on the ``resolved'' strata instead of the embedded ones. 

Before stating the main results, let us fix some notation, for Theorem~\ref{maintheorem} and for Corollary~\ref{differenceinJbar}.

For each pair $(G,D)\in \mathfrak{C}_{g,n}(\phi^+)$, denote by $\pi_{G,D}\col \mathcal{C}_{G,D}\to \mathcal{J}_{G,D}$ the pullback to $\mathcal{J}_{G,D}$ of the universal quasistable family $\Jmb{d}{g}{n+1}(\phi^+;P) \to \Jmb{d}{g}{n}(\phi^+)$ . The total space $\mathcal{C}_{\mathcal{J}_{G,D}}$ has one irreducilbe component $\mathcal{C}_v^+:=\mathcal{C}_{G,D,v}$ for each vertex $v$ of $G$. We denote by $\pi_v^+:=\pi_{G,D,v}\col\mathcal{C}_v^+\to \mathcal{J}_{G,D}$ the induced map. Also, for each $V\subset V(G)$, we denote by $\pi_V^+\col \bigcup_{v\in V}\mathcal{C}_v^+\to \mathcal{J}_{G,D}$ the induced map on the union. We write $X^+=X_{G,D}^+:=\mathcal{C}_{\leg(1)}^+$ 
and $\Sigma^+ = X^+\cap \mathcal{C}_{\{\leg(1)\}^c}^+$. We also write $Y_V^+=\mathcal{C}^+_{V^c}$ for every $V\subset V(G)$.

    We can extend these notations to $\jphitilde$. Let us recall the geometry of $\jphitilde$ from Section~\ref{Sec: blowup construction}. For each $1 \leq i \leq m$ set $\beta_i:=(G_i,D_i)$ to be the vine strata from Construction~\ref{con: blowup}, so  $\mathcal{J}_{\beta_i}$ are the strata of $\Jb{d}{g,n} (\phi^+)$ whose strict transforms of the images are blown-up, in the given order, in $\Jb{d}{g,n}(\phi^+)$, to obtain $\jphitilde$. Recall that $\catJtildeE$ is the category of the (resolutions of the closed) strata of $\jphitilde$ that are in the intersection of the exceptional divisors $E_i'$. 
    
Objects are triples $(G,D,V_{\bullet})$ for $(G,D) \in \catJ$. Each stratum $\widetilde{\mathcal{J}}_{G,D,V_\bullet}$ admits a forgetful morphism $p_{G,D,V_\bullet}$ to $\mathcal{J}_{G,D}$, and we define $\pi_v$, $\pi_V$, $X$, $\Sigma$ and  $Y_V$  as the pullbacks via $p_{G,D,V_\bullet}$ of the corresponding items defined in the previous paragraph for $(G,D)$. Also, we set $Y_{G,D,V_\bullet}:=\bigcap_{V\in V_\bullet} Y_V$, note that, by Proposition~\ref{prop:extremal_cap_cup},  $Y_{G,D,V_\bullet}$ is nonempty if and only if $V_\bullet$ is a chain (as in the definition of $\catJtildeY$ in Definition~\ref{Def: Jcategories}), and in that case, we have that $Y_{G,D,V_\bullet} = Y_{\max(V_\bullet)}$.

 Also, for some triple $(G,D, V_\bullet)$, we define
\begin{equation} \small \label{defH}
F^X_+:=- R^{\bullet} (\pi_X^+)_\ast \mathcal{L}^+(-\Sigma^+)_{|X^+}; \  F_V^+:=- R^{\bullet} (\pi_{V^c})_\ast \mathcal{L}^+_{|Y_V^+}; \
H_V^+:=F_V^+-\sum_{\substack{V' \in V_{\bullet},\\ V' \gtrdot V}} F_{V'}^+
\end{equation}
and
\begin{equation} \small \label{defHtilde}
{F}^X:=- R^{\bullet} (\pi_X)_\ast \mathcal{L}(-\widetilde{\Sigma})_{|X}; \ \ {F}_V:=- R^{\bullet} (\pi_{V^c})_\ast \mathcal{L}_{|Y_V}; \ \ \textrm{and} \quad
{H}_V:= F_V-\sum_{\substack{V' \in V_{\bullet},\\ V' \gtrdot V}}  F_{V'}.
\end{equation}
 Note also that $\mathcal{L}$ is the pullback of $\mathcal{L}^+$, hence ${F}^X$, ${F}_V$ and ${H}_V$ are the pullback via $p_{G,D,V}$ of $F^X_+$, $F_V^+$ and $H_V^+$ respectively.

Let $E_i$  be the exceptional stratum of the blowup morphism $\jphitilde\to \Jb{d}{g,n} (\phi^+)$, so that $E_i \to E_i' \subset \jphitilde$ is the exceptional divisor, and each $E_i'$ is contracted to $\mathcal{J}'_{\beta_i}$. Following the notation in the previous paragraph, we let then ${X}_i' \cup {Y}_i'$ denote the two irreducible components of the restriction to $E_i'$ of the universal quasistable family $\jphitildeP \to \jphitilde$, where ${X}_i'$ is the component containing the first marked point and ${Y}_i'$ is the other component, and denote by ${X}_i$ and ${Y}_i$ the base change to $E_i$ of ${X}_i'$ and ${Y}_i'$. Recall that the divisor $Y'$ in Theorem~\ref{Thm: resolution} is precisely $\sum_i {Y}_i'$.

We now define  psi-classes following \eqref{def: psi}. Each edge $e \in E(G)$ defines a morphism  $f_e \colon \mathcal{J}_{G,D} \to \mathcal{J}_{G',D'}$ to some codimension one stratum $(G',D')$, and we set 
\begin{equation}
   \Psi_{G',D'}:= -c_1\left(N_{\mathcal{J}_{G',D'}} \Jmb dgn(\phi^+)\right), \quad \Psi_{(G,D,e)}:= f_e^*(\Psi_{G',D'}).
\end{equation}
In Remark~\ref{relate-psi} we will discuss how these compare to the usual psi-classes on $\Mmb gn$.

Similarly, for a triple $(G,D,V_\bullet)$ in $\catJtildeE$, we have that $S_{(G,D,V_\bullet)}=V_\bullet$ (recall the definition of $S_{(G,D,V_\bullet)}$ in Subsection~\ref{subsec:normal_strat} Item~(4)) and each $V \in V_{\bullet}$ defines a morphism $f_{G,D, V} \colon  \jtilde_{G,D, V_\bullet} \to E_i$ for some $i=1,\ldots,m$. As in  Definition~\ref{def: ncs} Item~(3), we set $\LL_V:=f_{G,D,V}^*N_{E_i}(\jphitilde$ and define the  psi-classes \[ \Psi_i:= -c_1 \left(N_{E_i}(\jphitilde \right), \quad \Psi_{G,D, V}:= f_{G,D,V}^*(\Psi_i)=-c_1(\mathbb{L}_V)\]  on $E_i$ and on $\jtilde_{G,D,V_{\bullet}} $ respectively.

Finally, define the coefficient
\[
b_{G,D,V}((j_V)_{V \in V_\bullet}; (k_V)_{V \in V_{\bullet}}):= - \binom{k_V+ g_V-d_V-\sum_{V' \geq V} j_{V'}+(k_{V'}+1)}{k_{V}+1} 
\]
for each vectors $(j_V\geq 0)_{V \in V_{\bullet}}$ and $(k_V\geq 0)_{V \in V_{\bullet}}$  of  nonnegative integers.

\begin{theorem} \label{maintheorem} The difference of total Chern classes 
\begin{equation} \label{minus}
c_t(-R^{\bullet} \pi_\ast \mathcal{L})-c_t(-R^{\bullet} \pi_\ast \mathcal{L}(-Y'))
\end{equation}
in $A^\ast \left(\jphitilde\right)$ equals the boundary class
\begin{equation} \small \label{mainformula}
    \sum_{\substack{\Gamma=(G,D,V_{\bullet}) \in \\ \catJtildeE \setminus\{\bullet\}}} \frac{-f_{\Gamma  \ast}}{\left| \Aut(\Gamma)\right|} \Bigg(  \sum_{\substack{(j_V \geq 0)_{V \in V_{\bullet}}, \\(k_V\geq 0)_{V \in V_{\bullet}} \\ s\geq 0}}   c_s(\widetilde{F}^X) \cdot \prod_{V \in V_{\bullet}}  b_{G,D, V}((j_V)_V, (k_V)_V)  \cdot c_{j_V}(\widetilde{H}_V) \cdot \Psi_V^{k_V}   \Bigg),
\end{equation}
where the sum runs over all resolved strata $\Gamma \in \catJtildeE$  (intersection of exceptional divisors) of $\jphitilde$ (see Section~\ref{Sec: blowup construction}), except the terminal object (the open stratum).
\end{theorem}

Note that Formula~\eqref{mainformula} gives a total Chern class from which one can immediately deduce the difference of the Brill--Noether classes on $\jphitilde$. This is a more direct formula than e.g. the main result of \cite{prvZ}, where an explicit formula is given for the Chern character, which then requires inversion to obtain the desired Chern class.
\begin{proof}
We start our calculation by making use of the short exact sequence
\begin{equation} \label{exact1}
0 \to \mathcal{L}(-Y') \to \mathcal{L} \to \mathcal{L}|_{Y'} \to 0.
\end{equation}
on the quasistable family $\jphitildeP \to \jphitilde$. For  convenience, define 
\[
F:=-R^\bullet\pi_\ast(\mathcal{L}), \quad \widetilde{F}:= R^\bullet\pi_\ast(\mathcal{L}_{| Y'}).
\]
We then apply Whitney's formula
\begin{equation} \label{whitney}
c_t(-R^\bullet\pi_\ast(\mathcal{L}(-Y')))- c_t(-R^\bullet\pi_\ast(\mathcal{L}))= (c_t(\widetilde{F}) - 1)\cdot c_t(F)
\end{equation}
for the total Chern class of the three terms in  \eqref{exact1}. This computes the opposite of \eqref{minus}. From now on, we will mostly work on the term $c_t(\widetilde{F})-1$. 

For each $\Gamma' \in \catJtildeE$ we let 
\[
F_{\Gamma'}^Y:=-R^\bullet\pi_\ast \mathcal{L}_{| Y_{\Gamma'}}.
\]
We now apply the following:
\begin{lemma}  \label{lem: incl-excl} (Inclusion-exclusion principle for a simple normal crossing stratification.)
Let $\mathcal{D}$ be a simple normal crossing divisor in $X$, and let $\cat$ be its category of  strata. 
Then the following equality holds in the $K$-theory of $X$:
\[
\mathcal{L}|_\mathcal{D}=\sum_{\alpha\in \cat}(-1)^{\codim(\alpha)-1}\mathcal{L}|_{\mathcal{D}_\alpha}
\]
\end{lemma}

\begin{proof} 
It is enough to prove the statement for the case of the structure sheaf $\mathcal{L}=\mathcal{O}$. Assuming that $\mathcal{D}=\mathcal{D}_1+ \mathcal{D}_2$, we have the short exact sequences
\begin{gather*}
0 \to \mathcal{O}(-\mathcal{D}_1-\mathcal{D}_2)\to \mathcal{O}\to \mathcal{O}_{\mathcal{D}_1+\mathcal{D}_2} \to  0 \\
0 \to \mathcal{O}(-\mathcal{D}_1-\mathcal{D}_2)\to \mathcal{O}(-\mathcal{D}_2)\to \mathcal{O}_{\mathcal{D}_1}(-\mathcal{D}_2) \to 0 \\
0 \to \mathcal{O}(-\mathcal{D}_2)\to \mathcal{O}\to \mathcal{O}_{\mathcal{D}_2}\to 0\\
0\to \mathcal{O}_{\mathcal{D}_1}(-\mathcal{D}_2)\to \mathcal{O}_{\mathcal{D}_1}\to \mathcal{O}_{\mathcal{D}_1\cap \mathcal{D}_2} \to 0.
\end{gather*}
By combining these, we obtain the equality
\[
\mathcal{O}_{\mathcal{D}_1+\mathcal{D}_2} = \mathcal{O}_{\mathcal{D}_1} + \mathcal{O}_{\mathcal{D}_2} - \mathcal{O}_{\mathcal{D}_1 \cap \mathcal{D}_2}.
\]
The statement is then obtained by repeatedly decomposing $\mathcal{D}$ until all summands are irreducible. 
\end{proof}

By combining Lemma~\ref{lem: incl-excl} with Lemma~\ref{Cor:Y snc} (the fact that the $Y_i'$ are indeed simple normal crossing), and the multiplicativity of the total Chern class, together with the fact that $Y_{\Gamma} \to Y_{\Gamma}'$ is \'etale of degree $|\Aut(\Gamma)|$ (by Corollary~\ref{Cor:Y snc}), 
we obtain

\begin{equation} \label{step1}
    c_t(F) \cdot (-1+c_t(\widetilde{F})) =c_t(F) \cdot \left(-1+ \prod_{\Gamma' \in \catJtildeY \setminus \{\bullet\}} c_t\left({(-1)^{\codim \Gamma'}} \frac{f_{\Gamma'  \ast}  F_{\Gamma'}^Y}{|\Aut(\Gamma')|}\right)\right)
\end{equation}
where $\catJtildeY\subseteq \catJtildeE$ is the image in $\jphitilde$ of the  stratification induced by $Y_1', \ldots, Y_m'$ (and in the product we have removed its terminal object),  and 
\[
f_{\Gamma'}\colon \widetilde{\mathcal{J}}_{\Gamma'} \to\jphitilde=: \widetilde{\mathcal{J}}
\]
is the (resolution of the closed) stratum $\widetilde{\mathcal{J}}'_{\Gamma'}$.

For a fixed $\Gamma'$, we now aim to write each factor of the product in the RHS of \eqref{step1} as a pushforward via the corresponding  stratum. 
We apply Formula~\eqref{prop: Grothendieck-Riemann-Roch} (Grothendieck-Riemann-Roch for the total Chern class)
to obtain that each factor 

\[c_t\left({(-1)^{\codim \Gamma'}} \frac{f_{\Gamma'  \ast}  F_{\Gamma'}^Y}{|\Aut(\Gamma')|}\right)\] equals 
\begin{equation} \label{step2} \small
   1+ \sum_{\substack{\Gamma \in \catJtildeE,\\ k \geq 1}}\frac{f_{\Gamma \ast}}{\left|\Aut(\Gamma)\right|} \left(\sum_{\substack{\{f_1, \ldots, f_k\} \in \\ \sInt( (f_{\Gamma})^k)_{f_{\Gamma'}} }} \frac{\prod_{j=1}^k \left( f_j^* c_t(\left(\bigwedge^{\bullet}N^{\vee}_{\jtilde_{\Gamma}} \jtilde  \otimes (-1)^{\codim(\Gamma')} F_{\Gamma'}^Y\right)-1\right)}{c_{\text{top}} N_{\jtilde_{\Gamma}} \jtilde  }\right). 
\end{equation}
Note that  intersections in $\catJtildeY$  are not necessarily objects of the latter category, so the product is  taken over  $\catJtildeE \supseteq  \catJtildeY$.

After that, to continue our derivation from Formula~\eqref{step1}, we aim to calculate the product of the terms  in \eqref{step2} for varying $\Gamma'$. 
We  apply the excess intersection formula Proposition~\ref{prop:intersection} to the product~\eqref{step1}, to deduce that it equals 
 \begin{multline}
\label{step3}
 \sum_{\Gamma \in \catJtildeE \setminus \{\bullet\}} \frac{1}{\left|\Aut(\Gamma)\right|}\\ f_{\Gamma \ast} \Bigg(\sum_{\substack{\Gamma_1, \ldots, \Gamma_k  \in \catJtildeY\\ Q_t \subseteq \oMor(\Gamma, \Gamma_t)  \textrm{ for all } t=1, \ldots, k \\ \textrm{such that } \cup_t Q_t \textrm{ is generic} \\}}  
 \frac{c_t(F_{\Gamma}) \cdot  \prod_{\substack{j=1, \ldots, k,\\ f \in Q_j}} \left(f^* c_t(\left(\bigwedge^{\bullet}N^{\vee}_{\jtilde_{\Gamma_j}} \jtilde  \otimes (-1)^{\codim(\Gamma_j)} F_{\Gamma_j}^Y\right)-1\right)}{c_{\text{top}} N_{\jtilde_{\Gamma}} \jtilde  }\Bigg).
\end{multline}

Now we focus on simplifying the term inside the pushforward $f_{\Gamma \ast}$.  Following Proposition~\ref{prop:bijection_parts}.for fixed $\Gamma=(G, D, V_{\bullet})$, there is a natural bijection between the set of morphisms
\[
\left\{ \{Q_t \subseteq \Aut(\Gamma)\backslash\Mor(\Gamma, \Gamma_t)\}_{t=1, \ldots, k} \textrm{ for } \Gamma_1, \ldots, \Gamma_k \in \catJtildeY \text{ s.t. } \cup_t Q_t \textrm{ is generic}\right\} 
\]
and the set
\[
\left\{ \{\ell_1, \ldots, \ell_M \} \subseteq \chains(V_{\bullet}) \textrm{ such that } V_{\bullet} = \cup_{i=1}^M \ell_i  \right\}
\]
with $M=|Q_1| + \ldots + |Q_k|$, given by 
\[
\{Q_1, \ldots, Q_k\} \mapsto \bigcup_{t=1}^k \{f^*(V_{\Gamma_j,\bullet})\}_{f \in Q_j}.
\]

Moreover, if $f_{\ell} \colon \Gamma \to \Gamma_t$ is a contraction that corresponds to the chain $\ell \subseteq V_{\bullet}$, then $f_{\ell}^*(F^Y_{\Gamma_t})=F_{\Gamma, \max(\ell)}$ (in particular, the latter only depends on $\max(\ell) \in V_{\bullet}$, and not on the whole chain). Furthermore, the pullback $f_{\ell}^*(N_{\jtilde_{ \Gamma_t}} \jtilde)$ equals a direct sum of line bundles, which allows us to expand the wedge product
\[\bigwedge^{\bullet} f_{\ell}^*(N_{\jtilde_{ \Gamma_t}} \jtilde)=\bigwedge^{\bullet} \bigoplus_{V \in \ell} \LL_V= \sum_{S \subseteq \ell}  (-1)^{|S|} \bigotimes_{V \in S} \LL_V. \]

In light of this, we  rewrite the numerator inside the pushforward via $f_{\Gamma}$ in \eqref{step3} as
\[ \small
  \sum_{\substack{\Gamma_1, \ldots, \Gamma_k  \in \catJtildeY\\ Q_t \subseteq \oMor(\Gamma, \Gamma_t),\ t=1, \ldots, k \\ \textrm{s.t. } \cup_t Q_t \textrm{ is generic} \\}} { c_t(F_{\Gamma}) \cdot \prod_{\substack{j=1, \ldots, k,\\ f \in Q_j}} \left(f^* c_t\left(\bigwedge^{\bullet}N^{\vee}_{\jtilde_{\Gamma_j}} \jtilde  \otimes (-1)^{\codim \Gamma_j}F_{\Gamma_j}^Y\right)-1\right)},\] which equals \begin{equation} \label{step4}
   c_t(F_{\Gamma}) \cdot \sum_{\substack{\{\ell_1, \ldots, \ell_M\} \subseteq  \\ \chains(V_{\bullet})  \textrm{ s.t. } \\ \ell_1 \cup \ldots \cup \ell_M= V_{\bullet}}}\prod_{i=1}^M \prod_{S \subseteq \ell_i} \left( c_t\left((-1)^{|S|} \bigotimes_{V \in \ell_i} \LL_V^\vee  \otimes (-1)^{|\ell_i|} F_{\Gamma, \max(\ell_i)}\right)-1\right).
  \end{equation}
Next, we apply the inclusion-exclusion principle in the form
\[
\sum_{\substack{\{\ell_1, \ldots, \ell_M\} \subseteq \chains(V_{\bullet})}}\varphi(\ell_1, \ldots, \ell_M)=\sum_{K \subseteq V_{\bullet}}  \sum_{\substack{\{\ell_1, \ldots, \ell_M\} \subseteq \\ \chains(V_{\bullet}) \textrm{ s.t. } \\ \ell_1 \cup \ldots \cup \ell_M= V_{\bullet}\setminus K} } (-1)^{|K|} \varphi(\ell_1, \ldots, \ell_M)
\]
for any function $\varphi\colon \chains{V_{\bullet}} \to \mathbb{Z}$, to eliminate the condition that $\bigcup_{i=1}^M \ell_i=V_{\bullet}$ in the last set of indices of \eqref{step4}. We thus obtain that \eqref{step4} equals
\begin{equation} \label{step5}
    \sum_{K \subseteq V_{\bullet}} (-1)^{|K|}c_t(F_{\Gamma}) \cdot \prod_{S \in \chains(V_{\bullet}\setminus K)}  c_t \Bigg( \bigotimes_{V \in S} \LL_V^\vee  \otimes \sum_{\substack{\ell \in \chains(V_{\bullet} \setminus K)\\ \textrm{such that } S \subseteq \ell}} (-1)^{|S|+|\ell|} F_{\Gamma, \max(\ell)}\Bigg)
\end{equation}
We now apply Lemma~\ref{lem:sum_tree} to simplify \eqref{step5}, so it becomes

\begin{equation} \label{step6} \small
    \sum_{K \subseteq V_{\bullet}} (-1)^{|K|}  c_t\left( F_X\right) \cdot \prod_{V_0 \in V_{\bullet} \setminus K} c_t \Bigg( \bigotimes_{\substack{V \in V_\bullet\setminus K \\ V\leq V_0}} \LL_V^\vee  \otimes H_{K, V_0}\Bigg).
\end{equation}

After all these simplifications, we now go back and replace \eqref{step6} as the numerator of the term in \eqref{step3} that is pushed forward via $f_{\Gamma}$, to obtain that \eqref{step3} equals
 \begin{multline} 
\label{step7}
 \sum_{\Gamma \in \catJtildeE \setminus \{\bullet\}} \frac{f_{\Gamma \ast}}{\left|\Aut(\Gamma)\right|} \\ \Bigg( 
 \frac{    \sum_{K \subseteq V_{\bullet}} (-1)^{|K|}  c_t\left( F_X\right) \cdot \prod_{V_0 \in V_{\bullet} \setminus K} c_t \left( \bigotimes_{V \leq V_0} \LL_V^\vee  \otimes H_{K, V_0}\right)}{c_{\text{top}} N_{\jtilde_{\Gamma}} \jtilde  }\Bigg).
\end{multline}

Our final step to conclude  repeatedly uses Formula~\eqref{tensorbylinebundle} for the total Chern class of the tensor product of a K-theory element times a line bundle, and then divide by
\begin{equation} \label{ctop}
c_{\text{top}} N_{\jtilde_\Gamma} \jtilde  = \prod_{V \in V_{\bullet}} -\Psi_V.
\end{equation}
After combining the binomial coefficients by means of Vandermonde's identity, we obtain that Formula~\eqref{step7} equals the final formula~\eqref{mainformula}. (One way to obtain the formula is to consider only the case $K=\emptyset$ in \eqref{step6}, then expanding as a polynomial in $\{\Psi_V\}_{V \in V_{\bullet}}$, and considering only the monomial containing $\prod_{V \in V_{\bullet}} \Psi_V^{a_V}$ for all $a_V~\geq~1$, and then lowering the exponents $a_V$ by one because of the division by the term in \eqref{ctop}).
\end{proof}

We now prove the ancillary results used in the proof of Theorem~\ref{maintheorem}.

\begin{lemma}
\label{lem:sum_tree}
Let $V_\bullet$ be a rooted forest, and let $(x_V)_{V \in V_{\bullet}}$ be formal variables. Let $S\subseteq V_\bullet$ be a chain in $V_\bullet$. Then
\[
\sum_{\substack{\ell\in \chains(V_\bullet)\\ \text{such that } S\subseteq \ell}} (-1)^{|S|+|\ell|}x_{\max(\ell)} \] equals \begin{enumerate}
    \item 
$-\sum_{V\in \min(V_\bullet)} x_V$, if $S=\emptyset$ \item
$0$,  if there is $V\in V_\bullet\setminus S$ such that $V< \max(S)$  and $S\cup\{V\}$ is a chain, and
\item $
x_{\max(S)}- \sum_{\substack{V\in V_\bullet \\ V\gtrdot \max(S)}} x_V$, in all other cases.
\end{enumerate}
\end{lemma}
\begin{proof}
Assume that $S$ is nonempty and that there exists $V\in V_\bullet\setminus S$ such that $V<\max(S)$ and $S\cup\{V\}$ still is a chain. Then we can write
\begin{multline*}
    \sum_{\substack{C\in \chains(V_\bullet)\\ S\subset C}} (-1)^{|S|+|C|-1}x_{\max(C)} = \\ = \sum_{\substack{C\in \chains(V_\bullet)\\ S\subset C, V\notin C}} (-1)^{|S|+|C|-1}x_{\max(C)} + (-1)^{|S|+|C\cup\{ V\}|-1}x_{\max(C\cup \{V\})} 
\end{multline*}
since $\max(C)=\max(C\cup\{V\})$, we have that the sum is $0$.

Assume that $S$ is nonempty and denote by $\max(S) = V_0$. Also assume that $S =\{ V\in V_\bullet: V\leq V_0\} $. Then we can write
\[
\sum_{\substack{C\in \chains(V_\bullet)\\ S\subseteq C}} (-1)^{|S|+|C|-1}x_{\max(C)} = \sum_{V\in V_\bullet}x_V\sum_{\substack{C\in \chains(V_\bullet)\\ S\subseteq C, \max(C) = V}} (-1)^{|S|+|C|-1}
\]
If $V=V_0$, then the condition $S\subseteq C$ and $\max(C) = V_0$ is equivalent to $C=S$, so 
\[
\sum_{\substack{C\in \chains(V_\bullet)\\ S\subseteq C, \max(C) = V}} (-1)^{|S|+|C|-1} = -1.
\]

If $V\gtrdot V_0$, then the condition $S\subset C$ and $\max(C) = V_0$ is equivalent to $C=S\cup\{V\}$, so
\[
\sum_{\substack{C\in \chains(V_\bullet)\\ S\subseteq C, \max(C) = V}} (-1)^{|S|+|C|-1} = 1.
\]

If $V < V_0$, then the sum is empty, and so it is $0$.

If $V\gg V_0$, choose $V'$ such that $V_0< V'< V$, and then
\[
\sum_{\substack{C\in \chains(V_\bullet)\\ S\subseteq C, \max(C) = V}} (-1)^{|S|+|C|-1} = \\ 
\sum_{\substack{C\in \chains(V_\bullet), \\ S\subseteq C, \max(C) = V \\ V'\notin C}} (-1)^{|S|+|C|-1} + (-1)^{|S|+|C\cup\{V'\}|-1}
\]
which equals $0$.

The case $S=\emptyset$ is similar.
\end{proof}

Our next task is to take the pushforward of Formula~\eqref{mainformula} via the blowdown morphism $p \colon \jphitilde \to \Jb{d}{g,n} (\phi^+)$, to produce an explicit graph formula for the difference of the Brill--Noether classes. Recall the notation for the strata $\mathcal{J}_{G,D}$ that was set in the beginning of this section.


For $V \in V_{\bullet}$, we define the ``close upper edges'' and ``far upper edges'' as
\begin{equation} \label{def: CU FU}
\cu(V):= E(V, \nex(V) \setminus V) \ \textrm{and} \ \fu(V) := E(V, \nex(V)^c)
\end{equation}
so there is a decomposition
\[
E(V,V^c)= \cu(V) \sqcup \fu(V).
\]

 For every collection $(g_V)_{V \in V_{\bullet}}$ of  nonnegative integers, define the class \[c_{(G,D, V_{\bullet})}((g_V)_{V \in V_{\bullet}})\in A^{\bullet}(\overline{\mathcal{J}}_{(G,D)})\] to equal 
\begin{multline}
\label{pushofpsi}
\sum_{\substack{(a_{e,V})_{V \in V_{\bullet}, e \in \fu(V)} \\ (g_{e,V})_{V \in V_{\bullet}, e \in \cu(V)} \\ \textrm{such that, for all }V \in V_{\bullet}, \\  \sum (g_{e,V} +1) - \sum a_{e,V} = g_V+1}}  \prod_{e \in E(G) } \Psi_{(G,D,e)}^{g_{e, k(e)- \sum_{V \in S(e)}}a_{e,V}} \cdot \\ \cdot \prod_{V \in S(e)}(-1)^{a_{e,V}} \binom{g_{e, k(e)} - \sum_{\substack{V' \in S(e) \\ V \subsetneq V'}}, a_{e,V'} }{a_{e,V}}
\end{multline} 
where, the  $a_{e,V}$ and $g_{e,V}$ vary over the nonnegative integers, and for $e \in E(G)$, we let $k(e) \in V_{\bullet}$ be the unique  (by Proposition~\ref{prop:properties_next}) element such that $e \in \cu(k(e))$, and we let $S(e):=\{V\in V_{\bullet} : e \in \fu(V)\}$.

For a given specialization $h \colon (G',D', V'_{\bullet})\to (G,D,V_{\bullet})$, we can define the pullback  $h^*((g_V))~=~h^*((g_V)_{V \in V_{\bullet}})_{V' \in V'_{\bullet}}$ by:
\begin{equation} \label{PB gv}
h^*((g_V)_{V \in V_{\bullet}})_{V'}:= \begin{cases} g_V & \textrm{ if } V'=h^{-1}(V) \textrm{ for some } V \in V_{\bullet}; \\ -1 & \textrm{ otherwise.} \end{cases}
\end{equation}

We have then the following pushforward result.

\begin{proposition} \label{pushpsi}
 The following pushforward formula holds
\begin{multline}
    p_* \left( \frac{f_{(G,D,V_{\bullet})*}}{\left|\Aut(G,D,V_{\bullet})\right|}\left( \prod_{V \in V_{\bullet}}  \Psi_V^{g_V}\right)\right) =\\  \sum_{(G',D') \in \catJ} \frac{f_{(G',D') *}}{\left|\Aut(G',D')\right|}\Bigg(\sum_{\substack{V'_{\bullet} \text{ a \gc in } \ext(G',D'), \\ h \in \oMor((G', D', V'), (G,D,V))}} c_{(G',D',V_{\bullet}')}(h^*((g_V))) \Bigg)
\end{multline}
\end{proposition}

\begin{proof}
Follows from Corollary~\ref{prop: comparison-psi}.
\end{proof}
Our final step is to take the pushforward of our formula in Theorem~\ref{maintheorem} via the blowdown morphism $p \colon \jphitilde \to \Jmb dgn(\phi^+)$. Note first that the K-theory elements defined in \eqref{defH} are pullbacks via $p$ of similar classes, which we will denote with the same name in the next result. The pushforward via $p$ is then obtained by combining Theorem~\ref{maintheorem} and Proposition~\ref{pushpsi}.
\begin{corollary} \label{differenceinJbar}
The difference $\mathsf{w}_d(\phi^+)- \mathsf{Id}^*\mathsf w_d(\phi^-) $ in $A^{g-d} \big(\Jmb dgn(\phi^+)\big)$ equals

\begin{multline} \label{relation}
-\sum_{(G,D) \in \catJ} \frac{1}{|\Aut(G,D)|} f_{(G,D) \ast} \\ \Bigg( \sum_{\substack{  V_{\bullet} \textrm{ a \gc  in }  \ext(G,D), \\ s+\sum_{V} j_{V} + \sum g_{e} =g-d-|E(G)|}} \alpha(s,(j_V), (g_e)) \cdot  c_s(F^X_+) \cdot \prod_{V \in V_{\bullet}} c_{j_{V}} (H_{V}^+) \prod_{e} \Psi_{(G,D,e)}^{g_e}\Bigg),\end{multline}
where each coefficient $\alpha(s,(j_V)_{V \in V_{\bullet}}, (g_e)_{e \in E(G)})$ is defined by
\begin{multline} \small
\label{eq:alpha}
    \sum_{(a_{e,V})_{(e,V)}} (-1)^{|V_{\bullet}|} \prod_{e \in E(G)} \prod_{V \in S(e)} (-1)^{a_{e,V}}\binom{g_e + \sum_{\substack{V' \in S(e) \\ V' \subseteq V}}, a_{e,V'} }{a_{e,V}} \cdot \\ \small
     \prod_{V \in V_{\bullet}} \binom{\bigg(\sum_{V \subseteq V'} (\rk H_{V'}^+-j_{V'})\bigg) - \bigg(\sum_{e \in \operatorname{CN}(V)} (g_e+1 + \sum_{\substack{V' \subseteq V,\\e \in \fu(V')}} a_{e,V'})\bigg)} {\sum_{e \in \cu(V)}g_e+ 1+\sum_{\substack{e \in \cu(V) \\ V' \in S(e)}} a_{e, V'} - \sum_{e \in \fu(V)}a_{e,V} } ,
\end{multline}
where $\operatorname{CN}(V):=E(\nex(V), \nex(V)^c)$, each $a_{(e,V)}$ ranges over the integers, and the indices $(e,V)$ range over all $V \in V_{\bullet}$ and over all $e \in E(V, \fu(V))$.
\end{corollary}
 (Note that, because of the last binomial, the summand is zero except for finitely many natural number values of $a_{e,V}$, Also, note that $H^+_V$  depends on $V_{\bullet}$).
\begin{proof}
First observe that the result amounts to taking the degree $g-d$ part of the pushforward via $p$ of Formula~\eqref{mainformula}.

The calculation that we are attempting has the form
\begin{equation} \label{one}
p_* \left( \sum_{(G,D,V_{\bullet}) \in \catJtildeE} \frac{f_{(G,D,V_{\bullet}) *}}{\left|\Aut(G,D,V_{\bullet})\right|} \left(\sum_{(g_V)_{V \in V_{\bullet}}} p_{(G,D)}^* (\beta_{(g_V)_{V\in V_{\bullet}}}) \prod_{V \in V_{\bullet}} \Psi_V^{g_V}\right) \right).
\end{equation}
for suitable classes $\beta_{(g_V)_{V \in V_{\bullet}}} \in A^*(\overline{\mathcal{J}}_{(G,D)})$ as in \eqref{mainformula}. Since $F^X$ and $H_V$ are pullback of $F^X_+$ and $H_V^+$, by the push-pull formula, and by Proposition~\ref{pushpsi}, we obtain that \eqref{one} equals
\begin{multline}\label{two} 
 \sum_{(G',D') \in \catJ} \frac{f_{(G',D') *}} {\left|\Aut(G',D')\right|} \\ \Bigg(\sum_{\substack{V'_{\bullet} \textrm{ a \gc in } \ext(G',D'); \\ h \in \oMor((G',D',V'_{\bullet}),(G,D,V_{\bullet}))\\ (g_V \geq 0)_{V \in V_{\bullet}}}}   h^*\beta_{{(g_{V)_{{V \in V_{\bullet}}}}}} \cdot c_{(G',D',V'_{\bullet})}(h^*((g_V)))\Bigg).
\end{multline}
For a tuple $(g_{V'}\geq -1)_{V'\in V'_\bullet}$ we define $h\colon (G',D',V'_\bullet)\to (G,D,V_\bullet)$ as the unique contraction with the property that $g_{V'}\geq 0$ if and only $V'=h^{-1}(V)$ for some $V\in V_\bullet$. That is, we contract each collection of vertices $V'$ such that $g_{V'}=-1$ (see \ref{PB gv}). We then define $\beta_{(g_{V'})_{V'\in V'_\bullet}}:=h^*(\beta_{(g_V)_{V\in V_\bullet}})$.
Formula~\eqref{two} can then be simplified to
\begin{equation}
    \sum_{(G',D') \in \catJ} \frac{f_{(G',D') *}}{\left| \Aut(G',D')\right|} \Bigg(\sum_{\substack{V' \textrm{ a \gc } \\ \textrm{in } \ext(G',D'); \\ (g_{V'} \geq -1)_{V' \in V'_{\bullet}}}} \beta_{{(g_{V'})_{{V' \in V'_{\bullet}}}}} \cdot  c_{(G',D',V'_{\bullet})}((g_{V'}))\Bigg).
\end{equation}
Now to obtain the final result, we rename $(G',D')$ and $V'_{\bullet}$ into $(G,D)$ and $V_{\bullet}$. Then we eliminate the indices $(g_V)$ by means of the equality \[g_V= -1+\sum (g_{e,V}+1)-\sum a_{e,V},\] and we replace the indices $(g_{e,V})$ with indices $(g_e)$ defined by
$g_e:=g_{e,k(e)}-\sum_{V\in S(e)}a_{e,V}$.
\end{proof}

As promised earlier, here we compare the $\psi$ classes on Jacobians with the classical ones on  moduli of curves.
\begin{remark} \label{relate-psi} Denote by $f \colon \mathcal{J}_{G,D} \to \overline{\mathcal{M}}_G$  the forgetful morphism. For every $e \in E(G)$,  we have:
\[
\Psi_{(G,D,e)}= f^*\Psi_{G,e}+ \Delta_{G,D,e},
\]
where $\Psi_{G,e}=-c_1(\mathbb{L}_e)$ is the first Chern class of the normal line bundle corresponding to the node $e$ on the stratum $\overline{\mathcal{M}}_G \to \Mmb gn$, and $\Delta_{G,D,e}$ is the divisor in $\mathcal{J}_{G,D}$  whose points represent sheaves that fail to be locally free at the edge $e$.
\end{remark}

\subsection{The case of disjoint blowups}
Our main results, Theorem~\ref{maintheorem} and Corollary~\ref{differenceinJbar} massively simplify in the case when the $m$ vine strata $\beta_1, \ldots, \beta_m$ are disjoint.  This is for example the case for all  hyperplanes on divisorial (or compact type) vine strata \eqref{walls1} (Proposition~\ref{prop:ctwalls}) --in this case $m$ equals $1$ and no blowup is required--, and for all hyperplanes of the form \eqref{walls2} with $S \neq [n]$ (Proposition~\ref{prop:good_disjoint}).

In each of these cases, the category $\catJtildeE$ only contains the terminal object and the resolved strata $(\beta_i,V^i_{\bullet})$ where $V^i_\bullet = \{V_i\}$ contains only the one vertex set $V_i=\{\leg_{\beta_i}(1)\}$ for all $i=1, \ldots,m$ (Proposition~\ref{prop:good_disjoint}). 

Recall the notation from the previous section.
We set $X_i^+$, $Y_i^+$ (respectively, $X_i$, $Y_i$) to be the two components over $\beta_i$ (respectively, over $(\beta_i,V^i_\bullet)$). We denote by $g_{Y_i}$ the genus of the fiber of $Y_i$ and by $d_{Y_i}$ the degree of the universal line bundle on $Y_i$. We also set $F^{Y_i}_+=F_{V_i}^+$ and $F^{Y_i}=F_{V_i}$.  Let $t_i$ be the number of nodes of a general curve in $\beta_i$, so $|\Aut(\beta_i)|=t_i!$.

Then we have:
\begin{corollary} \label{disjoint} When the hyperplane $H=H(\phi^+,\phi^-)$ is such that the vine strata $\beta_1, \ldots, \beta_m$ are pairwise disjoint, Formula~\eqref{mainformula} simplifies to
\begin{equation}
 \sum_{i=1}^m \sum_{s_i, j_i,k_i \geq 0}  \frac{1}{t_i!} \binom{g_{Y_i}-d_{Y_i}-j_i-1}{k_i+1} \cdot {f_{{E_i} \ast}} \big( c_{s_i}(F^{X_i}) \cdot c_{j_i}(F^{Y_i}) \cdot \Psi_i^{k_i}\big)
\end{equation}
\end{corollary}

\begin{proof} Follows immediately from Theorem~\ref{maintheorem}. Note the simplification of the minus sign in the definition of $b_{G,D,V}((j_V);(k_V))$ and the minus sign before $f_{\Gamma*}$ in Equation~\eqref{mainformula}.
\end{proof}

We can also recast the main result of \ref{differenceinJbar}.  by simply taking the pushforward along each $\mathbb{P}^{t_i-1}$ bundle $p_i \colon {E}_i \to {\beta}_i$. Note that the K-theory elements $F^{X_i}$ and $F^{Y_i}$ are pullbacks of corresponding elements $F^{X_i}_+$ and $F^{Y_i}_+$ on ${\beta}_i$.   For each $i$ and $1 \leq r_i \leq t_i$, let $\Psi_{i,r_i}$ be the first Chern class of the conormal bundle to the $r_i$-th gluing on the resolved stratum ${\beta}_i$.
\begin{corollary} \label{Cor: wc-goodwalls}
When the hyperplane $H=H(\phi^+,\phi^-)$ is such that the vine strata $\beta_1, \ldots, \beta_m$ are pairwise disjoint, the formula in Corollary~\ref{differenceinJbar} equals
\begin{equation} \label{goodwalls} \small
\sum_{\substack{s_i+j_i+\lambda_i=g-d-t_i \\ \text{for all } i=1, \ldots, m}} \frac{1}{t_i!} \binom{g_{Y_i}-d_{Y_i}-j_i-1}{g-d-j_i-s_i} \cdot f_{{\beta_i} \ast}   \big(  c_{s_i}(F^{X_i}_+) \cdot c_{j_i}(F^{Y_i}_+) \cdot h_{\lambda_i}(\Psi_{i,1}, \ldots, \Psi_{i,t_i})\big)
\end{equation}
where $h_{\lambda_i}$ is the complete homogeneous polynomial of degree $\lambda_i$ in its entries.
\end{corollary}
\begin{proof}
This follows from Corollary~\ref{differenceinJbar} or, more directly, by applying the push-pull formula to Corollary~\ref{disjoint} combined with the fact that the pushforward of $\Psi_i^{k_i}$ along ${E_i} \to {\beta}_i$ equals $h_{k_i-t_i+1}(\Psi_{i,1}, \ldots, \Psi_{i,t_i})$.
\end{proof}

We now analyze some even more special cases of this, already special, formula.

\begin{remark} (The ``compact type'' hyperplanes). A special case of Corollaries~\ref{disjoint} and~\ref{Cor: wc-goodwalls} occurs when $H(\phi^+, \phi^-)$ is a hyperplane of the form~\eqref{walls1}. In this case the generic locus where $\phi^+$ differs from $\phi^-$ is a compact type boundary divisor. In particular, $m$ equals $1$ and $\jphitilde \to \Jmb dgn (\phi^+)$ is the identity. The unique vine stratum $\beta=\beta_1$ consists of a boundary divisor $\Delta_{g-g_Y,S}\subset \Mmb{g}{n}$ decorated with a unique pair of $\phi^+$-stable bidegrees, say $(d-d_Y,d_Y)$. In this case the degree $g-d$ part of the formula in~\ref{disjoint} coincides  with the formula in~\ref{Cor: wc-goodwalls}, and they equal to
\begin{equation} \label{WC - ct}
\sum_{s+j+\lambda=g-d-1}  \binom{g_{Y}-d_{Y}-j-1}{g-d-j-s} \cdot f_{{\beta} \ast} \big( c_{s}(F^{X}_+) \cdot c_{j}(F^{Y}_+) \cdot \Psi^{\lambda}\big).
\end{equation}
\end{remark}

\subsection{Wall crossing in low codimension}

We now analyze the first few cases of our main result, ordered by codimension. 
\subsubsection{Codimension $1$} Let $d=g-1$. In this case the classes $\mathsf{w}_{g-1}(\phi)$ are divisors, also known under the name of \emph{theta divisors}. This case was the main result of \cite{kp2}.

Because  each $\mathsf{w}_{g-1}(\phi)$ is a divisor class and $\Jmb dgn (\phi^+)$ is nonsingular, the wall--crossing term equals zero across any hyperplane not of the form \eqref{walls1}. 

Assume that the hyperplane crossed is $H=H(g-g_Y, 1, S; d-d_Y+\frac{1}{2})$. Then Formula~\eqref{WC - ct} collapses and it gives \[\mathsf{w}_{g-1}(\phi^+)-\mathsf{Id}^*\mathsf{w}_{g-1}(\phi^-)= (g_Y-d_Y-1) \cdot [{\mathcal{J}}_\beta]\] for $\beta=(G(g-g_Y, 1, S), (d-d_Y, d_Y))$. This recovers \cite[Theorem~4.1]{kp2} after observing that the divisor ${\mathcal{J}}_\beta \subset \Jmb dgn(\phi^+)$ is the pullback of $\Delta_{g-g_Y,S} \subset \Mmb{g}{n}$.

\subsubsection{Codimension $2$} When $d=g-2$, the 
classes $\mathsf{w}_{g-2}(\phi)$ have codimension $2$. There are $2$ types of hyperplanes where the wall--crossing term is not zero.

If the hyperplane has the form \eqref{walls1},  the vine stratum $\beta$ is a boundary divisor. Assuming that $H$  and  $\beta$ are as in the previous paragraph, then Formula~\ref{WC - ct} reads
 \[f_{\beta \ast} \left( \binom{g_Y-d_Y-1}{g-d-1} c_1(F^X_+) + \binom{g_Y-d_Y-2}{g-d-1} c_1(F^Y_+) + \binom{g_Y-d_Y-1}{g-d} \Psi \right).\]

If the hyperplane is of type \eqref{walls2}, then the only cases when the formula is nontrivial is for $H=H(g-g_Y-1, 2, S, d-d_Y-1)$. While this hyperplane might witness a change in stability on more than $1$ vine stratum, the intersection of any  $2$ would occur in codimension $>2$ and hence it would not be relevant. We can read the wall--crossing term off Formula~\eqref{Cor: wc-goodwalls}:
\[
 \sum_{i=1}^m \binom{g_{Y_i}-d_{Y_i}-1}{g-d} [{\mathcal{J}}_{\beta_i}]
\]
Here $\beta_i$ for $i=1, \ldots, m$ is a vine graph of the form $(G(g-g_{Y_i}-1,2,S), (d-d_{Y_i}, d_{Y_i}))$  where the stability condition changes.  (If $S^c$ is not empty, then $m=1$).

\subsection{Pullbacks via Abel--Jacobi sections}

 Fix integers ${\bf d}=(k;d_1, \ldots, d_n)$, $\bf{f}=(f_{i,S})_{i,S}$, and let $\mathcal{L}=\mathcal{L}_{{\bf d}, {\bf f}}$ be the line bundle on the universal curve $\Cmb gn$ defined in Subsection~\ref{Sec: AJ}. Let then $\phi^+$ and $\phi^-$ be on opposite sides of a hyperplane $H$ (Definition~\ref{Def: opposite sides}), and such that $\mathcal{L}$ is $\phi^+$ stable. This defines an Abel--Jacobi section $\sigma=\sigma_{{\bf d},{\bf f}}\col \Mmb{g}{n}\to \Jmb{d}{g}{n}(\phi^+)$. We now compute the pullback of Formula~\ref{relation} via $\sigma$.
 
 For every $G \in G_{g,n}$, define the divisor $D=D_{{\bf d}, {\bf f}}$ on $G$ as the multidegree of $\mathcal{L}$ on any curve whose dual graph equals $G$. We then have a poset $\ext(G)=\ext(G,D)$, depending on $\phi^+, \phi^-$, defined in Section~\ref{sec: combo-of-wc}.

 The total space $\mathcal{C}_G \to \overline{\mathcal{M}}_G$ has one irreducilbe component $\mathcal{C}_v:=\mathcal{C}_{G,v}$ for each vertex $v$ of $G$. We redefine $\pi_v=\pi_{G, v}\col\mathcal{C}_v\to \overline{\mathcal{M}}_G$. Also, for each $V\subset V(G)$, we denote by $\pi_V\col \bigcup_{v\in V}\mathcal{C}_v\to \overline{\mathcal{M}}_G$ the induced map on the union. We write $X=X_{G}=\mathcal{C}_{\leg(1)}$ 
and $\Sigma = X\cap \mathcal{C}_{\{\leg(1)\}^c}$. We also write $Y_V=\mathcal{C}_{V^c}$ for every $V\subset V(G)$.
 
We then define the following K-theory elements in $\overline{\mathcal{M}}_G$
 \[ \small
 F^X_{{\bf d}, {\bf f}}:=- R^{\bullet} (\pi_X)_\ast \mathcal{L}(-\Sigma)_{|X}; \ \ F_V^{{\bf d}, {\bf f}}:=- R^{\bullet} (\pi_{V^c})_\ast (\mathcal{L}_{|Y_V}); \ \ 
H_V^{{\bf d}, {\bf f}}:=F_V-\sum_{\substack{V' \in V_{\bullet},\\ V' \gtrdot V}} F_{V'}.
 \]

The line bundle $\mathcal{L}$ also defines a, possibly rational, section $\sigma_-\colon \Mmb{g}{n} \dashrightarrow \Jmb dgn (\phi^-)$.
\begin{corollary} The difference
\[
\sigma^*(\mathsf{w}_d(\phi^+))-\sigma_{-}^*(\mathsf{w}_d(\phi^-))
\]
equals
\begin{multline} \label{differenceinMbar}
-\sum_{G \in G_{g,n}} \frac{1}{|\Aut(G)|} f_{G \ast} \\ \Bigg( \sum_{\substack{  V_{\bullet} \textrm{ a \gc in }  \ext(G), \\ s+\sum_{V} j_{V} + \sum g_{e} =g-d-|E(G)|}} \alpha(s,(j_V), (g_e)) \cdot  c_s(F^X_{{\bf d}, {\bf f}}) \cdot \prod_{V \in V_{\bullet}} c_{j_{V}} (H_{V}^{{\bf d}, {\bf f}}) \prod_{e} \Psi_{(G,e)}^{g_e}\Bigg) \end{multline}
where the coefficient $\alpha$ is defined in Equation~\eqref{eq:alpha}. \end{corollary} 
\begin{proof}
    Follows directly by pulling back \eqref{relation} via $\sigma$.
\end{proof}
When the hyperplane $H$ is such that the vine curve strata  that fail $\phi^-$-stability are disjoint, the latter can be simplified, as for Formula~\eqref{goodwalls}.

\begin{remark} \label{PB-AJ} The pullback of Formula~\ref{relation} via the Abel--Jacobi section $\sigma_{{\bf d},{\bf f}}$ can be explicitly computed via \cite[Theorem~1]{prvZ}. 

 For a full forest $V_\bullet$ in $\ext(G)$, recall the definition of the next element from \eqref{def: next}.   Defining $Z_V:=\mathcal{C}_{\nex(V)\setminus V}$ and set $\Sigma_V=Z_V\cap \bigcup_{V'\gtrdot V} \mathcal{C}_{V'^c}$ and $\Sigma_V'=Z_V\cap  \mathcal{C}_{V}$, we have 
\[
H_V^{{\bf d},{\bf f}} = -R^\bullet(\pi_V)_{\ast}((\mathcal{L}_{{\bf d},{\bf f}})_{|Z_{V}}(-\Sigma_V)),
\]
 and the line bundle $(\mathcal{L}_{{\bf d},{\bf f}})_{|Z_{V}}(-\Sigma_V) $ equals
 \[
\omega_{Z_V/\overline{\mathcal{M}}_G}^k\Bigg(k\Sigma'_V+(k-1)\Sigma_V+\sum_{\substack{\leg(j)\in \\ \nex(V)\setminus V}}d_jP_j+\sum_{\leg(S^c)\subseteq \nex(V)\setminus V} f_{i,S^c}C_{i, S^c}|_{Z_V}\Bigg).
 \]
We note that $Z_V$ is the disjoint union $\mathcal{C}_v$ for $v\in \nex(V)\setminus V$ (see Proposition~\ref{prop:properties_next}). Moreover, the line bundle above, restricted to each one of these components, is precisely the pullback via the projection $\overline{\mathcal{M}}_G\to \overline{\mathcal{M}}_{g(v),\text{val}(v)}$ of a line bundle as in \cite[Formula~0.1]{prvZ}.
\end{remark}

\begin{example} \label{ex-AJ} As an illustration of Remark~\ref{PB-AJ}, we write the simpler case where $H$ corresponds to changing the stability condition on a single vine stratum $\beta$ with $t$ nodes, components of genus $g_X$ and $g_Y$, with markings $S_\beta$ and $S_\beta^c$ respectively. Then $\overline{\mathcal{M}}_\beta= \overline{\mathcal{M}}_{g_X,|S_\beta|+t}\times \overline{\mathcal{M}}_{g_Y,|S_\beta^c|+t}$ and we denote by $p_X$ and $p_Y$ the projections. In this case we have that
\begin{align*}
F^X_{\mathcal{L}} =&p_X^*\Bigg( -R^\bullet(\pi_\ast^X)\Bigg(\omega_X^k((k-1)\Sigma+\sum_{j\in S_\beta}d_jP_j+\sum_{\substack{i\leq g_X \\ 1\in S\subseteq S_\beta}}f_{i,S_\beta\setminus S} \cdot C^X_{i,S_\beta\setminus S})\Bigg)\Bigg)\\
F^Y_{\mathcal{L}} =&p_Y^*\Bigg( -R^\bullet(\pi_\ast^Y)\Bigg(\omega_Y^k(k\Sigma+\sum_{j\in S_\beta^c}d_jP_j+\sum_{\substack{i\leq g_Y \\  S\subseteq S_\beta^c}}f_{i,S_\beta^c\setminus S} \cdot C^Y_{i,S_\beta^c\setminus S})\Bigg)\Bigg)\\
\end{align*}
 Formula~\eqref{differenceinMbar} for the difference $\sigma^*(\mathsf{w}_d(\phi^+))-\sigma_{-}^*(\mathsf{w}_d(\phi^-))$ in this case becomes
\[
\sum_{\substack{s+j+\lambda\\=g-d-t}} \binom{g_{Y}-d_{Y}-j-1}{g-d-j-s} \frac{f_{{\beta} *}}{t!} \Bigg(  c_{s}( F^X_{\mathcal{L}}) \cdot c_{j}(F^Y_{\mathcal{L}}) \cdot h_{\lambda}(\Psi_{1}, \ldots, \Psi_{t})\Bigg).
\]
The Chern classes above are computed in \cite[Theorem~1]{prvZ}.
\end{example}

\bibliographystyle{alpha}
\bibliography{biblio-curves}

\newcommand{\etalchar}[1]{$^{#1}$}
\begin{thebibliography}{KKMSD73}

\bibitem[ACG11]{gac}
Enrico Arbarello, Maurizio Cornalba, and Phillip~A. Griffiths.
\newblock {\em Geometry of algebraic curves. {V}olume {II}}, volume 268 of {\em
  Grundlehren der mathematischen Wissenschaften [Fundamental Principles of
  Mathematical Sciences]}.
\newblock Springer, Heidelberg, 2011.
\newblock With a contribution by Joseph Daniel Harris.

\bibitem[ACGH85]{acgh}
E.~Arbarello, M.~Cornalba, P.~A. Griffiths, and J.~Harris.
\newblock {\em Geometry of algebraic curves. {V}ol. {I}}, volume 267 of {\em
  Grundlehren der mathematischen Wissenschaften [Fundamental Principles of
  Mathematical Sciences]}.
\newblock Springer-Verlag, New York, 1985.

\bibitem[Alu10]{aluffi}
Paolo Aluffi.
\newblock Chern classes of blow-ups.
\newblock {\em Math. Proc. Cambridge Philos. Soc.}, 148(2):227--242, 2010.

\bibitem[AP20]{abreupacini}
Alex Abreu and Marco Pacini.
\newblock The universal tropical {J}acobian and the skeleton of the {E}steves'
  universal {J}acobian.
\newblock {\em Proc. Lond. Math. Soc. (3)}, 120(3):328--369, 2020.

\bibitem[AP21]{apresol}
Alex Abreu and Marco Pacini.
\newblock The resolution of the universal {A}bel map via tropical geometry and
  applications.
\newblock {\em Adv. Math.}, 378:Paper No. 107520, 62, 2021.

\bibitem[BHP{\etalchar{+}}]{bhpss}
Younghan Bae, David Holmes, Rahul Pandharipande, Johannes Schmitt, and Rosa
  Schwarz.
\newblock Pixton's formula and {A}bel--{J}acobi theory on the {P}icard stack.
\newblock \href{https://arxiv.org/abs/2004.08676}{arXiv:2004.08676}.

\bibitem[BL05]{borisov}
Lev Borisov and Anatoly Libgober.
\newblock Mc{K}ay correspondence for elliptic genera.
\newblock {\em Ann. of Math. (2)}, 161(3):1521--1569, 2005.

\bibitem[Cap94]{caporaso}
Lucia Caporaso.
\newblock A compactification of the universal {P}icard variety over the moduli
  space of stable curves.
\newblock {\em J. Amer. Math. Soc.}, 7(3):589--660, 1994.

\bibitem[CCUW20]{ccuw}
Renzo Cavalieri, Melody Chan, Martin Ulirsch, and Jonathan Wise.
\newblock A moduli stack of tropical curves.
\newblock {\em Forum Math. Sigma}, 8:Paper No. e23, 93, 2020.

\bibitem[CGH{\etalchar{+}}]{holmescomp}
Dawei Chen, Samuel Grushevsky, David Holmes, Martin M\"oller, and Johannes
  Schmitt.
\newblock A tale of two moduli spaces: logarithmic and multi-scale
  differentials.
\newblock \href{https://arxiv.org/abs/2212.04704}{arXiv:2212.04704}.

\bibitem[CJ18]{claderjanda}
Emily Clader and Felix Janda.
\newblock Pixton's double ramification cycle relations.
\newblock {\em Geom. Topol.}, 22(2):1069--1108, 2018.

\bibitem[CMKV15]{cmkv1}
Sebastian Casalaina-Martin, Jesse~Leo Kass, and Filippo Viviani.
\newblock The local structure of compactified {J}acobians.
\newblock {\em Proc. Lond. Math. Soc. (3)}, 110(2):510--542, 2015.

\bibitem[CMKV17]{CMKV17}
Sebastian Casalaina-Martin, Jesse~Leo Kass, and Filippo Viviani.
\newblock The singularities and birational geometry of the compactified
  universal {J}acobian.
\newblock {\em Algebr. Geom.}, 4(3):353--393, 2017.

\bibitem[Dud18]{dudin}
Bashar Dudin.
\newblock Compactified universal {J}acobian and the double ramification cycle.
\newblock {\em Int. Math. Res. Not. IMRN}, (8):2416--2446, 2018.

\bibitem[EP16]{estevespacini}
Eduardo Esteves and Marco Pacini.
\newblock Semistable modifications of families of curves and compactified
  {J}acobians.
\newblock {\em Ark. Mat.}, 54(1):55--83, 2016.

\bibitem[Est01]{esteves}
Eduardo Esteves.
\newblock Compactifying the relative {J}acobian over families of reduced
  curves.
\newblock {\em Trans. Amer. Math. Soc.}, 353(8):3045--3095, 2001.

\bibitem[Ful98]{fulton}
William Fulton.
\newblock {\em Intersection theory}, volume~2 of {\em Ergebnisse der Mathematik
  und ihrer Grenzgebiete. 3. Folge. A Series of Modern Surveys in Mathematics
  [Results in Mathematics and Related Areas. 3rd Series. A Series of Modern
  Surveys in Mathematics]}.
\newblock Springer-Verlag, Berlin, second edition, 1998.

\bibitem[GZ14]{gruzak}
Samuel Grushevsky and Dmitry Zakharov.
\newblock The zero section of the universal semiabelian variety and the double
  ramification cycle.
\newblock {\em Duke Math. J.}, 163(5):953--982, 2014.

\bibitem[HKP18]{HKP}
David Holmes, Jesse~Leo Kass, and Nicola Pagani.
\newblock Extending the double ramification cycle using {J}acobians.
\newblock {\em Eur. J. Math.}, 4(3):1087--1099, 2018.

\bibitem[HMP{\etalchar{+}}]{hmpps}
David Holmes, Samouil Molcho, Rahul Pandharipande, Aaron Pixton, and Johannes
  Schmitt.
\newblock Logarithmic double ramification cycles.
\newblock \href{https://arxiv.org/abs/2207.06778}{arXiv:2207.06778}.

\bibitem[Hol21]{holmes}
David Holmes.
\newblock Extending the double ramification cycle by resolving the
  {A}bel-{J}acobi map.
\newblock {\em J. Inst. Math. Jussieu}, 20(1):331--359, 2021.

\bibitem[KKMSD73]{kkms}
G.~Kempf, Finn~Faye Knudsen, D.~Mumford, and B.~Saint-Donat.
\newblock {\em Toroidal embeddings. {I}}.
\newblock Lecture Notes in Mathematics, Vol. 339. Springer-Verlag, Berlin-New
  York, 1973.

\bibitem[KP17]{kp2}
Jesse~Leo Kass and Nicola Pagani.
\newblock Extensions of the universal theta divisor.
\newblock {\em Adv. Math.}, 321:221--268, 2017.

\bibitem[KP19]{kp3}
Jesse~Leo Kass and Nicola Pagani.
\newblock The stability space of compactified universal {J}acobians.
\newblock {\em Trans. Amer. Math. Soc.}, 372(7):4851--4887, 2019.

\bibitem[Mel19]{melouniversal}
Margarida Melo.
\newblock Universal compactified {J}acobians.
\newblock {\em Port. Math.}, 76(2):101--122, 2019.

\bibitem[MMUV22]{mmuv}
Margarida Melo, Samouil Molcho, Martin Ulirsch, and Filippo Viviani.
\newblock Tropicalization of the universal {J}acobian.
\newblock {\em \'{E}pijournal G\'{e}om. Alg\'{e}brique}, 6:Art. 15, 51, 2022.

\bibitem[Mol22]{molchorub}
Sam Molcho.
\newblock Smooth compactifications of the {A}bel-{J}acobi section, 2022.

\bibitem[MPS23]{mps}
Samouil Molcho, Rahul Pandharipande, and Johannes Schmitt.
\newblock The hodge bundle, the universal 0-section, and the log {C}how ring of
  the moduli space of curves.
\newblock {\em Compositio Mathematica}, 159(2):306--354, 2023.

\bibitem[MW20]{marcuswise}
Steffen Marcus and Jonathan Wise.
\newblock Logarithmic compactification of the {A}bel-{J}acobi section.
\newblock {\em Proc. Lond. Math. Soc. (3)}, 121(5):1207--1250, 2020.

\bibitem[OS79]{oda79}
Tadao Oda and C.~S. Seshadri.
\newblock Compactifications of the generalized {J}acobian variety.
\newblock {\em Trans. Amer. Math. Soc.}, 253:1--90, 1979.

\bibitem[Pan18a]{pandhacalculus}
Rahul Pandharipande.
\newblock A calculus for the moduli space of curves.
\newblock In {\em Algebraic geometry: {S}alt {L}ake {C}ity 2015}, volume~97 of
  {\em Proc. Sympos. Pure Math.}, pages 459--487. Amer. Math. Soc., Providence,
  RI, 2018.

\bibitem[Pan18b]{pandacalculus}
Rahul Pandharipande.
\newblock A calculus for the moduli space of curves.
\newblock In {\em Algebraic geometry --- Salt Lake City 2015, Part I},
  number~97 in Proc. Sympos. Pure Math., pages 459--488. International Press,
  Boston, 2018.

\bibitem[Pix12]{pixton}
Aaron Pixton.
\newblock Conjectural relations in the tautological ring of
  $\overline{M}_{g,n}$, 2012.

\bibitem[PRvZ20]{prvZ}
Nicola Pagani, Andrea~T. Ricolfi, and Jason van Zelm.
\newblock Pullbacks of universal {B}rill-{N}oether classes via {A}bel-{J}acobi
  morphisms.
\newblock {\em Math. Nachr.}, 293(11):2187--2207, 2020.

\bibitem[Yin16]{yin}
Qizheng Yin.
\newblock Cycles on curves and {J}acobians: a tale of two tautological rings.
\newblock {\em Algebr. Geom.}, 3(2):179--210, 2016.

\end{thebibliography}

\end{document}